\documentclass[11pt]{article}
\usepackage[T1]{fontenc}
\usepackage[latin1]{inputenc}
\usepackage{amsmath,amssymb}
\usepackage{graphics}
\usepackage{epic}
\usepackage{epsfig}
\usepackage{subfigure}
\usepackage{psfrag}

\newtheorem{thm}{Theorem}[section]
\newtheorem{lem}[thm]{Lemma}
\newtheorem{cor}[thm]{Corollary}

\newtheorem{defi}[thm]{Definition}
\newtheorem{prop}[thm]{Proposition}
\newtheorem{rem}[thm]{Remark}

\newenvironment{proof}{\noindent {\bf Proof \phantom{9}}}
{\hfill $\square$ \vspace{0.25cm}}

\def\be{\begin{eqnarray}}
\def\ee{\end{eqnarray}}
\def\ben{\begin{eqnarray*}}
\def\een{\end{eqnarray*}}

\numberwithin{equation}{section}
\numberwithin{figure}{section}

\def\be{\begin{eqnarray}}
\def\ee{\end{eqnarray}}

\newcommand{\RR}{\mathbb{R}}

\newcommand{\PP}{\mathbb{P}}

\newtheorem{example}[thm]{Example}

\def\bit{\begin{itemize}}
\def\eit{\end{itemize}}
\reversemarginpar   
\def\bc{\begin{center}}
\def\ec{\end{center}}
\def\bthm{\begin{thm}}
\def\ethm{\end{thm}}
\def\bcor{\begin{cor}}
\def\ecor{\end{cor}}
\def\bprop{\begin{prop}}
\def\eprop{\end{prop}}
\def\blem{\begin{lem}}
\def\elem{\end{lem}}

\def\brem{\begin{rem}}
\def\erem{\end{rem}}

\def\bdes{\begin{description}}
\def\edes{\end{description}}

\def\beq{\begin{equation}}
\def\eeq{\end{equation}}
\def\benu{\begin{enumerate}}
\def\eenu{\end{enumerate}}
\def\beqar{\begin{eqnarray}}
\def\eeqar{\end{eqnarray}}
\def\beqarr{\begin{eqnarray*}}
\def\eeqarr{\end{eqnarray*}}

\def\RR{{\mathbb R}}  

\def\part{\partial}

\def\d#1dt{\frac{d#1}{dt}}    

\title{\bf Polymorphic evolution sequence and evolutionary
  branching}

\author{ Nicolas Champagnat\thanks{EPI TOSCA, INRIA Sophia Antipolis
    -- M\'editerran\'ee, 2004 route des
Lucioles, BP. 93, 06902 Sophia Antipolis Cedex, France; e-mail:
Nicolas.Champagnat@sophia.inria.fr.}, Sylvie
M\'el\'eard\thanks{CMAP, Ecole Polytechnique, CNRS, route de
Saclay, 91128 Palaiseau Cedex-France; e-mail:
sylvie.meleard@polytechnique.edu.}}

\date{\today}

\begin{document}

\maketitle

\begin{abstract}
  We are interested in the study of models describing the evolution of
  a polymorphic population with mutation and selection in the specific
  scales of the biological framework of adaptive dynamics. The
  population size is assumed to be large and the mutation rate
  small. We prove that under a good combination of these two scales,
  the population process is approximated in the long time scale of
  mutations by a Markov pure jump process describing the successive
  trait equilibria of the population. This process, which generalizes
  the so-called trait substitution sequence, is called polymorphic
  evolution sequence. Then we introduce a scaling of the size of
  mutations and we study the polymorphic evolution sequence in the
  limit of small mutations. From this study in the neighborhood of
  evolutionary singularities, we obtain a full mathematical
  justification of a heuristic criterion for the phenomenon of
  evolutionary branching. To this end we finely analyze the asymptotic
  behavior of 3-dimensional competitive Lotka-Volterra systems.
\end{abstract}

\bigskip
\emph{MSC 2000 subject classification:} 92D25, 60J80, 37N25, 92D15, 60J75
\bigskip

\emph{Key-words:} Mutation-selection individual-based model,
fitness of invasion,  adaptive dynamics,  polymorphic evolution
sequence, competitive Lotka-Volterra system, evolutionary
branching.

\section{Introduction}
\label{sec:intro}

We consider an asexual population in which each individual's
ability to survive and reproduce is characterized by a
quantitative trait, such as the size, the age at maturity, or the
rate of food intake. Evolution, acting on the trait distribution
of the population, is the consequence of three basic mechanisms:
\emph{heredity}, which transmits traits to new offsprings,
\emph{mutation}, driving a variation in the trait values in the
population, and \emph{selection} between these different trait
values, which is due to the competition between individuals for
limited resources or area. Adaptive dynamics models aim at
studying the interplay between these different
mechanisms~\cite{hofbauer-sigmund-90,marrow-law-al-92,metz-nisbet-al-92}.
Our approach is based on a microscopic individual-based model that
details the ecological dynamics of each individual.  From the
simulated dynamics of this process initially issued from a
monomorphic population, we observe that it is essentially
single-modal centered around a trait that evolves continuously,
until some time where the population divides into two separate
sub-populations that are still in interaction but are centered
around distinct traits at a distance increasing with time. This
phenomenon, called Evolutionary Branching, is thought to be a
possible explanation of phenotypic separation without geographic
separation~\cite{dieckmann-doebeli-99}. (One speaks about
sympatric speciation though the population is asexual). Our aim in
this paper is to understand the dynamics of the process in long
time scales and to highlight the evolutionary branching
phenomenon. In particular, we want to prove the conjecture stated
by Metz et al. \cite{metz-geritz-al-96} and giving conditions on
the parameters of the model allowing one to predict whether
evolutionary branching will occur or not.

To this aim, we follow the basic description of adaptive dynamics
based on the biologically motivated assumptions of rare mutations
and large population. Under these assumptions, we prove that the
microscopic process describing the ecological dynamics  can be
approximated by a Markov pure jump process on the set of point
measures on the trait space. The transitions of this process are
given by the long time behavior of competitive Lotka-Volterra
systems. They describe the succession of mutant invasions followed
by a fast competition phase between the mutant population and the
resident one. In the mutation time scale, and for large
populations, the successful traits in the competition are given by
the nontrivial equilibria of Lotka-Volterra systems which model
the dynamics of the sizes of each sub-population corresponding to
each resident or mutant trait. We thus generalize the situation
introduced in \cite{metz-geritz-al-96} and mathematically
developed in \cite{champagnat-06}, where the parameters of the
model prevent the coexistence of two traits. In that case, the
microscopic model converges to a monomorphic (one trait support)
pure jump process, called Trait Substitution Sequence (TSS). This
limit involves a timescale separation between the mutations and
the population dynamics driving the competition between traits.

In this article, we relax the assumption of non-coexistence and
obtain a \emph{polymorphic evolution sequence} (PES), allowing
coexistence of several traits in the population, from the same
microscopic model described in Section~\ref{sec:model}. In
Section~\ref{sec:heuristics}, we introduce the competitive
symmetric Lotka-Volterra systems describing the competition
between traits.  We prove in Section~\ref{sec:cvge} that the PES
takes the form of a Markov jump process on the set of measures on
the trait space ${\cal X}$ that are finite sums of Dirac masses
with positive weights, and we characterize the transitions of this
process in terms of the long time behaviour of competitive
Lotka-Volterra systems. In Section~\ref{sec:CP}, we explain why
the assumptions ensuring the convergence to the PES are satisfied
as long as no more than two traits coexist. In this case, the
dynamics of the PES can be explicitely characterized. Next
(Section~\ref{sec:br}), we study the transition from a monomorphic
population to a stable dimorphic population, and give a full
mathematical justification of the criterion for evolutionary
branching proposed in \cite{metz-geritz-al-96}, under the
assumption of small mutation effects. To this end, we first show
in Sections~\ref{sec:caneq} and~\ref{sec:ES} that, away from
evolutionary singularities, the support of the PES stays monorphic
and converges to an ODE known as the ``canonical
equation''~\cite{dieckmann-law-96}. Finally, in
Section~\ref{sec:thm-br}, we characterize the situations where
evolutionary branching occurs by specializing to our situation the
results of Zeeman~\cite{zeeman-93} on the asymptotic behavior of
3-dimensional competitive Lotka-Volterra systems.

Let us insist on the importance of the limits. Here we are concerned
by the combination of the limits of large populations and rare
mutations, followed by a limit of small mutations. An alternative
approach would be first to study the limit of large population alone,
giving in the limit an integro-differential partial differential
equation for the density of traits~\cite{champagnat-ferriere-al-06};
and next to study a limit of small mutations on this equation with a
proper time scaling that would lead to some dynamics on the set of
finite sums of Dirac masses on the trait space. The second part of
this program has already been partly studied
in~\cite{diekmann-jabin-al-05} in a specific model, but is related to
difficult problems on Hamilton-Jacobi equations with
constraints~\cite{barles-perthame-06}. In this case, evolutionary
branching is numerically observed, but not yet fully
justified. Another approach would be to combine the three limits we
consider directly at the level of the microscopic model, allowing one
to study the evolutionary process on several time
scales~\cite{bovier-champagnat-08}. This requires a finer analysis of
the invasion and competition phases after the appearance of a new
mutant. Note that all these approaches are based on the same idea of
separation between the time scales of mutation and competition.

\section{Models and Polymorphic Evolution Sequence (PES)}
\label{sec:model}

Let us introduce here the main models on which our approach is based.

\subsection{The individual-based model}
\label{sec:micro}

The microscopic model we use is an individual-based model with
density-dependence, which has been already studied in ecological
or evolutionary contexts by many
authors~\cite{fournier-meleard-04,champagnat-ferriere-al-06}.

The trait space ${\cal X}$ is assumed to be a compact subset of
$\mathbb{R}^l$, $l\geq 1$.  For any $x,y\in{\cal X}$, we introduce the
following biological parameters
\begin{description}
\item[$\lambda(x)\in\mathbb{R}_+$] is the rate of birth from an individual
  holding trait $x$.
\item[$\mu(x)\in\mathbb{R}_+$] is the rate of ``natural'' death for an
  individual holding trait $x$.
\item[$r(x):=\lambda(x)-\mu(x)$] is the ``natural'' growth rate of
trait $x$.
\item[$K\in\mathbb{N}$] is a parameter scaling the
population size and
  the resources.
   \item[${\alpha(x,y)\over K} \in\mathbb{R}_+$] is the competition
kernel
  representing the pressure felt by an individual holding trait $x$
  from an individual holding trait $y$. It is not assumed to be a
  symmetric function.
\item[$ u_K\, p(x)$] with $ u_K, p(x)\in(0,1]$, is the
  probability that a
  mutation occurs in a birth from an individual with trait $x$.
   Small $u_K$ means rare
  mutations.
\item[$m(x,h)dh$] is the law of $h=y-x$, where the mutant trait
$y$ is
  born from an individual with trait $x$. Its support is a subset of
  ${\cal X}-x=\{ y-x:y\in{\cal X}\}$.
\end{description}

\noindent We consider, at any time $t\geq 0$, a finite number
$N_t$ of individuals, each of them holding a trait value in ${\cal
X}$. Let us denote by $x_1,\ldots,x_{N_t}$ the trait values of
these individuals. The state of the population at time $t\geq 0$,
rescaled by $K$, is described by the finite point measure on
${\cal X}$
\begin{equation}
  \label{eq:nu_t}
  \nu^K_t=\frac{1}{K}\sum_{i=1}^{N_t}\delta_{x_i},
\end{equation}
where $\delta_x$ is the Dirac measure at $x$. Let $\:
\langle\nu,f\rangle$ denote the integral of the measurable
function $f$ with respect to the measure $\nu$ and $\mbox{Supp}(\nu)$
denote its support.

\noindent Then $\: \langle\nu^K_t,{\bf 1}\rangle=\frac{N_t}{K}$ and
for any $x\in {\cal X}$, the positive number $\langle\nu^K_t,{\bf
1}_{\{x\}}\rangle$ is called {\bf the density} at time $t$ of
trait $x$.

\noindent  Let ${\cal M}_F$ denote the set of finite nonnegative
measures on ${\cal X}$, equipped with the weak topology, and
define
\begin{equation*}
  {\cal M}^K=\left\{\frac{1}{K}\sum_{i=1}^n\delta_{x_i}:n\geq 0,\
  x_1,\ldots,x_n\in{\cal X}\right\}.
\end{equation*}

\noindent An individual holding trait $x$ in the population
$\nu^K_t$ gives birth to another individual with rate $\lambda(x)$
and dies with rate
\begin{equation*}
  \mu(x)+\int\alpha(x,y)\nu^K_t(dy)=\mu(x)
  +\frac{1}{K}\sum_{i=1}^{N_t}\alpha(x,x_i).
\end{equation*}
The parameter $K$ scales the strength of competition, thus
allowing the coexistence of more individuals in the population. A
newborn holds the same trait value as its progenitor with
probability $1-u_K p(x)$, and with probability $u_K p(x)$, the
newborn is a mutant whose trait value $y$ is chosen according to
$y=x+h$, where $h$ is a random variable with law $m(x,h)dh$. In
other words, the process $(\nu^K_t,t\geq 0)$ is a ${\cal
  M}^K$-valued Markov process with infinitesimal generator defined for
any bounded measurable functions $\phi$ from ${\cal M}^K$ to
$\mathbb{R}$ by
\begin{align}
  L^K\phi(\nu) & =\int_{\cal
    X}\left(\phi\left(\nu+\frac{\delta_x}{K}\right)
    -\phi(\nu)\right)(1-u_K p(x))\lambda(x)K\nu(dx) \notag \\
  & +\int_{\cal X}\int_{\mathbb{R}^l}
  \left(\phi\left(\nu+\frac{\delta_{x+h}}{K}\right)-\phi(\nu)\right)
  u_K p(x)\lambda(x)m(x,h)dhK\nu(dx) \notag \\
  & +\int_{\cal X}\left(\phi\left(\nu-\frac{\delta_x}{K}
    \right)-\phi(\nu)\right)\left(\mu(x)+\int_{\cal
      X}\alpha(x,y)\nu(dy)\right)K\nu(dx).
  \label{eq:generator-renormalized-IPS}
\end{align}
For $\nu\in {\cal M}^K$, the integrals with respect to $K\nu(dx)$
in~(\ref{eq:generator-renormalized-IPS}) correspond to sums over all
individuals in the population.  The first term (linear) describes the
births without mutation, the second term (linear) describes the births
with mutation, and the third term (non-linear) describes the deaths by
oldness or competition. The density-dependent non-linearity of the
third term models the competition in the population, and hence drives
the selection process.
\medskip

\noindent Let us denote by~(A) the following three assumptions
\begin{description}
\item[\textmd{(A1)}] $\lambda$, $\mu$ and $\alpha$ are measurable functions,
  and there exist $\bar{\lambda},\bar{\mu},\bar{\alpha}<+\infty$ such that
  \begin{equation*}
    \lambda(\cdot)\leq\bar{\lambda},\quad \mu(\cdot)\leq\bar{\mu}
    \quad\mbox{and}\quad\alpha(\cdot,\cdot)\leq\bar{\alpha}.
  \end{equation*}
  \item[\textmd{(A2)}] $r(x)=\lambda(x)-\mu(x)>0$ for any $x\in{\cal X}$, and there exists
  $\underline{\alpha}>0$ such that
$ \ \underline{\alpha}\leq\alpha(\cdot,\cdot)$.

\item[\textmd{(A3)}] There exists a function
$\bar{m}:\mathbb{R}^l\rightarrow\mathbb{R}_+$ such
  that $m(x,h)\leq \bar{m}(h)$ for any $x\in{\cal X}$ and
  $h\in\mathbb{R}^l$, and $\int \bar{m}(h)dh<\infty$.

\end{description}
\noindent For fixed $K$, under~(A1) and~(A3) and assuming that
$\mathbf{E}(\langle\nu^K_0,\mathbf{1}\rangle)<\infty$, the existence
and uniqueness in law of a process on $\mathbb{D}(\mathbb{R}_+, {\cal
  M}^K)$ with infinitesimal generator $L^K$ has been proved
in~\cite{fournier-meleard-04}.  Assumption~(A2) prevents the
population to explode and to go extinct too fast.

\subsection{An example}
\label{sec:ex}

The birth-death-competition-mutation process described above has
been heuristically studied in various ecological or evolutionary
contexts. Let us illustrate the phenomenon of evolutionary
branching we are interested in with a simple example,
adapted from a classical model (Roughgarden~\cite{roughgarden-79},
Dieckmann and Doebeli~\cite{dieckmann-doebeli-99}). In this model,
there is a single optimal trait value for the birth rate and a
symmetric competition kernel. The parameters are the following:
\begin{equation}
  \label{eq:param}
  \begin{gathered}
    {\cal X}=[-2,2]; \quad \mu(x)\equiv 0; \quad p(x)\equiv p, \\
    \lambda(x)=\exp(-x^2/2\sigma_b^2), \\
    \alpha(x,y)=\tilde{\alpha}(x-y)=\exp(-(x-y)^2/2\sigma^2_\alpha).
  \end{gathered}
\end{equation}
and $m(x,h)dh$ is the law of a ${\cal N}(0,\sigma^2)$ r.v.\ $Y$
(centered Gaussian with variance $\sigma^2$) conditioned on
$x+Y\in{\cal X}$.

\noindent The growth rate $\lambda(x)$ is maximal at $x=0$ and there
is local competition between traits, in the sense that $\alpha(x,y)$
is maximal for $x=y$ and is close to 0 when $|x-y|$ is large. If the
competition kernel was flat ($\alpha\equiv 1$), evolution would favor
mutant traits with maximal growth rate.  However, if competition is
local, numerical simulations of the microscopic model give different
patterns, as shown in Fig.~\ref{fig:ex}. The pattern of
Fig.~\ref{fig:ex}(b), where the population, initially composed of
traits concentrated around a single trait value, is driven by the
evolutionary forces to states where the population is composed of two
(or more) groups, concentrated around different trait values. This
phenomenon is called evolutionary branching and has been observed in
many biological models (see
e.g.~\cite{metz-geritz-al-96,kisdi-99,geritz-vandermeijden-al-99}).
It is believed to be a possible mechanism of traits separation that
could lead to speciation~\cite{dieckmann-doebeli-99}.

\medskip \noindent
In this particular model, the possibility of evolutionary
branching seems to be governed by the values of $\sigma_b$ and
$\sigma_\alpha$, which represent respectively the width of the
trait region with high growth rate and the interaction range. In
Fig.~\ref{fig:ex}(a), $\sigma_\alpha>\sigma_b$ and there is no
evolutionary branching, whereas in Fig.~\ref{fig:ex}(b),
$\sigma_\alpha<\sigma_b$ and evolutionary branching occurs. We
observe in both simulations that, in a first phase, the population
trait support is concentrated around a mean trait value that
converges to 0. In a second phase, new mutants feel two different
selective pressures: high growth rate (traits close to 0) and
competition (traits far from the rest of the population). If
$\sigma_\alpha$ is small, the selection pressure is weaker for
traits away from 0 and allows the apparition of new branches.
The goal of this article is to justify mathematically this
heuristics.

\begin{figure}[h]
  \centering
  \mbox{\subfigure[$p=0.1$, $K=1000$, $\sigma=0.01$, $\sigma_b=0.9$, $\sigma_\alpha=1.0$.]%
{\epsfig{
figure=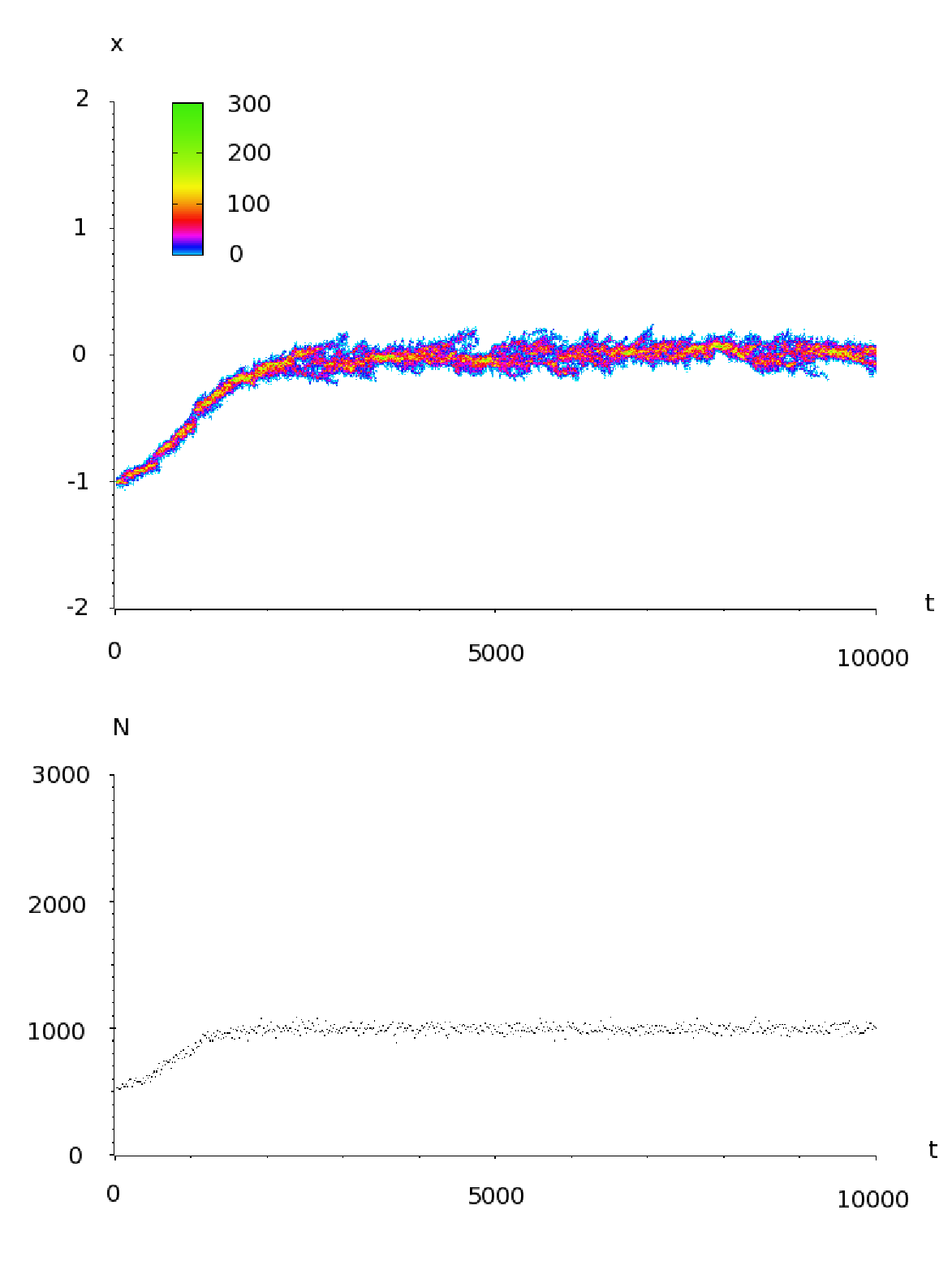, width=.47\textwidth}}\quad
    \subfigure[$p=0.1$, $K=1000$, $\sigma=0.01$, $\sigma_b=0.9$, $\sigma_\alpha=0.7$.]%
{\epsfig{
figure=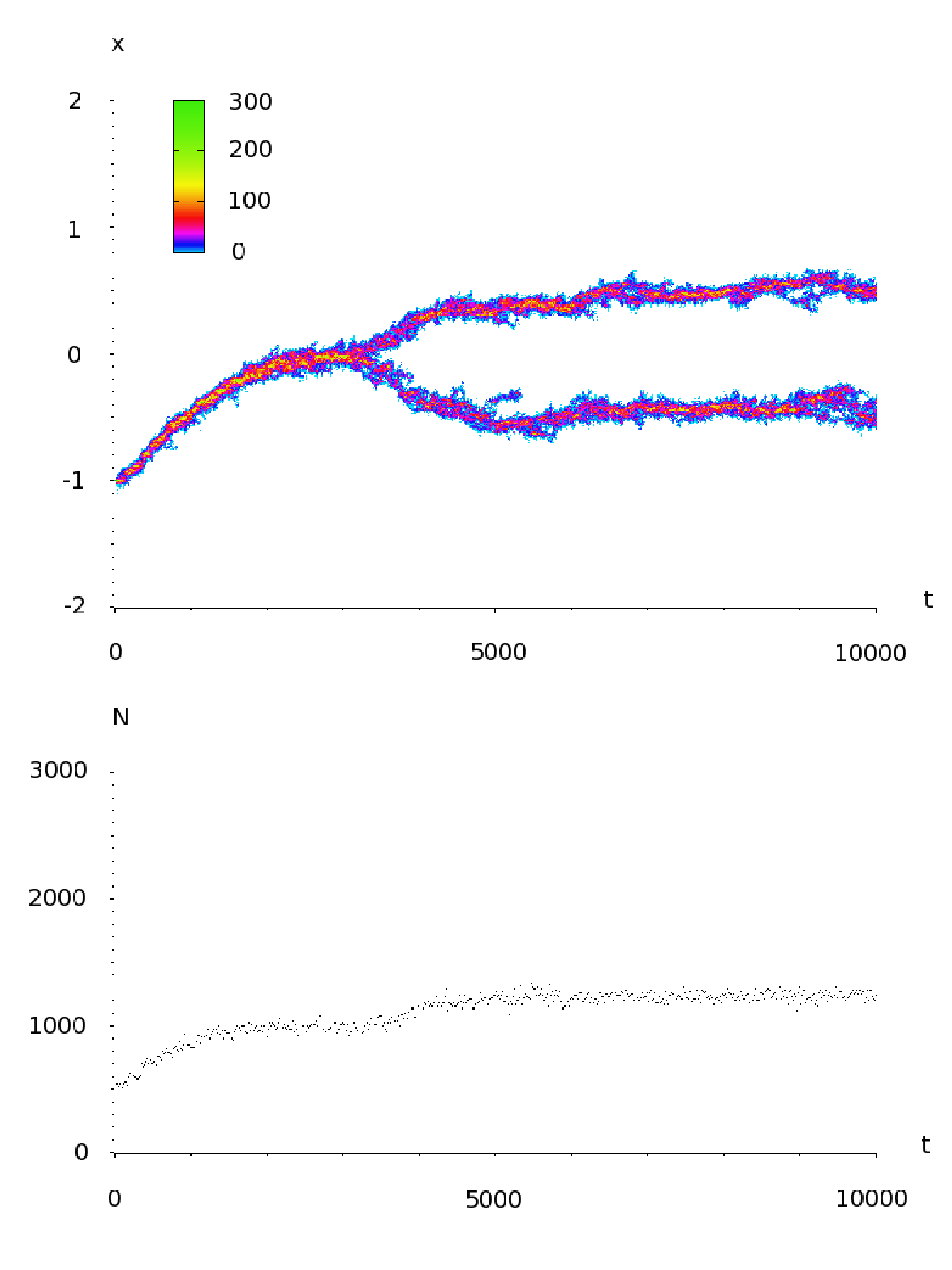, width=.47\textwidth}}} \\
\caption{{\small Numerical simulations of the trait distribution
    (upper panels) and population size (lower panels) of the
    microscopic model with parameters~(\ref{eq:param}).  The initial
    population is composed of $K$ individuals all with trait $-1.0$.}}
  \label{fig:ex}
\end{figure}

\subsection{On scales}
\label{sec:heuristics}

In order to analyze the phenomenon of evolutionary branching, we are
going to consider three biological asymptotics in the individual-based
model: large population ($K\rightarrow+\infty$), rare mutations
($u_K\rightarrow 0$) and small mutation amplitude.  The combination of
the two first scales will allow us to describe the polymorphic
evolution sequence, we will focus on. This limit amounts to
approximate the simulated dynamics of Fig.~\ref{fig:ex}(a) and~(b) of
the previous section by the one of Fig.~\ref{fig:ex2}(a) and~(b),
respectively.
\begin{figure}[h]
  \centering
  \mbox{\subfigure[$\mu=0.0001$, $K=1000$, $\sigma=0.08$, $\sigma_b=0.9$, $\sigma_\alpha=1.0$.]%
    {\epsfig{
        figure=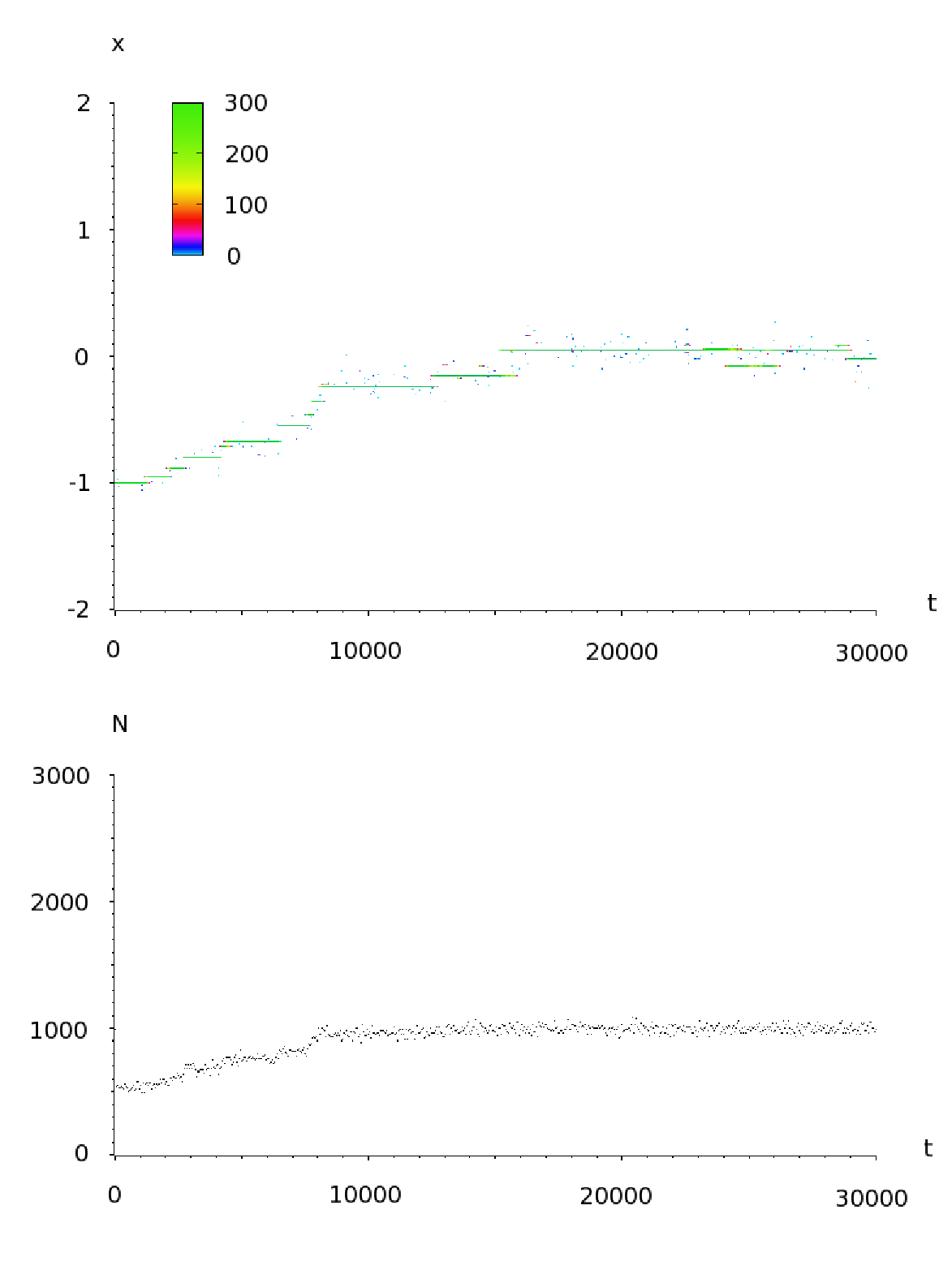, width=.47\textwidth}}\quad
    \subfigure[$\mu=0.0001$, $K=1000$, $\sigma=0.08$, $\sigma_b=0.9$, $\sigma_\alpha=0.7$.]%
    {\epsfig{
        figure=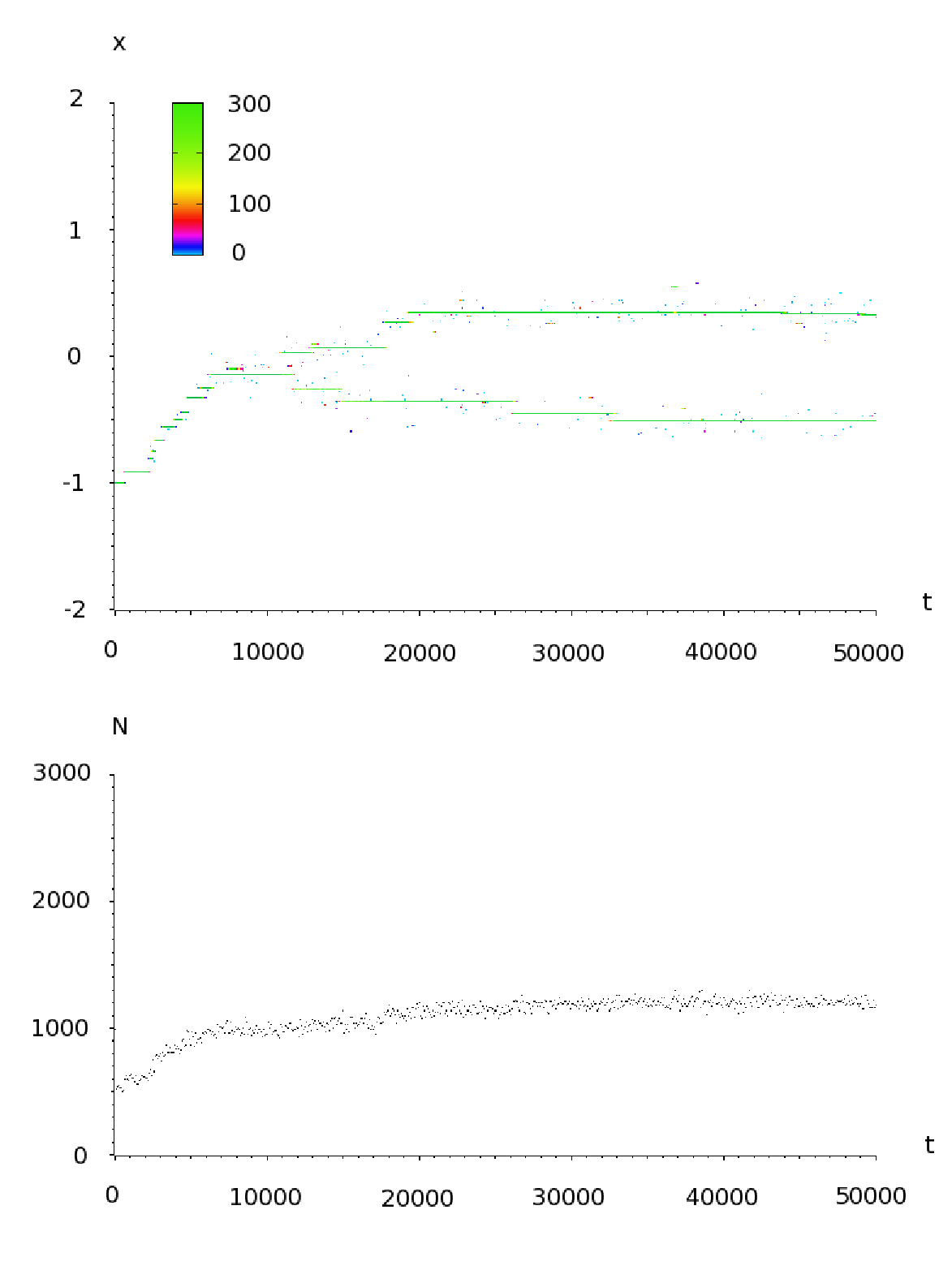, width=.47\textwidth}}} \\
  \caption{{\small Numerical simulations of the trait distribution (upper
    panels) and population size (lower
    panels) of the microscopic model with parameters~(\ref{eq:param}).
    The initial population is composed of $K$ individuals all with trait
    $-1.0$. The value of $\sigma$ is higher than in Fig.~\ref{fig:ex} so
    that the jumps are visible.}}
  \label{fig:ex2}
\end{figure}
These scales and the biological heuristics of this approach were
introduced in~\cite{metz-geritz-al-96}. The main interest of the
assumption of rare mutations is the separation between ecological
and evolutionary time scales: the selection process has sufficient
time between two mutations to eliminate disadvantaged traits. Then
evolution proceeds by a succession of phases of mutant invasion
and phases of competition between traits.   We will choose
parameters such that the ecological and evolutionary time scales
are separated, leading to an evolutionary dynamics where
competition phases are infinitesimal on the mutation time scale.
In addition, the large population assumption allows one to assume
a deterministic population dynamics between mutations, so that the
outcome of the competition can be predicted. More formally,
between two mutations, a finite number of traits are present,
namely $x_1,\ldots,x_d$, and the population dynamics can be
reduced to a Markov process in $\mathbb{N}^d$.

\noindent Assume that, for all $i\in\{1,\ldots,d\}$,
$\frac{1}{K}\langle \nu^K_0,\mathbf{1}_{\{x_i\}}\rangle$ has
bounded second-order moments and converge in distribution to
$n_i(0)\in\mathbb{R}_+$. Then, as proved
in~\cite[Thm.4.2]{champagnat-ferriere-al-08}, when
$K\rightarrow+\infty$, the process
$\frac{1}{K}(\langle\nu^K_t,\textbf{1}_{\{x_1\}}\rangle
,\ldots,\langle\nu^K_t,\textbf{1}_{\{x_d\}}\rangle)$ converges in
distribution for the Skorohod topology to the solution of the
$d$-dimensional competitive Lotka-Volterra system
$LV(d,\mathbf{x})$ with initial condition
$(n_1(0),\ldots,n_d(0))$.

\begin{defi}
  \label{not:LV}
  For any $\mathbf{x}=(x_1,\ldots,x_d)\in{\cal X}^d$, we denote by
  $LV(d,\mathbf{x})$ the competitive Lotka-Volterra system defined by
  \begin{equation}
    \label{eq:LVcompl}
    \dot{\mathbf{n}}(t)=F^{\mathbf{x}}(\mathbf{n}(t)),\quad
    1\leq i\leq d, \quad t\geq 0,
  \end{equation}
  where $\mathbf{n}(t)=(n_1(t),\ldots,n_d(t))$,
  \begin{equation}
    \label{eq:def-F_i}
    F^{\mathbf{x}}_i(\mathbf{n}):=n_iG^{\mathbf{x}}_i(\mathbf{n})\quad
    \hbox{ where }\quad
    G_i^{\mathbf{x}}(\mathbf{n}):=r(x_i)-\sum_{j=1}^d\alpha(x_i,x_j)n_j.
  \end{equation}
\end{defi}
\noindent
 The equilibria of $LV(d,\mathbf{x})$ are given by the
intersection of hyperplanes $(P_i)_{1\leq i\leq d}$, where $P_i$
has equation either $n_i=0$ or $G_i^{\mathbf{x}}(\mathbf{n})=0$.
We need to introduce the following notion of coexisting traits.

\begin{defi}
  \label{def:coex}
  For any $d\geq 0$, we say that $x_1,\ldots,x_d$
    \textbf{coexist} if $LV(d,\mathbf{x})$ admits a unique non-trivial
    equilibrium $\bar{\mathbf{n}}(\mathbf{x})\in(\RR_+^*)^d$ locally
    strongly stable, in the sense that the eigenvalues of the Jacobian
    matrix of $LV(d,\mathbf{x})$ at $\bar{\mathbf{n}}(\mathbf{x})$
    have all (strictly) negative real part. In particular, for all
    $i\in\{1,\ldots,d\}$,
    \begin{equation}
      \label{eq:F=0}
      G_i^{\mathbf{x}}(\bar{\mathbf{n}}(\mathbf{x}))=0 \quad \hbox{ and }
      \quad DF^{\mathbf{x}}(\bar{\mathbf{n}}(\mathbf{x}))=((-\alpha(x_i,x_j)\bar{\mathbf{n}}_i(\mathbf{x})))_{1\leq i,j\leq d}.
    \end{equation}

\end{defi}
\bigskip
\noindent In the monomorphic case ($d=1$) and when $r(x)>0$, the
competitive Lotka-Volterra system $LV(1,x)$ takes the form of the
so-called logistic equation
\begin{equation}
  \label{eq:LV-d1}
  \dot{n}_x=n_x(r(x)-\alpha(x,x)n_x).
\end{equation}
The unique stable equilibrium of this equation is
$\bar{n}(x)=r(x)/\alpha(x,x)$.

 \noindent Similarly, in the
dimorphic case where $d=2$, the system $LV(2,(x,y))$ takes the
form
\begin{equation}
  \label{eq:LV-d2}
  \begin{cases}
    \dot{n}_x=n_x(r(x)-\alpha(x,x)n_x-\alpha(x,y)n_y) \\
    \dot{n}_y=n_y(r(y)-\alpha(y,x)n_x-\alpha(y,y)n_y).
  \end{cases}
\end{equation}
Under Assumption (A2), the equilibria of~(\ref{eq:LV-d2}) are
$(0,0)$, $(\bar{n}(x),0)$, $(0,\bar{n}(y))$ and possibly a
non-trivial equilibrium in $(\mathbb{R}_+^*)^{2}$. It is known
(see e.g.~\cite{istas-05}) that the non trivial equilibrium exists
and is locally strongly stable, (traits $x$ and $y$ coexist), if
and only if $f(x;y)>0$ and $f(y;x)>0$, where
\begin{equation}
  \label{eq:fitness-d=1}
  f(y;x)=r(y)-\alpha(y,x)\bar{n}(x).
\end{equation}

\subsection{Convergence to the Polymorphic Evolution Sequence (PES)}
\label{sec:cvge}

Our goal here is to examine the asymptotic behavior of the
microscopic process when the population size grows to infinity as well
as the mutation rate converges to 0, in a long time scale. Before
stating our convergence result, we first give an idea of the
argument used, extending the biological heuristics
of~\cite{metz-geritz-al-96} and the special case of the trait
substitution sequence (TSS) developed in~\cite{champagnat-06} (see
also Section~\ref{sec:TSS}).

\subsubsection{Idea of the proof}
\label{sec:idea-pf}

\noindent Let us roughly describe the successive steps of
mutation, invasion and competition. The two steps of the invasion
of a mutant in a given population are firstly the stabilization of
the resident population before the mutation and secondly the
invasion of the mutant population after the mutation.

\noindent Fix $\eta>0$. In the first step, assuming that $d$
traits $x_1,\ldots,x_d$ that coexist are present, we prove that
the population densities
$(\langle\nu^K_t,\mathbf{1}_{\{x_1\}}\rangle,\ldots,\langle\nu^K_t,\mathbf{1}_{\{x_d\}}\rangle)$
belong to the $\eta$-neighborhood of
$\bar{\mathbf{n}}(\mathbf{x})$ with high probability for large $K$
until the next mutant $y$ appears. To this aim, we use large
deviation results on the problem of exit from a
domain~\cite{freidlin-wentzell-84} to prove that the time needed
for the population densities to leave the $\eta$-neighborhood of
$\bar{\mathbf{n}}(\mathbf{x})$ is bigger than $\exp(VK)$ for some
$V>0$ with high probability. Therefore, until this exit time, the
rate of mutation from trait $x_i$ in the population is close to
$u_Kp(x_i)\lambda(x_i)K\bar{n}_i(\mathbf{x})$ and thus, the first
mutation appears before this exit time if one assumes that
$$
\frac{1}{K u_K}\ll e^{VK}.
$$
In particular, the mutation rate from trait $x_i$ on the time scale
$t/K u_K$ is close to
$$
p(x_i)\lambda(x_i)\bar{n}_i(\mathbf{x}).
$$

\noindent In the second step, we divide the invasion of a given
mutant trait $y$ into 3 phases shown in Fig.~\ref{fig:inv-fix}, in
a similar way as done classically by population geneticists
dealing with selective
sweeps~\cite{kaplan-hudson-al-89}.

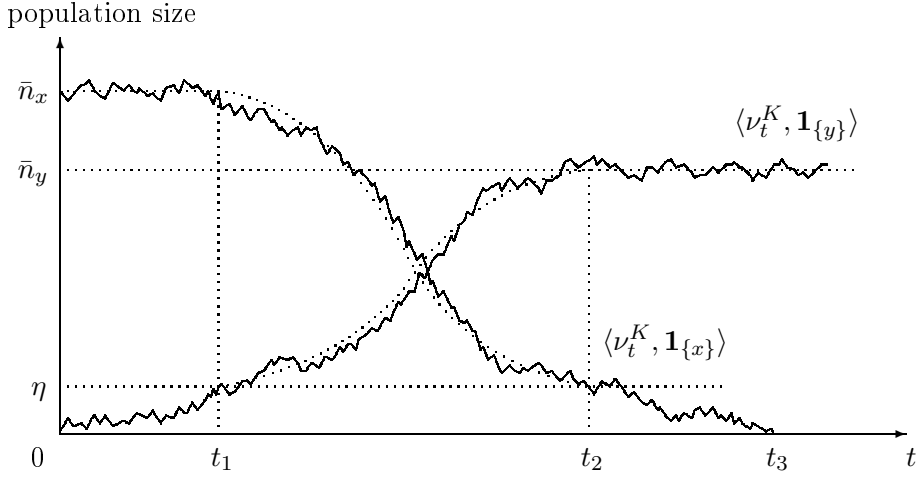
\begin{figure}[ht]
  \begin{center}
    \begin{picture}(350,180)(-20,-10)
      \put(0,0){\vector(1,0){320}} \put(0,0){\vector(0,1){150}}
      \put(-11,-12){0} \put(-11,15){$\eta$}
      \dottedline{3}(0,18)(250,18)
      \put(-16,97){$\bar{n}_y$} \put(-16,127){$\bar{n}_x$}
      \dottedline{3}(0,100)(300,100) \dottedline{3}(0,130)(60,130)
      \put(-20,156){population size} \put(57,-12){$t_1$}
      \put(197,-12){$t_2$} \put(267,-12){$t_3$} \put(320,-12){$t$}
      \dottedline{3}(60,0)(60,130) \dottedline{3}(200,0)(200,100)
      \qbezier[27](60,130)(100,125)(130,70)
      \qbezier[25](130,70)(150,28)(200,18)
      \qbezier[25](60,18)(105,26)(130,55)
      \qbezier[25](130,55)(155,95)(200,100)
      \dottedline{0.5}(0,1)(3,6)(5,3)(8,4)(10,2)(12,6)(13,7)(15,4)(17,7)%
(20,3)(24,5)(26,4)(29,8)(33,6)(35,9)(38,4)(40,7)(43,6)(45,10)(47,8)(48,8)%
(50,12)(52,9)(55,13)(56,15)(57,14)(58,17)(59,15)(60,18)(61,19)(64,15)%
(66,16)(69,20)(71,20)(72,19)(75,23)(77,24)(80,28)(81,26)(83,29)(86,28)%
(88,29)(90,27)(91,24)(93,24)(95,27)(98,25)(100,29)(101,29)(103,28)(106,33)%
(107,31)(109,34)(110,34)(112,32)(113,36)(115,37)(118,40)(120,39)(123,44)%
(125,43)(126,46)(128,50)(129,49)(132,54)(134,54)(136,61)(138,59)(141,66)%
(143,68)(145,68)(148,75)(149,73)(151,79)(153,81)(154,84)(156,83)(158,89)%
(160,93)(161,92)(164,94)(166,95)(167,92)(170,97)(173,93)(175,95)(177,97)%
(181,93)(182,91)(184,96)(186,96)(189,98)(191,103)(193,104)(196,101)(199,103)%
(202,105)(203,102)(205,104)(208,101)(210,102)(213,99)(215,97)(217,98)%
(220,96)(222,101)(224,101)(226,99)(229,103)(231,104)(233,101)(236,102)%
(239,99)(240,97)(242,98)(244,98)(247,101)(249,99)(252,102)(254,102)(255,103)%
(257,101)(258,101)(261,98)(262,96)(264,99)(267,97)(270,102)(272,102)%
(273,104)(275,101)(278,100)(279,98)(281,99)(283,97)(288,102)(290,102)
      \put(255,115){$\langle\nu^K_t,\mathbf{1}_{\{y\}}\rangle$}
      \dottedline{0.5}(0,130)(3,126)(6,130)(8,132)(9,132)(11,134)(13,130)%
(15,128)(17,127)(20,133)(24,130)(26,132)(29,129)(30,131)(33,128)(35,129)%
(38,126)(40,131)(43,132)(45,129)(47,134)(48,133)(50,131)(52,132)(55,128)%
(56,130)(57,127)(58,125)(59,128)(60,129)(61,123)(64,125)(66,121)(69,124)%
(71,120)(72,119)(75,123)(77,123)(80,118)(81,119)(83,115)(84,116)(86,113)%
(88,115)(90,116)(91,113)(93,116)(95,115)(97,116)(98,112)(100,108)(101,106)%
(103,109)(106,106)(107,102)(109,103)(111,98)(113,100)(116,97)(118,94)%
(119,95)(121,88)(123,90)(124,86)(126,80)(128,82)(130,77)(131,71)(132,73)%
(135,66)(137,68)(139,62)(140,57)(141,58)(143,52)(145,54)(148,48)(149,48)%
(151,43)(153,42)(154,43)(156,39)(157,39)(159,32)(160,33)(162,28)(164,30)%
(166,25)(167,23)(170,26)(172,24)(174,23)(175,25)(177,24)(181,27)(182,26)%
(184,23)(186,24)(189,20)(191,22)(193,24)(194,21)(196,19)(198,20)(200,18)%
(202,15)(203,17)(205,20)(208,18)(210,21)(213,17)(215,14)(217,16)(220,12)%
(222,11)(224,13)(226,9)(227,10)(229,7)(231,5)(233,7)(234,5)(236,8)%
(239,6)(240,9)(242,10)(244,7)(247,10)(249,6)(250,7)(252,5)(254,8)%
(255,4)(257,6)(258,4)(261,2)(262,4)(264,3)(267,1)(269,2)(270,0)
      \put(205,32){$\langle\nu^K_t,\mathbf{1}_{\{x\}}\rangle$}
    \end{picture}
  \end{center}
  \caption{{\small The three steps of the invasion of a
    mutant trait $y$ in a monomorphic population with trait $x$.}}
  \label{fig:inv-fix}
\end{figure}

\noindent In the first phase (between time 0 and $t_1$ in
Fig.~\ref{fig:inv-fix}), the number of mutant individuals is
small, and the resident population stays close to its equilibrium
density $\bar{\mathbf{n}}(\mathbf{x})$. Therefore, the dynamics of
the mutant individuals is close to a branching process with birth
rate $\lambda(y)$ and death rate
$\mu(y)+\sum_{i=1}^d\alpha(y,x_i)\bar{n}_i(\mathbf{x})$. Hence,
the growth rate of this branching process is equal to the
so-called fitness
\begin{equation}
  \label{eq:fitness}
  f(y; \textbf{x}) = f(y;x_1,\ldots,x_d)=r(y)-\sum_{j=1}^d\alpha(y,x_j)\bar{n}_j(\mathbf{x}),
\end{equation}
describing the ability of the initially rare mutant trait $y$ to
invade the equilibrium resident population with traits
$x_1,\ldots,x_d$. If this fitness is positive (i.e.\ if the
branching process is super-critical), the probability that the
mutant population reaches density $\eta >0$ at some time $t_1$ is
close to the probability that the branching process reaches $\eta
K$, which is itself close to its survival probability
$[f(y;\mathbf{x})]_+/\lambda(y)$ when $K$ is large.

\noindent In the second phase (between time $t_1$ and $t_2$ in
Fig.~\ref{fig:inv-fix}), we use the fact that, when
$K\rightarrow+\infty$,  the population densities
$(\langle\nu_t^K,\mathbf{1}_{\{x_1\}}\rangle,\ldots,\langle\nu_t^K,\mathbf{1}_{\{x_d\}}\rangle,\langle\nu_t^K,\mathbf{1}_{\{y\}}\rangle)$
are close to the solution of the Lotka-Volterra system
$LV(d+1,(x_1,\ldots,x_d,y))$ with same initial condition, on any
time interval $[0,T]$. We will need an assumption (called~(B1) in
Section~\ref{sec:hyp}) ensuring that, if $\eta$ is sufficiently
small, then any solution to the Lotka-Volterra system starting in
some neighborhood of
$(\bar{n}_1(\mathbf{x}),\ldots,\bar{n}_d(\mathbf{x}),0)$ converges
to a new equilibrium $\mathbf{n}^*\in\mathbb{R}^{d+1}$ as time
goes to infinity. Therefore, the population densities reach with
high probability the $\eta$-neighborhood of $\mathbf{n}^*$ at some
time $t_2$.

\noindent Finally, in the last phase, we use the same idea as in
the first phase: under the assumption (called~(B2) in
Section~\ref{sec:hyp}) that $\mathbf{n}^*$ is a strongly locally
stable equilibrium, we approximate the densities of the traits
$x_j$ such that $n^*_j=0$ by branching processes which are
sub-critical. Therefore, they reach $0$ in finite time and the
process comes back to the first step until the next mutation.

\noindent We will prove that the duration of these three phases is
of order $\log K$. Therefore, under the assumption
$$
\log K \ll \frac{1}{Ku_K},
$$
the next mutation occurs after these three phases are completed with
high probability.

\subsubsection{Assumptions}
\label{sec:hyp}

As explained above, we need to introduce two assumptions on the
Lotka-Volterra systems involved in the previous heuristics. These
assumptions involve the fitness function defined
in~(\ref{eq:fitness}). This function is linked to Lotka-Volterra
systems by the following property.

\begin{prop}
  \label{prop:fitness} Assume that the traits $x_1,\ldots,x_d\in{\cal X}$
  coexist. Then
  \begin{description}
  \item[\textmd{(i)}] For any $i\in\{1,\ldots,d\}$,
    $f(x_i;x_1,\ldots,x_d)=0$.
  \item[\textmd{(ii)}] If $f(y;x_1,\ldots,x_d)<0$, the equilibrium
    $\ (\bar{n}_1(\mathbf{x}),\ldots,\bar{n}_d(\mathbf{x}),0)$ of
    $\ LV(d+1,(x_1,\ldots,x_d,y))$ is locally strongly stable, and if
    $\ f(y;x_1,\ldots,x_d)>0$, this equilibrium is unstable.
  \end{description}
\end{prop}

\begin{proof}
The first point is immediate. The second point comes from the
following relation between Jacobian matrices of Lotka-Volterra
systems
\begin{equation*}
  DF^{(x_1,\ldots,x_d,y)}(\bar{n}_1(\mathbf{x}),\ldots,\bar{n}_d(\mathbf{x}),0)
  =\left(
    \begin{array}{c|c}
      DF^{\mathbf{x}}(\bar{\mathbf{n}}(\mathbf{x})) &
      \begin{array}{c}
        -\bar{n}_1(\mathbf{x})\alpha(x_1,y) \\ \vdots \\
        -\bar{n}_d(\mathbf{x})\alpha(x_d,y)
      \end{array}
      \\ \hline \\
      \begin{array}{ccc}
        0 & \ldots & 0
      \end{array}
      & f(y;\mathbf{x})
    \end{array}
  \right).
\end{equation*}
Since $x_1,\ldots,x_d$ coexist, all the eigenvalues of
$DF^{\mathbf{x}}(\bar{\mathbf{n}}(\mathbf{x}))$ have negative real
parts.
\end{proof}
\bigskip

\noindent Let~(B) denote the following
Assumptions~(B1) and~(B2).

\begin{description}
\item[\textmd{(B1)}] Given any $\mathbf{x}=(x_1,\ldots,x_d)\in{\cal
    X}^d$ such that $x_1,\ldots,x_d$ coexist, for Lebesgue almost any
  mutant trait $y\in{\cal X}$ such that
  $f(y;\mathbf{x})>0$, there exists a neighborhood ${\cal
    U}\subset\mathbb{R}^{d+1}$ of
  $(\bar{n}^1(\mathbf{x}),\ldots,\bar{n}^d(\mathbf{x}),0)$ such that
  all the solutions of $LV(d+1,(x_1,\ldots,x_{d},y))$ with initial
  condition in ${\cal U}\cap(\mathbb{R}^*_+)^{d+1}$ converge as
  $t\rightarrow+\infty$ to a
  unique equilibrium in $(\mathbb{R}_+)^{d+1}$, denoted by
  \begin{equation*}
    \mathbf{n}^*(x_1,\ldots,x_{d},y).
  \end{equation*}
\item[\textmd{(B2)}] Writing for simplicity $x_{d+1}=y$ and $\mathbf{n}^*$ for
  $\mathbf{n}^*(x_1,\ldots,x_{d+1})$, let
  \begin{equation*}
    I(\mathbf{n}^*):=\big\{i\in\{1,\ldots,d+1\}:
    n^*_i>0\big\}\quad \hbox{ and }\quad \mathbf{x}^*=(x_i;i\in
    I(\mathbf{n}^*)).
  \end{equation*}
  Then, for Lebesgue almost any mutant trait $x_{d+1}$ as above,
  $\{x_i;i\in I(\mathbf{n}^*)\}$ coexist and
  \begin{equation*}
   \hbox{for all } j\not\in
  I(\mathbf{n}^*)\ , \quad f(x_j;\mathbf{x}^*)<0.
  \end{equation*}
\end{description}

\noindent Assumption~(B1) prevents cycles or chaotic dynamics in
the Lotka-Volterra systems. Moreover, it also prevents situations
as in Fig.~\ref{fig:eq}, where the equilibrium $\mathbf{n}^*$ is
unstable. In this case, a solution of the Lotka-Volterra system
$LV(d+1,(x_1,\ldots,x_d,y))$ starting from a point in any
neighborhood of
$(\bar{n}^1(\mathbf{x}),\ldots,\bar{n}^d(\mathbf{x}),0)$,
represented by the curved line in Fig.~\ref{fig:eq}, does not need
to converge to $\mathbf{n}^*$.

\begin{figure}[h]
  \begin{center}
    \psfrag{N1}{$(\bar{n}^1(\mathbf{x}),\ldots,\bar{n}^d(\mathbf{x}),0)$}
    \psfrag{N2}{$\mathbf{n}^*$}
    \includegraphics[width=6cm]{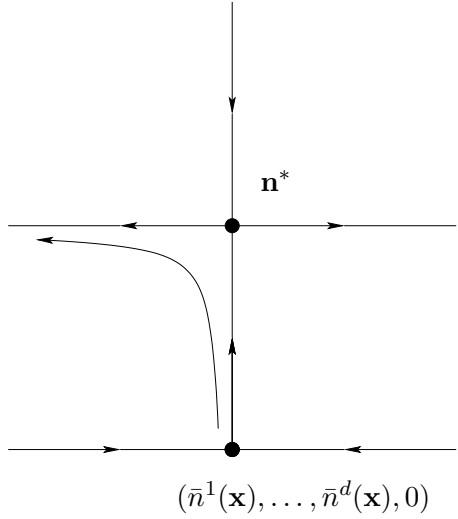}
  \end{center}
  \caption{{\small Assumption~(B1) prevents such situation.}}
  \label{fig:eq}
\end{figure}
\noindent Assumption~(B2) is stated in the way permitting one to
use the comparison with branching processes argument described in
Section~\ref{sec:idea-pf} when a mutant trait fixates in the
population. 

\begin{defi}
  \label{def:hyperbolic}
  An equilibrium $\mathbf{n}$ of $LV(d,(x_1,\ldots,x_d))$ is
  hyperbolic if the Jacobian matrix of $LV(d,(x_1,\ldots,x_d))$ at
  $\mathbf{n}$ has no eigenvalue with 0 real part.
\end{defi}
Assumption~(B2) can also be replaced by one of the following simpler
two assumptions.

\begin{description}
\item[\textmd{(B3)}] For Lebesgue almost any mutant trait $x_{d+1}$ as in~(B1),
  $\mathbf{n}^*$ is hyperbolic.
\item[\textmd{(B4)}] For Lebesgue almost any mutant trait
$x_{d+1}$ as in~(B1),
  $\mathbf{n}^*$ is strongly locally stable.
\end{description}

\begin{prop}
  \label{prop:equiv-hyp} Assumptions~(B1) and~(B2) are equivalent to
  Assumptions~(B1) and~(B3), and to
  Assumptions~(B1) and~(B4).
\end{prop}

\begin{proof}
  Let $k:=\mbox{Card}(I(\mathbf{n}^*))$.  Assume that
  $x_1,\ldots,x_{d+1}$ are reordered in a way such that
  $I(\mathbf{n}^*)=\{1,2,\ldots,k\}$. Then it is clear, by the
  definition of coexistence and the fact that
  \begin{equation*}
    DF^{(x_1,\ldots,x_{d+1})}(\mathbf{n}^*)
    =\left(
      \begin{array}{c|c}
        DF^{\mathbf{x}^*}(n^*_1,\ldots,n^*_k) &
        (-\alpha(x_i,x_j)n^*_j)_{1\leq i\leq k,\: k+1\leq j\leq d+1}
        \\ \hline \\
        \begin{array}{ccc}
          \mbox{\LARGE $0$}
        \end{array}
        &
        \begin{array}{ccc}
          f(x_{k+1},\mathbf{x}^*) & & 0 \\ & \ddots & \\ 0 & &
          f(x_{d+1},\mathbf{x}^*)
        \end{array}
      \end{array}
    \right)
  \end{equation*}
  that (B2) implies (B4) which also trivially implies
  (B3). Assuming~(B3), the stable manifold theorem (see
  e.g.~\cite{guckenheimer-holmes-83} pp. 13--14) says that the set of
  points such that the solution of $LV(d+1,(x_1,\ldots,x_{d+1}))$
  started at this point converges to $\mathbf{n}^*$ is a submanifold
  of $(\mathbb{R}_+^*)^{d+1}$ of dimension $l$, where $l$ is the number
  of eigenvalues of $DF^{(x_1,\ldots,x_{d+1})}(\mathbf{n}^*)$ with
  negative real part. In particular, if $l<d+1$, this manifold does
  not contain an open set of $(\mathbb{R}_+^*)^{d+1}$, which is in
  contradiction with~(B1). Therefore, $l=d+1$, which implies~(B2).
\end{proof}

\noindent Therefore, Assumption~(B2) essentially means that
$\mathbf{n}^*$ is hyperbolic, which is a property satisfied under
very weak assumptions. In Section ~\ref{sec:CP}, various
situations ensuring Assumptions~(B1) and~(B2) will be discussed.

\subsubsection{Definition of the PES and Convergence Theorem}
\label{sec:thm}

Before stating our convergence result, let us first describe the
limiting process $(Z_t;t\geq 0)$ of the population process
$(\nu^K_{t/K u_K};t\geq 0)$ on the mutation time scale. This is a pure
jump Markov process in ${\cal
  M}_0\subset{\cal M}_F$ defined by
\begin{equation*}
  {\cal M}_0:=\left\{\sum_{i=1}^d\bar{n}_i(\mathbf{x})\delta_{x_i};
    \:d\geq 1,\:x_1,\ldots,x_n\in{\cal X}\mbox{\ coexist}\right\},
\end{equation*}
which describes the successive population states at the
evolutionary (mutation) time scale.

\noindent As explained in Section~\ref{sec:idea-pf}, the quantity $\:
p(x_j)\lambda(x_j)\bar{n}_j(\mathbf{x})$ is the re-scaled mutation
rate in the resident sub-population with trait $x_j$ and size
$\bar{n}_j(\mathbf{x})$. When a mutant $x_j+h$ is chosen with law
$m(x_j,h)dh$, the quantity $\frac{[f(x_j+h;\mathbf{x})]_+}
{\lambda(x_j+h)}$ is the invasion probability of the mutant. Once the
latter has invaded, the new population state is given by the
asymptotic behavior of the Lotka-Volterra system described in
Assumption (B1). Because of the timescale separation \eqref{eq:u_K-K},
the stabilization of the population at its new equilibrium occurs
before the next mutation and within infinitesimal time.

\noindent Hence, the process $Z$ will jump   $$\hbox{from }\
\sum_{i=1}^d \bar{n}_i(\mathbf{x})\delta_{x_i}\  \hbox{ to }\
\sum_{i=1}^d n^*_i(x_1,\ldots,x_d,x_j+h)\delta_{x_i}
+n^*_{d+1}(x_1,\ldots,x_d,x_j+h)\delta_{x_j+h}$$ with
infinitesimal rate
\begin{equation}
\label{taux}
  p(x_j)\lambda(x_j)\bar{n}_j(\mathbf{x})\frac{[f(x_j+h;\mathbf{x})]_+}
  {\lambda(x_j+h)}m(x_j,h)dh
\end{equation}
for all $j\in\{1,\ldots,d\}$.  In other words, the infinitesimal
generator of the process $Z$ will be
\begin{multline}
  {\cal L}\varphi\Big(\sum_{i=1}^d\bar{n}_i(\mathbf{x})\delta_{x_i}\Big)
  =\int_{{\cal X}} dh\sum_{j=1}^d p(x_j)\lambda(x_j)\bar{n}_j(\mathbf{x})
  \frac{[f(x_j+h;\mathbf{x})]_+}
  {\lambda(x_j+h)}m(x_j,h)\times \\
  \Big(
    \varphi\Big(\sum_{i=1}^{d}n^*_i(x_1,\ldots,x_d,
    x_j+h)\delta_{x_i}+n^*_{d+1}(x_1,\ldots,x_d, x_j+h)\delta_{x_j+h}\Big)
    -\varphi\Big(\sum_{i=1}^d\bar{n}_i(\mathbf{x})\delta_{x_i}\Big)
  \Big).
  \label{eq:gene-PES}
\end{multline}

\noindent We call this process \emph{Polymorphic Evolution
Sequence} (PES), by analogy with the so-called ``Trait
Substitution Sequence'' (TSS) described in
Section~\ref{sec:TSS}.

\begin{prop}
  Under Assumptions \textup{(A)} and \textup{(B)}, the PES is
  well-defined on $\mathbb{R}_+$ and belongs almost surely to ${\cal
    M}_0$ for all time.
\end{prop}

\begin{proof}
  It follows from Assumption~(A) and from~(\ref{eq:F=0}) that the jump
  rates are bounded. Moreover, by Assumption~(B1),
  $\mathbf{n}^*(x_1,\ldots,x_n,y)$ is well-defined for almost all
  mutant traits $y$ such that $f(y;\mathbf{x})>0$, and by
  Assumption~(B2), for such $y$, $\sum_{i=1}^d
  n^*_i(x_1,\ldots,x_d,y)\delta_{x_i}
  +n^*_{d+1}(x_1,\ldots,x_d,y)\delta_{y}\in{\cal M}_0$.
\end{proof}

\begin{thm}
  \label{thm:PES-fdd}
  Assume \textup{(A)} and \textup{(B)}. Take $x_1,\ldots,x_d\in{\cal
    X}$ that coexist and assume that
  $\nu^K_0=\sum_{i=1}^dn^K_i\delta_{x_i}$ with $n^K_i\rightarrow
  \bar{n}_i(\mathbf{x})$ in probability for all $1\leq i\leq
  d$. Assume finally that
  \begin{equation}
    \label{eq:u_K-K}
    \forall V>0,\quad \log K\ll \frac{1}{Ku_K}\ll \exp(VK).
  \end{equation}
  Then, $(\nu^K_{t/K u_K};t\geq 0)$ converges to the process $(Z_t;
  t\geq 0)$ with infinitesimal generator~(\ref{eq:gene-PES}) and with
  initial condition
  $Z_0=\sum_{i=1}^d\bar{n}_i(\mathbf{x})\delta_{x_i}$. The convergence
  holds in the sense of finite dimensional distributions on ${\cal
    M}_F$ equipped with the topology induced by the functions
  $\nu\mapsto\langle \nu,f\rangle$ with $f$ bounded and measurable on
  ${\cal X}$.
\end{thm}

\noindent The proof of this result follows closely the heuristic
argument of Section~\ref{sec:idea-pf} and is very similar to the
proof of Theorem~1 of~\cite{champagnat-06}, that states a similar
result in the case where no pair of traits can coexist. We detail
in Appendix~\ref{sec:pf-PES} all the steps and results
of~\cite{champagnat-06} that are modified in order to prove
Theorem~\ref{thm:PES-fdd}.

\section{Particular cases and extensions of the PES}
\label{sec:CP}

In this section, we discuss various situations where
Assumptions~(B1) and~(B2) are satisfied allowing one to explicitly
obtain the PES.

\subsection{The "no coexistence" case: an extension of the trait substitution sequence (TSS)}
\label{sec:TSS}

In this section we characterize the case where the PES is well
defined until the first co-existence time of two different traits.
Assumption (B) with $d=1$  (only one resident trait) involves the
fitness function defined in \eqref{eq:fitness-d=1}.

\begin{prop} \label{prop:D1->B} Let us assume the hypothesis \begin{description}
\item[\textmd{(C1)}] For  all $x\in{\cal X}$, the set of $y$ such
that $f(y;x)=0$ has Lebesgue measure 0.
\end{description}
Then \textup{(B)} is satisfied for $d=1$.
\end{prop}

\begin{proof}The assumption (B) for $d=1$ involves 2-dimensional
competitive Lotka-Volterra systems. Their asymptotic behavior  is
well-known (see e.g.~\cite{istas-05}). In particular,
\begin{itemize}
\item if $f(x;y)>0$ and $f(y;x)<0$, any solution of $LV(2,(x,y))$
  starting from $\RR_+\times\RR_+^*$ converges to $(\bar{n}(x),0)$,
\item if $f(x;y)<0$ and $f(y;x)>0$, any solution of $LV(2,(x,y))$
  starting from $\RR_+^*\times\RR_+$ converges to $(0,\bar{n}(y))$,
\item if $f(x;y)>0$ and $f(y;x)>0$, any solution of $LV(2,(x,y))$
  starting from $(\RR_+^*)^2$ converges to $\bar{\mathbf{n}}(x,y)$,
\item if $f(x;y)<0$ and $f(y;x)<0$, $(\bar{n}(x),0)$ and
  $(0,\bar{n}(y))$ are both locally strongly stable.
\end{itemize}
Moreover, all the equilibria are hyperbolic if and only if
$f(y;x)\not=0$ and $f(x;y)\not=0$. Therefore, Assumption (C1)
implies Assumption~(B) for $d=1$ since $m(x,h)dh$ is absolutely
continuous w.r.t.\ Lebesgue's measure.
\end{proof}

\noindent Let us now introduce the following killed  PES
$({Z_t^{(1)}},t\geq 0)$ as a Markov jump process on ${\cal
M}_0\cup \{\partial\}$, where  $\partial$ is a cemetery state,
with infinitesimal generator ${\cal L}^{(1)}$ defined as follows.
Let $\nu:=\bar{n}(x)\delta_{x}$, then
\begin{multline}
  {\cal L}^{(1)}\varphi(\bar{n}(x)\delta_{x}) \\=\int_{{\cal X}}\Big(
  \varphi\Big(\bar{n}(x+h)\delta_{x+h}\Big)
  -\varphi(\bar{n}(x)\delta_{x}) \Big)
  p(x)\lambda(x)\bar{n}(x) \frac{[f(x+h;x)]_+}
  {\lambda(x+h)}\mathbf{1}_{\{f(x;x+h)<0\}}m(x,h)dh \\
   +\int_{{\cal X}}\big(\varphi(\partial)
  -\varphi(\bar{n}(x)\delta_{x})\big) p(x)\lambda(x)\bar{n}(x)
  \mathbf{1}_{\{f(x;x+h)>0, f(x+h;x)>0\}}m(x,h)dh. \label{eq:1gene-mod-PES}
\end{multline}
\noindent By construction, the killed PES $({Z_t^{(1)}},t\geq 0)$
is always monomorphic before killing. Once the killed PES reaches
the cemetery state $\partial$, it no longer jumps.

\noindent This modification amounts to construct the killed PES as the
PES, and send it to the cemetery state $\partial$ as soon as a mutant
trait $y$ appears in a monomorphic population of trait $x\in{\cal X}$
such that $x$ and $y$ coexist. Note that $\partial$ is reached as soon
as a mutant \emph{appears}, that could coexist with the resident
trait, even if this mutant actually does not \emph{invade} the
population. That explains why the invasion probability
$[f(y;x)]_+/\lambda(y)$ does not appear in the last line
of~(\ref{eq:1gene-mod-PES}).

\noindent The following proposition is a consequence of the
previous discussion.
\begin{prop}
  \label{prop:mod-PES-1}
  Under Assumptions \textup{(A)} and \textup{(C1)}, the killed PES
  $({Z_t^{(1)}},t\geq 0)$ is almost surely well-defined and belongs
  almost surely to ${\cal M}_0\cup\{\partial\}$ for all time.
\end{prop}

\noindent  The proof of the following result can be easily adapted
 from that of Theorem~\ref{thm:PES-fdd}.

\begin{cor}
  \label{cor:double}
  With the same assumption and notation as in
  Theorem~\textup{\ref{thm:PES-fdd}}, except that
  Assumption~\textup{(B)} is replaced by Assumption~\textup{(C1)} and
  that $d=1$, let
  $$
  \tau_{K}:=\inf\{t\geq
  0:\mbox{Supp}(\nu_t^K)=\{x,y\}\mbox{\ such that\ }(x,y) \hbox{ coexist}\}.
  $$
  Then the process
  $$
  \Big(\nu^K_{\frac{t}{K
      u_K}}\mathbf{1}_{\{\frac{t}{K
      u_K}\leq\tau_{K}\}}+\partial\,\mathbf{1}_{\{\frac{t}{K
      u_K} > \tau_{K}\}},t\geq 0\Big)
  $$
  converges as $K\rightarrow+\infty$ to the killed PES
  $({Z_t^{(1)}},t\geq 0)$ with initial condition
  $Z^{(1)}_0=\bar{n}(x)\delta_{x}$. The
  convergence is understood in the same sense as in
  Theorem~\textup{\ref{thm:PES-fdd}}.
\end{cor}

\begin{rem} \label{rem-tss} The killed PES generalizes the so-called
  ``Trait Substitution Sequence'' (TSS), introduced
  in~\textup{\cite{metz-geritz-al-96}}, and rigorously studied
  in~\textup{\cite{champagnat-06}}.  This TSS is obtained when the
  parameters of the microscopic model prevent the coexistence of two
  traits. Such an assumption, known as ``Invasion-Implies-Fixation''
  \textup{(IIF)} principle~\textup{\cite{geritz-gyllenberg-al-02}} is
  given by:
  \begin{description}
  \item[\textmd{(IIF)}] for all $x\in{\cal X}$, almost all
    $y\in{\cal X}$ such that $f(y;x)>0$ satisfy $f(x;y)<0$.
  \end{description}
  Hence, the TSS $Z$  has on $\mathbb{R}_+$ the form
  $$
  Z_t=\bar{n}(X_t)\delta_{X_t},\quad t\geq 0,
  $$
  where $X$ is a Markov pure jump process on ${\cal X}$ with
  infinitesimal generator
  \begin{equation}
    \label{eq:gen-TSS}
    L\varphi(x)=\int_{\mathbb{R}^l}(\varphi(x+h)-\varphi(x))p(x)\lambda(x)
    \bar{n}(x)\frac{[f(x+h;x)]_+}{\lambda(x+h)}m(x,h)dh.
  \end{equation}
\end{rem}
\medskip

\noindent The killed PES $(Z_t^{(1)},t\geq 0)$ prevents the
coexistence of two or more traits. Therefore, this process is not
suited to our study of evolutionary branching in
Section~\ref{sec:br}. To this end, we need to examine a more
general situation.

\subsection{The ``no triple coexistence'' case}
\label{sec:triple}

In this section we characterize the case where the PES is well
defined until the first coexistence time of three different
traits.

\noindent In the case $d=2$ the fitness function
\eqref{eq:fitness} of a mutant trait $z$ in a population with two
coexisting resident traits $x$ and $y$ is given by
\begin{equation}
  \label{eq:def-f3}
  f(z;x,y)=r(z)-\alpha(z,x)\bar{n}_1(x,y)-\alpha(z,y)\bar{n}_2(x,y)
\end{equation}
with
\begin{align}
  \bar{n}_1(x,y) & =\frac{r(x)\alpha(y,y)-r(y)\alpha(x,y)}
  {\alpha(x,x)\alpha(y,y)-\alpha(x,y)\alpha(y,x)},
  \label{n1comp} \\
  \bar{n}_2(x,y) & =\frac{r(y)\alpha(x,x)-r(x)\alpha(y,x)}
  {\alpha(x,x)\alpha(y,y)-\alpha(x,y)\alpha(y,x)}. \label{n2comp}
\end{align}
We need to extend this definition  to any $x,y\in{\cal X}$ such
that $f(x;y)f(y;x)>0$ (and not only for the ones that coexist). It
can be easily checked that
$\alpha(x,x)\alpha(y,y)-\alpha(x,y)\alpha(y,x)$ cannot be 0 under
this condition.

 \medskip \noindent We can now introduce the following assumption :
\begin{description}
 \item[\textmd{(C2)}] For all $x,y\in{\cal X}$
  that coexist, the set of $z$ such that $f(x;z)=0$, $f(z;x)=0$,
  $f(y;z)=0$, $f(z;y)=0$, $f(x;y,z)=0$ or $f(y;x,z)=0$ (when these
  last quantities are defined) has Lebesgue measure 0.
\end{description}

\begin{prop}  
  \label{prop:C2->B} 
  There exists a set $C_{coex}$ (defined in \eqref{eq:def-Cn}) such
  that Assumption \textup{(C2)} implies \textup{(B)} for $d=2$ and for
  all $(x,y,z)\in{\cal X}^3\setminus C_{coex}$.
\end{prop}

\begin{proof}
As in the previous section, we have to distinguish coexistence
 and non coexistence of three traits. To this aim  we need to
 introduce the  classification of the asymptotic behavior of 3-dimensional
 competitive Lotka-Volterra systems
 done by  Zeeman~\cite{zeeman-93}. Any  $3$-dimensional competitive Lotka-Volterra system
admits an invariant hypersurface $\Sigma$ called carrying simplex,
such that any non-zero solution of the system is asymptotic as
$t\rightarrow+\infty$ to one in $\Sigma$ (cf. \cite{hirsch-88}).
$\Sigma$ is a Lipschitz submanifold of $\RR^3_+$ homeomorphic to
the unit simplex in $\RR^3_+$ by radial projection. Moreover,
$\Sigma$ is a global attractor for the dynamics in
$\RR^3_+\setminus\{0\}$ (\cite[Thm.3]{hirsch-08}). In particular,
one can deduce from the asymptotic behavior of trajectories on
$\Sigma$ the asymptotic behavior of trajectories starting in a
neighborhood of $\Sigma$.

\noindent Zeeman obtained a full classification of the topological
equivalence classes of the 3-dimensional competitive Lotka-Volterra
systems by determining the 33 topological equivalence classes of those
systems restricted on their carrying simplex. (In an equivalence
class, the trajectories of the systems are related by a homeomorphism
of $\RR^3_+$). For a given system $LV(3,(x,y,z))$, the equivalence
class to which it belongs is determined by the sign of the
2-dimensional fitnesses $f(x;y),\ f(y;x),\ f(x;z),\ f(z;x),\ f(y;z),\
f(z;y)$ and of the 3-dimensional fitnesses $f(x;y,z),\ f(y;x,z),\
f(z;x,y)$ when they are defined. The equivalence classes
of~\cite{zeeman-93} are characterized by drawing on the unit simplex
of $\RR^3_+$ the fixed points and the limit cycles of the system,
joined by their stable and unstable manifolds\footnote{The stable
  manifold of an equilibrium is the set of starting points of the
  Lotka-Volterra system such that the solution converges to this
  equilibrium. The unstable manifold is defined in the same way, but
  for the time-reversed system.}.

\begin{figure}[h]
  \begin{center}
    \psfrag{X}{$x$}
    \psfrag{Y}{$y$}
    \includegraphics[width=8cm]{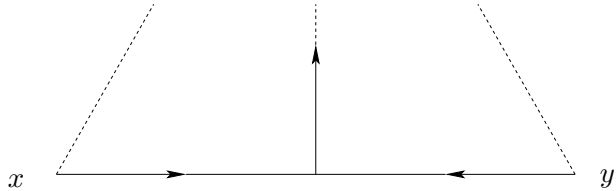}
  \end{center}
  \caption{{\small The pattern on the carrying simplex that corresponds to the
    situation of Assumption~(B). Traits $x$ and $y$ are the resident
    traits.}}
  \label{fig:pattern}
\end{figure}
\noindent The signs of the fitnesses correspond to the arrows in
each diagram. For example, $f(y;x)>0$ means that, on the edge of
the simplex that reach $x$ and $y$, there is an arrow starting
from $x$ in the direction of $y$. In other words, the unstable
manifold of $(\bar{n}(x),0,0)$ contains (a part of) the edge of
the simplex that reach $x$ to $y$. Similarly, $f(z;x,y)>0$ means
that $f(x;y)f(y;x)>0$, i.e.\ that $\ LV(3,(x,y,z))$ has as fixed
point $(\bar{n}_1(x,y),\bar{n}_2(x,y),0)$, represented as the
midpoint of the edge of the simplex linking $x$ and $y$, and that
this fixed point has an unstable manifold pointing in the
direction of the interior of the simplex. The situation
represented in Fig.~\ref{fig:pattern} corresponds to this case, when
$x$ and $y$ coexist.

\noindent In order to check if Assumption~(B) holds, we only need to
restrict to the equivalence classes in which two traits coexist (the
resident traits, say $x$ and $y$), and the third (mutant) trait (say
$z$) satisfy $f(z;x,y)>0$. This situation corresponds to the cases
where the carrying simplex has one side containing the pattern of
Fig.~\ref{fig:pattern}. Among the 33 equivalence classes
of~\cite{zeeman-93}, there are only 10 of them that satisfy this
requirement, shown in Fig~\ref{fig:10}. We label them with the same
numbers as in~\cite{zeeman-93}.  In Fig.~\ref{fig:10}, the figures
obtained by exchanging $x$ and $y$ belong to the same equivalence
class.  An attracting fixed point of $LV(3,(x,y,z))$ is represented by
a closed dot $\bullet$, a repulsive fixed point by an empty dot
$\circ$, a saddle point by the intersection of its stable and unstable
manifolds. When the interior fixed point (the non-trivial equilibrium)
is not a saddle point, it can be either stable or unstable. Depending
on cases, this equilibrium can also be surrounded by one or several
stable or unstable cycles. In particular, the sign of the fitnesses is
not sufficient to determine the precise asymptotic behavior of the
system near the interior equilibrium. The undetermined type of these
equilibria is represented in Fig.~\ref{fig:10} by the symbol $\odot$.

\begin{figure}[h]
  \begin{center}
    \psfrag{7}{7}
    \psfrag{8}{8}
    \psfrag{9}{9}
    \psfrag{10}{10}
    \psfrag{11}{11}
    \psfrag{12}{12}
    \psfrag{26}{26}
    \psfrag{29}{29}
    \psfrag{31}{31}
    \psfrag{33}{33}
    \includegraphics[width=12cm]{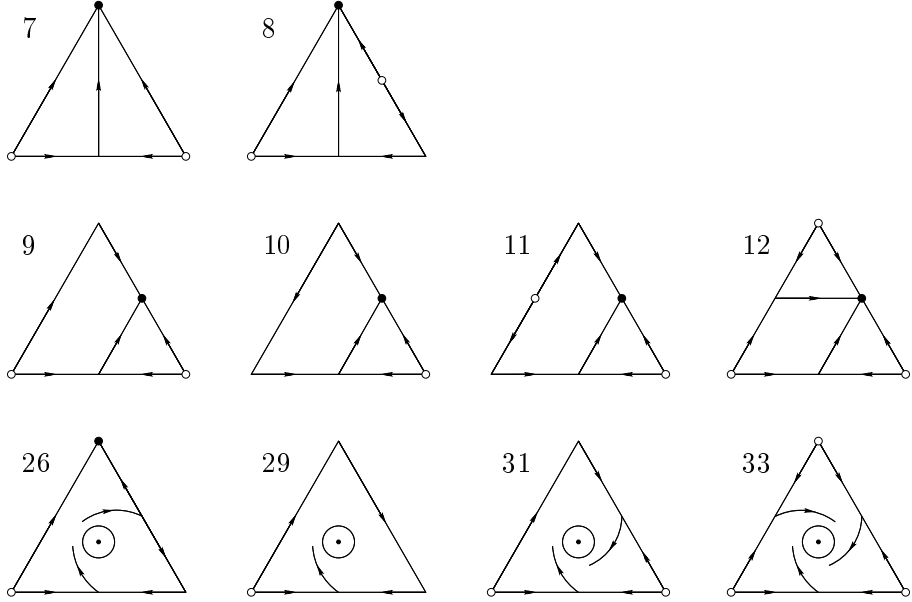}
  \end{center}
  \caption{{\small The phase portrait on $\Sigma$ for the 3-dimensional competitive
   Lotka-Volterra
    systems corresponding to the situation described in
    Assumption~(B).}}
  \label{fig:10}
\end{figure}
\noindent Since no pattern as in Fig.~\ref{fig:eq} occurs in
diagrams 7 to 12 in Fig.~\ref{fig:10}, we see that Assumption~(B1)
is always satisfied except possibly in the cases of diagrams 26,
29, 31 and 33. It could be violated either if the interior
equilibrium is surrounded by a stable cycle, or in diagram~26 in
the case where the unstable manifold of the equilibrium
$(\bar{n}_1(x,y),\bar{n}_2(x,y),0)$ (midpoint of the lower edge of
the simplex) admits the equilibrium
$(0,\bar{n}_1(y,z),\bar{n}_2(y,z))$ as limit point. Moreover, as
before, all the steady states are hyperbolic if all the 2- and
3-dimensional fitnesses are nonzero.

\noindent Thus, if we define the set $C_{coex}$ as
\begin{equation}
  \label{eq:def-Cn}
  C_{coex}:=\{(x,y,z)\in{\cal X}^3:LV(3,(x,y,z))\mbox{\
    belongs to classes 26, 29, 31 or 33}\},
\end{equation}
 Assumption (B) will be satisfied for all $(x,y,z)\in{\cal X}^3\backslash C_{coex}$
  as soon as Assumption (C2) is
 satisfied.

\medskip
\noindent Remark in addition that, as checked from
Fig.~\ref{fig:10}, if $x$ and $y$ coexist and $f(z;x,y)>0$, then
$(x,y,z)\in C_{coex}$ if and only if both of the following
properties are satisfied
\begin{description}
\item[\textmd{(P1)}] If $f(y;x,z)$ is well-defined, then $f(x;z)$,
  $f(z;x)$ and $f(y;x,z)$ have all the same sign.
\item[\textmd{(P2)}] If $f(x;y,z)$ is well-defined, then $f(y;z)$,
  $f(z;y)$ and $f(x;y,z)$ have all the same sign.
\end{description}
\end{proof}

\noindent Assumptions (C1) and (C2) will be summarized in
Assumption (C).

\noindent Similarly as in Section \ref{sec:TSS}, we define the
killed PES $(Z_t^{(2)},t\geq 0)$ as a Markov pure jump process on
${\cal M}_0 \cup \{\partial\}$, with infinitesimal  generator
${{\cal L}^{(2)}}$. The latter is  given by~(\ref{eq:gene-PES})
for $d=1$, and for $d=2$ and coexisting $x_1, x_2$, it is modified
as follows. Let $\nu:=\sum_{i=1}^2\bar{n}_i(x_1,x_2)\delta_{x_i}$,
then
\begin{align}
  &{{\cal L}^{(2)}}\varphi(\nu) =\int_{\mathbb{R}^l}\sum_{j=1}^2 \Big(
  \varphi\Big(\sum_{i=1}^{2}n^*_i(x_1,x_2,x_j+h)\delta_{x_i}+n^*_3(x_1,x_2,x_j+h)\delta_{x_j+h}\Big)
  -\varphi(\nu) \Big)\times \notag \\ & \quad
  p(x_j)\lambda(x_j)\bar{n}_j(x_1,x_2) \frac{[f(x_{j}+h;x_1,x_2)]_+}
  {\lambda(x_{j}+h)}\mathbf{1}_{\{(x_1,x_2,x_j+h)\not\in
    C_{coex}\}}m(x_j,h)dh \notag \\
  & \quad +\int_{\mathbb{R}^l}\sum_{j=1}^2 \big(\varphi(\partial)
  -\varphi(\nu)\big) p(x_j)\lambda(x_j)\bar{n}_j(x_1,x_2)
  \mathbf{1}_{\{(x_1,x_2,x_j+h)\in
    C_{coex}\}}m(x_j,h)dh. \label{eq:gene-mod-PES}
\end{align}

\noindent This modification amounts to construct the killed PES as
the PES, and send it to the cemetery state as soon as a mutant
trait $x_3$ appears in a dimorphic population of traits
$x_1,x_2\in{\cal X}$ such that the Lotka Volterra dynamics
associated with traits $x_1,x_2,x_3$ belongs to classes 26, 29, 31
or 33. Notice that the killed PES's support  has at most two traits at
each time.

\medskip\noindent
As in Section \ref{sec:TSS}, we deduce the following results.
\begin{prop}
  \label{prop:mod-PES-2}
  Under Assumptions \textup{(A)} and \textup{(C)}, the killed PES
  $(Z_t^{(2)}, t\geq 0)$ is almost surely well-defined and belongs
  almost surely to ${\cal M}_0\cup\{\partial\}$ for all time.
\end{prop}
 Moreover this PES is the limiting
process, on the mutation time scale, of the mutation-invasion
process killed at the first triple-coexistence time.
\begin{cor}
  \label{cor:triple}
  With the same assumption and notation as in
  Theorem~\textup{\ref{thm:PES-fdd}}, except that
  Assumption~\textup{(B)} is replaced by Assumption~\textup{(C)} and
  that $d\in\{1,2\}$, let
  $$
  \tilde\tau_{K}:=\inf\{t\geq
  0:\mbox{Supp}(\nu_t^K)=\{x,y,z\}\mbox{\ such that\ }(x,y,z)\in
  C_{coex}\}.
  $$
  Then the process
  $$
  \Big(\nu^K_{\frac{t}{K
      u_K}}\mathbf{1}_{\{\frac{t}{K
      u_K}\leq\tilde\tau_{K}\}}+\partial\ \mathbf{1}_{\{\frac{t}{K
      u_K} > \tilde\tau_{K}\}},t\geq 0\Big)
  $$
  converges as $K\rightarrow+\infty$ to the killed PES
  $(Z^{(2)}_t,t\geq 0)$ with initial condition
  $Z^{(2)}_0=\sum_{i=1}^d\bar{n}_i(\mathbf{x})\delta_{x_i}$.
\end{cor}

\noindent Note that the killed PES obtained in this section is
sufficient to study the phenomenon of evolutionary branching in
Section~\ref{sec:br} when ${\cal X}\subset\mathbb{R}$.

\section{Evolutionary branching and small jumps}
\label{sec:br}

We will assume, in all what follows, that the initial population is
monomorphic, in the sense that at time $0$, all individuals have the
same trait.

\noindent We have seen in Section~\ref{sec:TSS} that, as long as
 there is no coexistence of two traits in the population (Assumption (IIF)),
the support of the PES is reduced to a single trait and the
asymptotic dynamics of the population is given by the killed PES
$Z^{(1)}$ with generator~(\ref{eq:1gene-mod-PES}). In this
section, our aim is to characterize the traits around which (IIF)
 fails and how evolutionary branching can occur in this case, as observed in Fig.~\ref{fig:ex2}(b).
 To do
so, following a general idea of the biological
literature~\cite{metz-geritz-al-96,dieckmann-law-96,geritz-metz-al-97,geritz-gyllenberg-al-02,diekmann-jabin-al-05,geritz-05},
a key assumption is that the mutation amplitude is small. In this
situation, we will study the behavior of the PES on large time
scales which will allow us to observe a global evolutionary
dynamics.

\noindent In Subsection~\ref{sec:caneq}, we assume that (IIF) is
always satisfied and we study the TSS with a small mutation step
scaling $\varepsilon$. We prove that on the longer time scale ${t\over
  \varepsilon^2}$, the dynamics of the re-scaled TSS converges, when
$\varepsilon$ tends to zero, to the solution of a (deterministic) ODE,
called \emph{canonical equation of adaptive dynamics}, or, more simply
\emph{canonical equation}. In Subsection~\ref{sec:ES}, we come back to
the general case. We show that (IIF) is satisfied on the time scale of
the canonical equation and that evolutionary branching can only occur
on a longer time scale.  We are able to characterize the points,
called ``evolutionary singularities'', in the neighborhood of which
evolutionary branching may occur. In Subsection~\ref{sec:thm-br}, we
state and prove our main result of this section, giving a criterion
for evolutionary branching in the limit of small mutational jumps. We
thus rigorously prove a criterion stated with a heuristic
justification in~\cite{metz-geritz-al-96}.

\bigskip
\noindent Let us firstly  introduce  the following  additional
technical  Assumptions~(A'):
\begin{description}
\item[\textmd{(A'1)}] The trait space ${\cal X}$ is convex. It is
often  implicitly assumed for biological models with continuous
trait space.

\item[\textmd{(A'2)}] The distribution $m(x,h)dh$ has finite and
  bounded (in $x$) third-order moments.
\item[\textmd{(A'3)}] The map $x\mapsto m(x,h)dh$ is Lipschitz
  continuous from $ {\cal X}$ to the set of probability measures
  ${\cal P}(\mathbb{R}^l)$, for the Wasserstein metric
  \begin{multline*}
    \rho(P_1,P_2) = \inf\Big\{\int_{
      {\mathbb{R}^l}\times  {\mathbb{R}^l}} |x-y|\ R(dx,dy);
    R \in {\cal P}(\mathbb{R}^l\times \mathbb{R}^l) \hbox{  with
      marginals } P_1 \hbox{ and } P_2\Big\}.
  \end{multline*}
\item[\textmd{(A'4)}] The function
  $$
  g(y;x)=p(x)\lambda(x)\bar{n}(x) {f(y;x)\over \lambda(y)}
  $$
  is continuous on $ {\cal X}^2$ and of class ${\cal C}^1$ with respect to
  its first coordinate, where $f(\cdot;\cdot)$ is defined
  in~(\ref{eq:fitness-d=1}). Since ${\cal X}$ is a compact set of
  $\RR^l$, there exists a constant $\gamma > 0$ such that $\forall
  x,y\in {\cal X}$,
  $
  \ [g(y;x)]_+\leq \gamma.
  $
\end{description}

\bigskip
\noindent
 Later in this section, we will also need Assumption~(A''):
\begin{description}
\item[\textmd{(A'')}] The functions $\lambda(x)$ and $\mu(x)$ are
  ${\cal C}^3$ on ${\cal X}$ and the function $\alpha(x,y)$ is ${\cal C}^4$ on
  ${\cal X}^2$.
\end{description}

\medskip
\noindent Note that~(A'') implies~(A'4).

\bigskip
\noindent Finally, let us introduce the parameter
$\varepsilon\in(0,1]$ scaling the size of mutation. Since ${\cal
X}$ is convex, $x+\varepsilon Y\in{\cal X}$ a.s.\ for all
$x\in{\cal X}$ and $0\leq\varepsilon\leq 1$, where $Y$ is
distributed following $m(x,h)dh$. Therefore, it is possible to
define a PES in which mutational jumps are scaled by the parameter
$\varepsilon$, by replacing in its generator~(\ref{eq:gene-PES})
$m(x_j,h)dh$ by $m(x_j,h)dh\circ H^{-1}_{\varepsilon}$ for all
$j\in\{1,\ldots,d\}$, where $H_\varepsilon(h)=\varepsilon h$.
Under Assumptions~(A) and~(B), we define this way a ``rescaled
PES'' $(Z^\varepsilon_t,t\geq 0)$. If only Assumptions~(A) and~(C)
are satisfied, we do similar changes in~(\ref{eq:gene-mod-PES}) to
obtain a ``rescaled killed PES'' $(Z^{(2),\varepsilon}_t,t\geq
0)$. Finally, we do a time scaling of order $1/\varepsilon^2$ to
obtain the rescaled PES
$$
\tilde{Z}^{\varepsilon}_t=
\begin{cases}
  Z^\varepsilon_{t/\varepsilon^2} & \mbox{if Assumptions~(A) and~(B)
    are satisfied} \\
  Z^{(2),\varepsilon}_{t/\varepsilon^2} & \mbox{if only Assumptions~(A) and~(C)
    are satisfied.}
\end{cases}
$$
Since both $Z^\varepsilon_t$ and $Z^{(2),\varepsilon}_t$ agree as long as
there is no triple coexistence, and since we will only be interested
in the sequel to the cases where the PES is monomorphic or dimorphic,
we will not need to distinguish between these two cases.


\subsection{The TSS and the Canonical Equation of Adaptive Dynamics}
\label{sec:caneq}

Doing a similar time scaling as for $\tilde{Z}^\varepsilon$, we
can define for all $\varepsilon\in(0,1]$, the
$\varepsilon$-rescaled TSS $(X^\varepsilon_t,t\geq0)$ by modifying
the generator~(\ref{eq:gen-TSS}) as follows. For all ${\cal
  C}^1_b$-valued function $\varphi$,
\begin{equation}
  \label{eq:rescaled-generator}
  L^{\varepsilon}\varphi(x)=
  \frac{1}{\varepsilon^2}\int_{\mathbb{R}^l}(\varphi({x}+\varepsilon h)-
  \varphi({x}))[g(x+\varepsilon h;{x})]_+m(x,h)dh.
\end{equation}
From a mathematical point of view, the multiplicative term
$\varepsilon^{-2}$ takes into account that the integral term is of
order $\varepsilon^2$, because of $g(x;x)=0$ and Assumption~(A'4).

\medskip
\noindent Let us now state the convergence theorem of the rescaled
TSS to the canonical equation of adaptive dynamics. Its proof is
based on a standard uniqueness-compactness method.

\begin{thm}
  \label{thm:convergence}
  Assume~\textup{(A)} and \textup{(A')}. Suppose also that the family
  of initial states $\{X^{\varepsilon}_0\}_{0<\varepsilon\leq 1}$ is
  bounded in $\mathbb{L}^2$ and converges in law to a random variable
  $X_0$ as $\varepsilon\rightarrow 0$.

\medskip
\noindent
  Then  for each $T>0$, the sequence $(X^{\varepsilon})$ converges when $\varepsilon\rightarrow 0$,
  for the Skorohod topology of $\mathbb{D}([0,T],{\cal
    X})$, to the process
  $({x}(t),t\leq T)$ with initial state
  $X_0$ and with deterministic sample paths, unique solution
  of the ordinary differential equation, known as canonical equation of adaptive dynamics.
  \begin{equation}
    \label{eq:n-morphic-canonical-equation}
    \frac{dx}{dt}=\int_{\mathbb{R}^l}h
    [h\cdot\nabla_1g(x;{x})]_+\ m(x,h)dh.
  \end{equation}
\end{thm}

\begin{rem}
  \label{rem:sym-CEAD}
  In the case where $m(y,\cdot)$ is a symmetrical measure on
  $\mathbb{R}^l$ for all $y\in{\cal X}$,
  Equation~\textup{(\ref{eq:n-morphic-canonical-equation})} gets the
  classical form, heuristically introduced
  in~\textup{\cite{dieckmann-law-96}},
  \begin{equation}
    \label{eq:symmetrical-n-morphic-canonical-equation}
    \frac{dx}{dt}=\frac{1}{2}K(x)\nabla_1g(x;{x}),
  \end{equation}
  where $\ K(x)=\left(k_{ij}(x)\right)_{1\leq i,j\leq l}\ $ is the
  covariance matrix of $m(x,h)dh$.
\end{rem}

\begin{proof}

  \noindent (i) \textbf{Uniqueness of the solution of
  Equation \eqref{eq:n-morphic-canonical-equation} with given initial
  condition.}

\noindent
 Let us show that $\ a(x)= \int_{\mathbb{R}^l}h
  [h\cdot\nabla_1g(x;\mathbf{x})]_+\ m(x,h)dh$ is Lipschitz continuous
  on ${\cal X}$. We have
\begin{multline}
  \label{eq:Cauchy-Lipschitz}
    \|a({x})-a({x}')\|\leq\int_{\mathbb{R}^l}\|h\|
    \times\bigl|[h\cdot\nabla_1g(x;{x})]_+-
  [h\cdot\nabla_1g(x';{x}')]_+\bigr|\ m(x,h)dh \\ \quad
  +\left\|\int_{\mathbb{R}^l}h[h\cdot\nabla_1g(x';{x}')]_+
    (m(x,h)-m(x',h))dh\right\|.
\end{multline}

\medskip
\noindent Because of $|[a]_+-[b]_+|\leq|a-b|$, Assumptions  (A'2)
and (A'4), and that the support of all measures $m(x,h)dh$ is
included in a bounded set, the first term of the right hand side
of \eqref{eq:Cauchy-Lipschitz} is bounded by some constant times
$\|{x}-{x}'\|$.

\medskip
\noindent If we denote by $\xi$ the vector $\nabla_1g(x';{x'})$
and $\psi(h)=h[h\cdot\xi]_+$, then
\begin{equation*}
    \|\psi(h)-\psi(h')\|
    \leq\|(h-h')[h\cdot\xi]_+\|+\|h'([h\cdot\xi]_+-
    [h'\cdot\xi]_+)\|
    \leq 2 \|\xi\|\ \|h-h'\|\ (\|h\|+
    \|h'\|).
\end{equation*}
Thus, using the dual form of the Kantorovich-Rubinstein metric
(see Rachev~\cite{rachev-91}) and (A'3), one obtains that the
second term of the right-hand side of~(\ref{eq:Cauchy-Lipschitz})
is also bounded by some constant times $\|x-x'\|$. Hence
Cauchy-Lipschitz Theorem can be applied and $(\mathbf{x}(t),t\geq
0)$ is uniquely defined.

\medskip \noindent (ii) \textbf{The processes $X^\varepsilon,
\,\varepsilon>0$, with generator $L^\varepsilon$ can be
constructed on the same probability space.}

\noindent Recall the definition of $\gamma$ in Assumption~(A'4).

\begin{lem}
  \label{prop:construct-TSS}
  Assume \textup{(A)} and \textup{(A')}. Let $(\Omega,{\cal F},P)$ be
  a probability space and  $N(dh, d\theta, ds)$ be  a point Poisson measure
   on $ \mathbb{R}^l\times [0,1]\times
  \mathbb{R}_+ $ with intensity $\gamma \bar{m}(h) dh d\theta ds$.  Let
  $\varepsilon>0$ and denote by $N^\varepsilon$ the image measure of
   $N$ by the mapping $s \mapsto
  \varepsilon^2s$.  Let $X_0^\varepsilon$ be a ${\cal X}$-valued random
  variable, independent of $N$.  Then the process
  $X^\varepsilon$ defined by 
  \begin{equation}
    \label{def:proc-eps}
  X^\varepsilon_t=X^\varepsilon_0 + \int_{\mathbb{R}^l\times
    [0,1]\times [0,t]} (\varepsilon \, h)\ {\bf 1}_{\left\{\theta\leq
      {[g(X^\varepsilon_{s-}+\varepsilon h;X^\varepsilon_{s-})]_+\over
        \gamma}{m(X^\varepsilon_{s-},h)\over \bar{m}(h)}\right\}}
  N^\varepsilon(dh, d\theta, ds),
  \end{equation}
  is a jump Markov process with
  generator $L^\varepsilon$. Its law will be called
  $\mathbf{P}^{\varepsilon}_{X^{\varepsilon}_0}$.
\end{lem}

\noindent Indeed,  using Itô's formula, one observes that for a
bounded function $\varphi$ on ${\cal X}$, 
\begin{multline*}
\varphi(X^\varepsilon_{t})-\varphi(X_0^\varepsilon)\\
 \hskip 0.5cm -\int_0^t\int_{\mathbb{R}^l\times [0,1]}
\left(\varphi(X^\varepsilon_{\varepsilon^2 s}+\varepsilon
h)-\varphi(X^\varepsilon_{\varepsilon^2s})\right){\bf
1}_{\{\varepsilon^2 s\leq t\}}
g(X^\varepsilon_{\varepsilon^2s}+\varepsilon h;
X^\varepsilon_{\varepsilon^2s})m(X^\varepsilon_{\varepsilon^2
s},h)dh d\theta ds
\end{multline*}
 is a martingale, which implies the result.

\medskip \noindent (iii) \textbf{Tightness of the sequence of laws
$\{\mathbf{P}^{\varepsilon}_{X^{\varepsilon}_0}\}_{\varepsilon>0}$.}

\noindent  We will use the Aldous criterion~\cite{Aldous-78}. Let
$\tau$ be a stopping time less than $T$ and
$(\delta_{\varepsilon})$ positive numbers converging to $0$ when
$\varepsilon\rightarrow 0$.  We remark that $|g(x+\varepsilon h;
{x})|\leq\varepsilon C\|h\|$, by an expansion of $g$ with respect
to its first variable and the fact that $g(x; x)=0$,  and since
$\nabla_1 g$ is bounded by  a constant $C$. We have
\begin{equation*}
   \mathbf{E}(\|X^{\varepsilon}_{\tau+\delta_{\varepsilon}}
  -X^{\varepsilon}_{\tau}\|)
  =\mathbf{E}\left(
    \int^{\tau+\delta_{\varepsilon}}_{\tau}
    \int_{\mathbb{R}^l}\|\varepsilon h\|
    [g(X^{\varepsilon}_{s-}+\varepsilon
    h; X^{\varepsilon}_{s-})]_+\ m(X^{\varepsilon}_{s-},h)
    dh\frac{ds}{\varepsilon^2}\right)
   \leq CM_2\delta_{\varepsilon}, \label{eq:tightness}
\end{equation*}
where $M_2= \int \|h\|^2 \bar{m}(h) dh$.  Then, for any
$\alpha>0$,
\begin{equation*}
  \mathbf{P}(\|X^{\varepsilon}_{\tau_{\epsilon}+\delta_{\varepsilon}}-
  X^{\varepsilon}_{\tau_{\varepsilon}}\|>\alpha)
  \leq\frac{nCM_2}{\alpha}\delta_{\varepsilon}\rightarrow 0 \ \hbox{ when } \ \varepsilon\rightarrow 0.
\end{equation*}
 This gives the first part of the
Aldous criterion. For the second part, we have to prove the
uniform tightness of the laws of $(\sup_{t\leq
T}\|X^{\varepsilon}_t\|)_{\varepsilon>0}$. We use Itô's formula to
write $(X^\varepsilon_t)^2$ from \eqref{def:proc-eps}, Schwarz'
and Doob's inequalities and obtain that $\mathbf{E}(\sup_{t\leq T}
\|X^{\varepsilon}_t\|^2)\leq C_T
(\mathbf{E}(\|X^{\varepsilon}_0\|^2)+ 1)$, where $C_T$ is a
constant depending on time $T$, on $M_2$ and on an upper-bound of
$[g]_+$. Since  $(X_0^{\varepsilon})_{0<\varepsilon\leq 1}$ is
bounded in $\mathbb{L}^2$, the tightness of the laws of
$(\sup_{t\leq T}\|X^{\varepsilon}_t\|)_{\varepsilon>0}\ $ follows.

\medskip \noindent (iv) \textbf{Convergence of the generators.}

\noindent
Let us now prove that 
\begin{equation}
  \label{eq:convergence-generators}
  \forall\varphi\in{\cal C}^2_b({\cal X}),\
  {1\over \varepsilon^2} L^{\varepsilon}\varphi\rightarrow L^0\varphi\mbox{\ uniformly on\
  }{\cal X},
\end{equation}
where $L^{\varepsilon}$ is defined in~(\ref{eq:gen-TSS}) and
$L^0$ is defined by
\begin{equation*}
  L^0\varphi(\mathbf{x})=\int_{\mathbb{R}^l}
  (h\cdot\nabla \varphi({x}))
  [h\cdot\nabla_1g(x; {x})]_+\ m(x,h)dh,
\end{equation*}
 where $\nabla\varphi({x})$ is the
gradient vector of $\varphi({x})$. We have,
\begin{multline}
  \label{eq:bourrin}
  \left|{1\over \varepsilon^2}L^{\varepsilon}\varphi({x}) -
  L^0\varphi({x})\right|\leq  \int_{\mathbb{R}^l}\![h\cdot\nabla_1
  g(x; {x})]_+\!\times\!\left|
    \frac{\varphi({x}+\varepsilon h)-
      \varphi({x})}{\varepsilon}-h
    \cdot\nabla\varphi({x})\right|m(x,h)dh
  \\ +  \int_{\mathbb{R}^l}
  \left|\frac{\varphi({x}+\varepsilon
      h)-\varphi({x})}
    {\varepsilon}\right|\times \left|\left[\frac{g(x+\varepsilon
        h; {x})}{\varepsilon}\right]_+ -[h\cdot\nabla_1
    g(x; x){x})]_+\right|m(x,h)dh.
\end{multline}
 Let us call $I_1$ and $I_2$ the quantities inside the integral
in the  first and the second term, respectively. Now, $\varphi$ is
$\mathcal{C}^1$, $g(x; {x})=0$  and by~ Assumption (A'),
 $g(x; y)$ is $\mathcal{C}^1$ with
respect to the first variable $x$. So, we can find $\theta_1$,
$\theta_2$ and $\theta_3$ in $[0,1]$ depending on ${x}$ and  $h$
such that
\begin{align*}
  I_1 & =[h\cdot\nabla_1
  g(x; {x})]_+\times |h\cdot\nabla\varphi({x}
  +\theta_3\varepsilon
  h)-h\cdot\nabla\varphi({x})|;\\
  I_2 & =|h\cdot\nabla\varphi({x}+\theta_1\varepsilon h)|\times
  |[h\cdot\nabla_1 g(x+\theta_2\varepsilon h; {x})]_+-
  [h\cdot\nabla_1 g(x; {x})]_+|.
\end{align*}
Since $\varphi$ is in $\mathcal{C}^2_b$, and because of
Assumption~(A'), we can choose a number $C$ such that
$\nabla\varphi$ and $\nabla_1 g$ are both $C$-Lipschitz and
bounded by $C$ on ${\cal X}$ and ${\cal X}^{2}$ respectively. Then
\begin{align*}
   I_1&\leq C\| h\|\times\| h\| C\|\theta_3\varepsilon
    h\|\leq\varepsilon C^2\|h\|^3;\\
    I_2&\leq C\| h\|\times |h\cdot\nabla_1 g(x+\theta_2\varepsilon
    h,{x})-h\cdot\nabla_1
    g(x,{x})|\leq\varepsilon C^2\| h\|^3.
\end{align*}
 It remains to put these bounds in Equation~(\ref{eq:bourrin})
to obtain:
\begin{equation*}
  \left|{1\over \varepsilon^2}L^{\varepsilon}\varphi({x})-L^0\varphi({x})\right|
  \leq 2\varepsilon C^2\int_{\mathbb{R}^l}\|h\|^3m(x,h)dh.
\end{equation*}
We conclude using Assumption~(A'2).

\medskip \noindent (v) \textbf{Martingale problem for limiting distributions.}

\noindent Finally, let us show that any accumulation point
$\mathbf{P}$ of the family of laws
$\{\mathbf{P}^{\varepsilon}_{X^{\varepsilon}_0}\}$ on
$\mathbb{D}([0,T],{\cal X})$ is the law of the process ${X}$
solution to~(\ref{eq:n-morphic-canonical-equation}) with initial
state $X_0$. Fix such a $\mathbf{P}$. Let us endow the space
$\mathbb{D}([0,T],{\cal X})$ with the canonical filtration ${\cal
  G}_t$, and for any $\varphi\in{\cal
  C}^2({\cal X})$, let us define on this space the
processes 
\begin{align*}
  M^{\varphi}_t(w) & =\varphi(w_t)-\varphi(w_0)-
  \int_0^t L^0\varphi(w_s)ds \\
   M^{\varepsilon,\varphi}_t(w) & =
  \varphi(w_t)-\varphi(w_0)-\int_0^t {1\over \varepsilon^2}
  L^{\varepsilon}\varphi(w_s)ds.
\end{align*}
 We will show that $M^{\varphi}=0$ $\mathbf{P}$-a.s. Fix
$\varphi\in{\cal C}^2({\cal X})$. It is standard, using It\^o
formula for jump processes, to show that, under
$\mathbf{P}^{\varepsilon}_{X^{\varepsilon}_0}$,
$M^{\varepsilon,\varphi}$ is a square-integrable ${\cal
  G}_t$-martingale and that
\begin{multline*}
    M^{\varepsilon,\varphi}_t(X^{\varepsilon})=
    \int_{\mathbb{R}^l\times[0,1]\times[0,t]}
    (\varphi(X^{\varepsilon}_s+\varepsilon
    h)-\varphi(X^{\varepsilon}_s))\\ \mathbf{1}_{\left\{
        \theta\leq\frac{[g(X^{\varepsilon}_{s-}+\varepsilon
          h,X^{\varepsilon}_{s-})]_+}{\gamma}
        \frac{m(X^{\varepsilon}_{s-},h)}{\bar{m}(h)}\right\} }
    \tilde{N}^\varepsilon\left(dh,d\theta,{ds}\right)
\end{multline*}
where $\tilde{N}^\varepsilon=N^\varepsilon-q^\varepsilon$ is the
compensated Poisson measure associated with $N^\varepsilon$, and
$q^\varepsilon(dh, d\theta, ds)$ is the image measure of $\gamma
\bar{m}(h) dh d\theta ds$ by $s \mapsto \varepsilon^2 s$. Thus,
using computation similar to~(\ref{eq:tightness}),
\begin{align}
   \mathbf{E}^\varepsilon(\langle M^{\varepsilon,\varphi}\rangle_t ) & =\frac{1}{\varepsilon^2}
  \mathbf{E}^\varepsilon\left( \int_0^t\int_{\mathbb{R}^l}(\varphi(X^{\varepsilon}_s+\varepsilon
  h)-\varphi(X^{\varepsilon}_s))^2
[g(X^{\varepsilon}_s+\varepsilon
  h,X^{\varepsilon}_s)]_+\ m(X^{\varepsilon}_s,h)dhds\right) \notag \\
  & \leq CC'M_3\,t\,\varepsilon, \label{eq:quadr-var}
\end{align}
where $\mathbf{E}^\varepsilon$ denotes the expectation under
$\mathbf{P}^{\varepsilon}_{X^{\varepsilon}_0}$,  $C'$ is a bound
for $\nabla\varphi$, and $M_3$  a bound of the third-order moment
of $m(y,h)dh$. Using~(\ref{eq:quadr-var}) and the fact that
$M^{\varphi}_t(w)=M^{\varepsilon,\varphi}_t(w) +\int_0^t({1\over
\varepsilon^2}L^{\varepsilon}\varphi(w_s)-L^0\varphi(w_s))ds$, it
follows that
\begin{equation*}
  \mathbf{E^\varepsilon}(|M_t^{\varphi}|^2)\leq
  2t^2\|{1\over
\varepsilon^2}\ L^{\varepsilon}\varphi-L^0\varphi\|^2_{\infty}
  +2C^2C'^2M_3^2t^2\varepsilon^2
\end{equation*}
which converges to $0$ when $\varepsilon\rightarrow 0$ thanks
to~(\ref{eq:convergence-generators}). Moreover by
\eqref{def:proc-eps}, we have that almost surely, $\ \sup_{t\leq
T} \|X^\varepsilon_t-X^\varepsilon_{t-}\|\leq C''\varepsilon$,
which implies that each limit process $X$ with law $\mathbf{P}$ is
almost surely continuous. Hence, for any $t\in [0,T]$, the
functional $\omega \mapsto \varphi(w_t)-\varphi(w_0)-
  \int_0^t L^0\varphi(w_s)ds$
 is continuous at $X$ for the weak
topology and since $\mathbf{P}$ is the weak limit of an extracted
sequence of $(\mathbf{P}^{\varepsilon}_{X^{\varepsilon}_0})$, it
follows that, under $\mathbf{P}$, $M^{\varphi}(w)=0$ a.s, which
concludes the proof.\end{proof}

\subsection{PES and Evolutionary Singularities}
\label{sec:ES}

\textbf{Until the end of Section~\ref{sec:br}, we will assume for
  simplicity that the trait space is one-dimensional ($l=1$),
  i.e. ${\cal X}\subset\RR$.}

\medskip
\noindent We have proved in the last subsection that, when
$\varepsilon \to 0$, the TSS is very close to the solution of the
canonical equation \eqref{eq:n-morphic-canonical-equation} on any
time interval $[0,T]$. The equilibria of this equation are given
by the points $x^*$ such that either $\partial_1 g(x^*;x^*)=0$, or
$\int_{\mathbb{R}_+}m(x^*,h)dh=0$ and $\partial_1g(x^*;x^*)>0$, or
$\int_{\mathbb{R}_-}m(x^*,h)dh=0$ and $\partial_1 g(x^*;x^*)<0$. We will
concentrate on the points such that $\partial_1 g(x^*;x^*)=0$, or
equivalently, $\partial_1 f(x^*;x^*)=0$, since
$$
\partial_1 g(x;x)= {1\over \lambda(x)} \partial_1 f(x;x)
p(x) \lambda(x) \bar{n}(x) = p(x)  \bar{n}(x) \partial_1 f(x;x).
$$
Remark that, since $f(x;x)=0$ for all $x\in {\cal X}$,
\begin{align}
  & \partial_1 f(x;x)+\partial_2 f(x;x)=0,\quad\forall x\in{\cal X}
  \label{eq:der-1-fitn} \\
  & \partial_{11} f(x;x)+2\partial_{12}
  f(x;x)+\partial_{22}f(x;x)=0,\quad\forall x\in{\cal X}. \label{eq:der-2-fitn}
\end{align}
 Therefore, $\partial_1f(x^*;x^*)=\partial_2 f(x^*;x^*)=0$.

\medskip
\noindent \begin{defi} \label{es} Points $x^*$ such that $\
\partial_1
  g(x^*; x^*)=0$, or equivalently, $\partial_1
  f(x^*;x^*)=\partial_2f(x^*; x^*)=0$ are called evolutionary
  singularities (ES).
\end{defi}

\begin{lem}
\label{lem:I_T}
Assume~\textup{(A)},~\textup{(A')} and~\textup{(A'')}.
\begin{description}
\item[\textmd{(1)}] The solution $x(t)$ of
\eqref{eq:n-morphic-canonical-equation} starting from a point that
is not an ES cannot attain an ES in finite time.
\item[\textmd{(2)}] Assume that $x(0)$ is not an ES and let $I_T =
  \{x(t), t\in [0,T]\}$. Then, for any sufficiently small $\eta>0$,
  for any $x$ at a distance to $I_T$ smaller than $\eta$ and for any
  $y$ sufficiently close to $x$, $x$ and $y$ satisfy~(IIF) and
  $(y-x)f(y;x)$ has constant sign.
\end{description}
\end{lem}

\begin{proof}
(1) Let $c$ be a constant such that $x\mapsto\int_{\RR}
h[h\cdot\partial_1 g(x;x)]_+m(x,h)dh$ is $c$-Lipschitz (the fact
that this is a Lipschitz function is shown in the proof of Theorem
\ref{thm:convergence}). Then, for any ES $x^*$,
$$
\Big|\frac{d}{dt}(x(t)-x^*)^2\Big|\leq 2\ |\dot{x}(t)|\
|x(t)-x^*|\leq 2\ c\ (x(t)-x^*)^2.
$$
Thus, $|x(t)-x^*|\geq |x(0)-x^*|\exp(-ct)>0$.

\medskip
\noindent  (2) Remark first that, from Point (1), $ C=\inf_{x\in
I_T}|\partial_1 f(x(t);
  x(t))|>0.
$
Therefore, for $\eta>0$ sufficiently small, $\{x\in{\cal X}:
\mbox{dist}(x,I_T)\leq\eta\}\subset\{x\in{\cal X}:|\partial_1 f(x(t);
  x(t))|>C/2\}$. Fix such an $\eta$.

\noindent Let us now consider some point $x$ in ${\cal X}$ such
that $\partial_1 f(x; x)>C/2$. Consider first $y$ in ${\cal X}$
such that $x<y$. Using that $f(x; x)=0$ and~(\ref{eq:der-1-fitn}),
a second-order expansion of $f(y; x)$ at $(x,x)$ implies that
$f(y; x)> C(y-x)/4$ provided that $|y-x|<{C\over 2C'}$, where
$C'>0$ is a constant uniformly upper-bounding the second-order
derivatives of $f(\cdot; \cdot)$ on the compact set ${\cal X}^2$.
Under the same condition, $f(x; y)<C(x-y)/4$. Therefore,
$f(x;y)f(y;x)<0$ if $|y-x|$ is small enough and $(y-x)f(y;x)$ has
constant sign. This reasoning gives the same conclusion if $y<x$
or $\partial_1 f(x; x)<-C/2$, giving the required result.
\end{proof}

\bigskip \noindent  Now we come back to the rescaled PES
$(\tilde{Z}^\varepsilon_t,t\geq 0)$ defined in the beginning of this
section and assume that its initial condition $\tilde{Z}^\varepsilon_0$ is
monomorphic.  We want to determine when evolutionary branching can
occur in this process. This requires that (IIF) (ensuring
non coexistence) fails. For $\varepsilon>0$, we define the first
coexistence time
$$
\tau^\varepsilon =\inf\{t>0,\  f(\tilde{Z}^\varepsilon_t,;
\tilde{Z}^\varepsilon_{t-})>0 \hbox{ and }
f(\tilde{Z}^\varepsilon_{t-}; \tilde{Z}^\varepsilon_t) >0\},
$$
and for any
  $\eta>0$, the entrance time of the process in a
$\eta$-neighborhood of an ES $x^*$,
  \begin{equation}
    \label{eq:def-theta}
    \theta^\varepsilon_{\eta}=\inf\{t\geq 0,\
    \mbox{\textup{Supp}}(\tilde{Z}^\varepsilon_t)\cap(x^*-\eta,x^*+\eta)
    \not=\emptyset\}.
  \end{equation}

\begin{thm}
  \label{thm:pes=tss}
  Assume that~\textup{(A)},~\textup{(A')},~\textup{(A'')}
  and~\textup{(B)} or~\textup{(C)} hold. Assume also that
  $\tilde{Z}_0^\varepsilon=\bar{n}(x)\delta_x$ where $x\in{\cal X}$ is
  \emph{not} an evolutionary singularity. Then,
\begin{description}
\item[\textmd{(i)}]
 For any $T>0$,
  $$
  \ \lim_{\varepsilon\rightarrow 0}\mathbb{P}(\tau^\varepsilon>T)=1.
  $$
  Moreover, for all $\eta>0$,
  $$
  \lim_{\varepsilon\rightarrow 0}\mathbb{P}(\forall t \in
  [0,T],\ \mbox{\textup{Card}}
  (\mbox{\textup{Supp}}(\tilde{Z}^\varepsilon_{t})) =1,\
  \|\mbox{\textup{Supp}}(\tilde{Z}^\varepsilon_{t})- x(t)\|\leq \eta)=1.
  $$
\item[\textmd{(ii)}] For any $\eta>0$, there exists
$\varepsilon_0>0$
  such that, for all $\varepsilon<\varepsilon_0$,
  \begin{gather}
    \PP(\theta^\varepsilon_{\eta}<\tau^\varepsilon) = 1 \quad
    \mbox{and}\notag \\
    \label{eq:pes=tss-ps}
    \PP(\forall t\in[0,\theta^\varepsilon_{\eta}],\
    \mbox{\textup{Supp}}(\tilde{Z}^\varepsilon_{t})=\{Y^{\varepsilon}_t\}\
    \mbox{with}\ t\mapsto Y^\varepsilon_t\ \mbox{monotonous on}\
    [0,\theta^\varepsilon_{\eta}])=1.
  \end{gather}
  \end{description}
\end{thm}

\begin{proof}
  (i) Before the stopping time $\tau^\varepsilon$, and  since
  the initial condition is monomorphic, it is clear that the support of
  $\tilde{Z}^\varepsilon_{t}$ is a singleton whose dynamics is that of
  the rescaled TSS $(X^\varepsilon_t,t\geq 0)$.
  Because of Theorem~\ref{thm:convergence}, when $\varepsilon \to 0$,
  the TSS is close to the canonical equation.  In particular, for all
  $\eta>0$, its values on the time interval $[0,T]$ belong to the set
  $\{x\in{\cal X}:\mbox{dist}(x,I_T)\leq\eta\}$ with probability
  converging to 1. Moreover, since ${\cal X}$ is compact,
  $\mbox{Supp}(m(x,\cdot))\subset{\cal X}-x$ is included in the closed
  ball of $\mathbb{R}^l$ centered at 0 with diameter
  $2\mbox{diam}({\cal X})$. Therefore, the distance between a mutant
  trait and the trait of its progenitor in the rescaled PES
  $\tilde{Z}^\varepsilon$ is a.s.\ less that
  $\varepsilon c$, where $c$ is a constant. Hence, the result immediately follows from
  Lemma~\ref{lem:I_T}.

(ii) We also deduce from this lemma that for any $T>0$ such that
$I_T\cap(x^*-3\eta/2,x^*+3\eta/2)=\emptyset$,
  $
  \lim_{\varepsilon\rightarrow 0}\PP(\theta^\varepsilon_{\eta}>T)=1
  $. Moreover, the process $Y_t^\varepsilon$ in~(\ref{eq:pes=tss-ps}), which is exactly the
TSS of the previous section, is \emph{almost surely} monotonous
before time $\theta^\varepsilon_\eta$.
\end{proof}

\begin{rem} Theorem~\ref{thm:pes=tss} implies that, when the
initial population is monomorphic and away from evolutionary
singularities, evolutionary branching can only occur in the
neighborhood of an evolutionary singularity and on a longer time
scale than $T/\varepsilon^2$ when $\varepsilon\rightarrow 0$, for
all $T>0$.
\end{rem}

\medskip \noindent The next result shows that we can restrict to ES
that are not repulsive for the canonical equation.

\begin{prop}
  \label{prop:attract}
  Under the assumptions of Theorem~\textup{\ref{thm:pes=tss}},
  coexistence of two traits can only occur in the neighborhood of
  evolutionary singularities $x^*\in{\cal X}$ which are not
  repulsive, i.e. which
satisfy
\begin{equation}
    \label{eq:attract}
    \partial_{22}f(x^*;x^*)\geq\partial_{11}f(x^*;x^*).
  \end{equation}
  More precisely, for any neighborhood ${\cal U}$ of the set of
  evolutionary singularities satisfying~\textup{(\ref{eq:attract})},
  for all $\varepsilon$ small enough,
  $$
  \PP(\tau^\varepsilon<+\infty\mbox{\ and\
  }\mbox{\textup{Supp}}(Z^\varepsilon_{\tau^\varepsilon-})\not\in{\cal
    U})=0.
  $$
\end{prop}

\begin{proof} Let us remark that an ES such that
\begin{equation}
\label{eq:repuls}
\partial_{11}f(x^*;x^*)+\partial_{12}f(x^*;x^*) >0.
\end{equation}
is always a repulsive point for the canonical equation, in the
sense that, for any solution $x(t)$ of the canonical equation
starting sufficiently close from $x^*$, the distance between
$x(t)$ and $x^*$ is non-decreasing in the neighborhood of time 0.
In other words, there exists a neighborhood ${\cal U}$ of $x^*$
such that no solution of the canonical equation starting out of
${\cal U}$ can enter ${\cal U}$. To this end, we remark that
\eqref{eq:repuls} implies that there exists $\eta_{x^*}$ with
\begin{itemize}
\item $\partial_1 g(x; x)>0\ $ if $\ x\in (x^*,x^*+\eta_{x^*}]$,
\item $\partial_1 g(x; x)<0\ $ if $\ x\in [x^*-\eta_{x^*}, x^*)$,
\end{itemize}
and conclude in view of \eqref{eq:n-morphic-canonical-equation}.

\noindent Observe that,
by~(\ref{eq:der-2-fitn}),~(\ref{eq:repuls}) is equivalent to
$\partial_{11}f(x^*;x^*)-\partial_{22}f(x^*;x^*)>0$.

\noindent Let $S$ be the set of repulsive ES and define $ {\cal
  V}=\cup_{x^*\in S} (x^*-\eta_{x^*},x^*-\eta_{x^*})$. Fix ${\cal
  U}$ as in the statement of Proposition~\ref{prop:attract} and assume
(without loss of generality) that ${\cal U}\cap{\cal V}=\emptyset$ and
$x\not\in{\cal U}\cup{\cal V}$. Let $[a,b]$ be any connected component
of ${\cal X}\setminus({\cal U}\cup{\cal V})$. Since $\partial_1
f(y,y)\not=0$ for all $y\in[a,b]$, reproducing the argument of the
proof of Theorem~\ref{thm:pes=tss} easily shows that coexistence never
happens in a monomorphic population with trait in ${\cal
  X}\setminus({\cal U}\cup{\cal V})$ if $\varepsilon$ is sufficiently
small. Similarly, for $\varepsilon$ sufficiently small, no mutant in
${\cal V}$ born from a monomorphic population with trait not belonging
to ${\cal V}$ has a positive fitness. Therefore, the TSS cannot drive
the population inside ${\cal V}$ starting from outside. Thus
Proposition~\ref{prop:attract} is clear.
\end{proof}

\subsection{Evolutionary branching criterion}
\label{sec:thm-br}

In this section we will prove a criterion of evolutionary
branching. We need the following last assumption.

\begin{description}
\item[\textmd{(A''')}] For any $x$ in the interior of ${\cal X}$,
  $\int_{\mathbb{R}_-}m(x,h)dh>0$ and $\int_{\mathbb{R}_+}m(x,h)dh>0$.
\end{description}

\subsubsection{Definition and main result}
 We first need to precisely define what we mean by
evolutionary branching.

\begin{defi}
  \label{def:br}
  Let $x^*$ be an ES. For all $\eta>0$,
  we call $\eta$-branching the event
  \begin{itemize}
  \item there exists $t_1>0$ such that the support of the PES at time $t_1$  is
    composed of a single point belonging to $[x^*-\eta,x^*+\eta]$
  \item there exists $t_2>t_1$ such that the support of the PES at time $t_2$ is
    composed of exactly 2 points distant of more than $\eta/2$
  \item between $t_1$ and $t_2$, the support of the PES is always a
    subset of $[x^*-\eta,x^*+\eta]$, and is always composed of at most 2
    traits, and has increasing (in time) diameter.
  \end{itemize}
\end{defi}
We only consider \emph{binary} evolutionary branching. We will
actually prove that the simultaneous subdivision of a single
branch into three branches (or more) is a.s.\ impossible. Note
that this notion of evolutionary branching requires the
coexistence of two traits, but also that these two traits diverge
from one another.

\bigskip
\noindent Our main result is the following.
\begin{thm}
  \label{thm:br}
  Assume~\textup{(A)},~\textup{(A')},~\textup{(A'')}, ~\textup{(A''')}  and
  either~\textup{(B)} or~\textup{(C)}. Assume also that
  $Z^\varepsilon_0=\bar{n}(x)\delta_x$ and that the canonical equation
  with initial
  condition $x$ converges to an ES $x^*$ in the interior of ${\cal X}$
  such that
  \begin{align}
    & \partial_{22}f(x^*;x^*)>\partial_{11}f(x^*;x^*)
    \label{eq:cond1} \\
    \mbox{and}\quad & \partial_{22}f(x^*;x^*)+\partial_{11}f(x^*;x^*)\not=0.
    \label{eq:cond2}
  \end{align}
  Then, for all sufficiently small $\eta$, there exists
  $\varepsilon_0>0$ such that for all $\varepsilon<\varepsilon_0$,
  \begin{description}
  \item[\textmd{(a)}] if $\partial_{11}f(x^*;x^*)>0$,
    $\PP^\varepsilon(\eta\mbox{-branching})=1$.
  \item[\textmd{(b)}] if $\partial_{11}f(x^*;x^*)<0$,
    $\PP^\varepsilon(\eta\mbox{-branching})=0$. Moreover,
    $$
    \PP^\varepsilon\big(\forall t\geq\theta_\eta^\varepsilon,\
    \mbox{\textup{Card}}
    (\mbox{\textup{Supp}}(\tilde{Z}^\varepsilon_{t}))\leq 2\
    \mbox{and}\
    \mbox{\textup{Supp}}(\tilde{Z}^\varepsilon_{t})
    \subset(x^*-\eta,x^*+\eta)\:\big)=1,
    $$
    where $\theta_\eta^\varepsilon$ has been defined
    in~\textup{(\ref{eq:def-theta})}.
  \end{description}
\end{thm}
This criterion appeared for the first time
in~\cite{metz-geritz-al-96} with an heuristic justification. We
see that, locally around $x^*$, one of the two following events
can occur almost surely: either there is binary evolutionary
branching and the two branches diverge monotonously, or there is
no evolutionary branching, and the population stays forever inside
any neighborhood of $x^*$. Coexistence can occur in this case, but
cannot drive the support of the population away from a small
neighborhood of $x^*$. We will actually prove that, in this case,
as soon as there is coexistence of two traits in the population,
 the diameter of the support of the PES can only decrease
until it reaches 0 (i.e.\ until the next time when the population
becomes monomorphic).

\medskip
\noindent We give in the following subsections a full proof of
this result. In Section~\ref{sec:regul}, we will prove regularity
results on the 2- and 3-dimensional fitness functions and give
their second order expansions in the neighborhood of evolutionary
singularities. A first corollary of this result is given in
Section~\ref{sec:BEB} where, using the results of M.-L.
Zeeman~\cite{zeeman-93} and Fig.~\ref{fig:10}, we will show that
no triple coexistence can occur in the neighborhood of
evolutionary singularities. Finally, a case by case study of the
zone of coexistence and of the signs of fitness functions in the
neighborhood of an evolutionary singularity will allow us to
conclude the proof in Section~\ref{sec:coex}.

\medskip \noindent
Before coming to the proof and in order to illustrate the
difference between coexistence and evolutionary branching, we
state a result that will be needed in the course of the proof of
Theorem~\ref{thm:br}. Its proof will be given in Subsection 4.3.4.
We recall that two traits $x$ and $y$ coexist if and only if
$f(x;y)>0$ and $f(y;x)>0$.

\begin{prop}
  \label{prop:coex}
  Assume~(A) and that $\lambda$, $\mu$ and $\alpha$ are ${\cal
    C}^2$. Let $x^*\in{\cal X}$ be any ES.
  \begin{description}
  \item[\textmd{(a)}] If $\partial_{11}f(x^*;x^*)+\partial_{22}f(x^*;x^*)>0$,
   then for all neighborhood ${\cal U}$
    of $x^*$, there exist $x,y\in{\cal U}$ that coexist.
  \item[\textmd{(b)}] If $\partial_{11}f(x^*;x^*)+\partial_{22}f(x^*;x^*)<0$,
  then there exists a neighborhood
    ${\cal U}$ of $x^*$ such that any $x,y\in{\cal U}$ do not coexist.
  \end{description}
\end{prop}

\noindent
This shows that the criterion of evolutionary branching
($\partial_{11}f(x^*;x^*)>0$) is different from the criterion of
coexistence ($\partial_{11}f(x^*;x^*)+\partial_{22}f(x^*;x^*)>0$).
In particular, if one assumes as in Theorem~\ref{thm:br} that
$\partial_{22}f(x^*;x^*)>\partial_{11}f(x^*;x^*)$, the
evolutionary branching condition $\partial_{11}f(x^*;x^*)>0$
implies the coexistence criterion
$\partial_{11}f(x^*;x^*)+\partial_{22}f(x^*;x^*)>0$, as expected.

\subsubsection{Example}
Let us  come back to the example introduced in Subsection 2.2.

\medskip \noindent The fitness function is
\begin{align*} f(y;x)&= \lambda(y)-\alpha(y,x)\bar{n}(x) \\
&=  \exp\Big(-{y^2\over 2\sigma^2_b}\Big) -
\exp\Big(-{(x-y)^2\over 2\sigma^2_\alpha}\Big)\exp\Big(-{x^2\over
2\sigma^2_b}\Big).\end{align*}
 Computation gives  
\begin{equation*}
  \partial_1f(x^*;x^*)= -{x^*\over \sigma^2_b} \exp\Big(-{(x^*)^2\over
    2\sigma^2_b}\Big) =0 \Longleftrightarrow x^* = 0.
\end{equation*}
Moreover, $\
\partial_{11}f(0;0)= {1\over \sigma^2_\alpha}-{1\over
\sigma^2_b}\ $ and
 $\ \partial_{22}f(0;0)= {1\over \sigma^2_\alpha} + {1\over
\sigma^2_b}.$  Thus, the coexistence criterion of Proposition
\ref{prop:coex} (a) is always satisfied. We furthermore observe
that \eqref{eq:cond1} and \eqref{eq:cond2} hold, and that
$$
\partial_{11}f(0;0)>0 \Longleftrightarrow \sigma_\alpha <
\sigma_b.
$$
Then if $\sigma_\alpha < \sigma_b$, we have almost
surely branching and when $\sigma_\alpha > \sigma_b$, we have
only coexistence. This is consistent with Fig.~\ref{fig:ex2}
(a) and (b).

\subsubsection{Trait smoothness of fitnesses around evolutionary
  singularities}
\label{sec:regul}

The problem of local expansion of fitness functions has been already studied
in~\cite{durinx-metz-al-08} for general models. In this section, we
establish regularity and expansion results on our 2- and 3-dimensional
fitness functions in the neighborhood of evolutionary singularities.
To this aim, we need the following lemma.

\begin{lem}
  \label{lem:regul}
  Let $h(x,y,z)$ be a ${\cal C}^k$ function for $k\geq 1$ defined on
  ${\cal X}^3$ such that $h(x,x,z)=0$ for all $x,z\in{\cal X}$. Then,
  the function
  \begin{equation*}
    (x,y,z)\mapsto\frac{h(x,y,z)}{x-y}
  \end{equation*}
  can be extended on $\{x=y\}$ as a ${\cal C}^{k-1}$ function
  $\hat{h}(x,y,z)$ on ${\cal X}^3$ by setting
  $\hat{h}(x,x,z)=\partial_1 h(x,x,z)$ for all $x,z\in{\cal X}$.
\end{lem}

\begin{proof}
  Taylor's formula with integral remainder yields
  $$
  \frac{h(x,y,z)}{x-y}=\int_0^1\partial_1 h(y+(x-y)u,y,z)du
  $$
  for all $x\not=y$. The right-hand side also has a sense for $x=y$
  and defines a ${\cal C}^{k-1}$ function on ${\cal X}^3$.
\end{proof}

\bigskip
\noindent Let $x^*\in{\cal X}$ be an ES as in the statement of
Theorem~\ref{thm:br}. By Assumptions~(A) and~(A''), the
2-dimensional fitness function $f(y;x)$ defined
in~(\ref{eq:fitness-d=1}) is well-defined for all $x,y\in{\cal X}$
and is a ${\cal C}^3$ function. We extend the definition of  the
3-dimensional fitness function
$$
f(z;x,y)=r(z)-\alpha(z,x)\bar{n}_1(x,y)-\alpha(z,y)\bar{n}_2(x,y),
$$
where $\bar{n}_i(x,y)$, $i=1,2$, are defined in~(\ref{n1comp})
and~(\ref{n2comp}) to all $x,y\in{\cal X}$  such that
$$
\alpha(x,x)\alpha(y,y)-\alpha(x,y)\alpha(y,x)\neq 0.
$$

\noindent We will also use the notation
\begin{equation}
  \label{eq:def-a-c}
  a=\partial_{11}f(x^*;x^*)\quad\mbox{and}\quad c=\partial_{22}f(x^*;x^*).
\end{equation}
Note that, by \eqref{eq:der-2-fitn},
\begin{equation}
  \label{eq:pf-regul-0}
  \partial_{12}f(x^*;x^*)=-\frac{a+c}{2}.
\end{equation}

\begin{prop}
  \label{prop:regul}
  Under the assumptions of Theorem~\textup{\ref{thm:br}},
  the following properties hold.
  \begin{description}
  \item[\textmd{(i)}] For all $x,y\in{\cal X}$ in a neighborhood of
    $x^*$,
    $$
    x\not= y\quad\Longrightarrow\quad
    \alpha(x,x)\alpha(y,y)\not=\alpha(x,y)\alpha(y,x).
    $$
    This implies in particular that $\bar{\mathbf{n}}(x,y)$ and
    $f(\cdot;x,y)$ are well-defined for such $x,y$.
  \item[\textmd{(ii)}]  When $x,y\rightarrow x^*$ in
    such a way that $x\not=y$, and  for all $z\in{\cal X}$,
    \begin{align}
     & \bar{n}_1(x,y)+\bar{n}_2(x,y)\ \longrightarrow\
      \bar{n}(x^*)=\frac{r(x^*)}{\alpha(x^*,x^*)}
      ;\label{eq:lim-barn}\\
     & f(z;x,y)\ \longrightarrow\ f(z;x^*). \label{eq:lim-f3}
    \end{align}
  \item[\textmd{(iii)}] With the notation~\textup{(\ref{eq:def-a-c})},
    as $x,y\rightarrow x^*$,
    \begin{equation}
      \label{eq:dl-f2}
      f(y; x)=\frac{1}{2}(x-y)\big(c(x-x^*)-a(y-x^*)\big)
      +o\big(|x-y|\,(|x-x^*|+|y-x^*|)\big).
    \end{equation}
  \item[\textmd{(iv)}] The function $f(z;x,y)$ can be extended as a
    ${\cal C}^2$ function on $\{(x,y,z):z\in{\cal X},x,y\in{\cal U}\}$ for
    some neighborhood ${\cal U}$ of $x^*$ in ${\cal X}$. Still
    denoting by $f(z;x,y)$ the extended function, as $x,y\rightarrow
    x^*$,
    \begin{equation}
      \label{eq:dl-f3}
      f(z;x,y)=\frac{a}{2}(z-x)(z-y)+o\big(|z-x|\,|z-y|\big).
    \end{equation}
  \end{description}
\end{prop}

\begin{proof}
  Let $D(x,y):=\alpha(x,x)\alpha(y,y)-\alpha(x,y)\alpha(y,x)$. It
  follows from Lemma~\ref{lem:regul} that $D(x,y)/(x-y)$ can be
  extended on ${\cal X}^2$ as a ${\cal C}^3$ function, which has value
  $$
  \partial_1\alpha(x,x)\alpha(x,x)+\partial_2\alpha(x,x)\alpha(x,x)
  -\partial_1\alpha(x,x)\alpha(x,x)-\alpha(x,x)\partial_2\alpha(x,x)=0
  $$
  at the point $(x,x)$. Therefore, Lemma~\ref{lem:regul} can be
  applied once more to prove that $D(x,y)/(x-y)^2$ can be extended as
  a ${\cal C}^2$ function $\hat{D}(x,y)$ on ${\cal X}^2$. Hence, an
  elementary computation involving the second-order Taylor expansion
  of $D(x,y)$ yields that
  $$
  D(x,y)=(x-y)^2\big( \alpha(x^*,x^*) \partial_{12} \alpha(x^*,x^*)
  -\partial_1 \alpha(x^*,x^*) \partial_2 \alpha(x^*,x^*)\big) + o(|x-y|^2).
  $$
  Thus, Point~(i) follows from the fact that
  $\alpha(x^*,x^*)\partial_{12}\alpha(x^*,x^*)
  \not=\partial_1\alpha(x^*,x^*)\partial_2\alpha(x^*,x^*)$, which is a
  consequence of~(\ref{eq:cond2}). Indeed, an elementary computation
  shows that
  \begin{gather*}
    a=r''(x^*)
    -r(x^*)\frac{\partial_{11}\alpha(x^*,x^*)}{\alpha(x^*,x^*)} \\
    \begin{aligned}
    & \mbox{and}\quad c=-r''(x^*)
    +2r'(x^*)\frac{\partial_1\alpha(x^{*},x^{*})}{\alpha(x^*,x^*)} \\ &
    +r(x^*)\frac{\alpha(x^*,x^*)
      \big(\partial_{11}\alpha(x^*,x^*)+2\partial_{12}\alpha(x^*,x^*)\big)
      -2\partial_1\alpha(x^*,x^*)
      \big(\partial_1\alpha(x^*,x^*)+\partial_2\alpha(x^*,x^*)\big)}
    {\alpha(x^*,x^*)^2}.
    \end{aligned}
  \end{gather*}
  Using the fact that
  \begin{equation}
    \label{eq:pf-regul-1}
    r'(x^*)=r(x^*)\frac{\partial_1\alpha(x^*,x^*)}{\alpha(x^*,x^*)}
  \end{equation}
  since $x^*$ is an ES, we have that
  $$
  \alpha^2(x^*,x^*)(a+c)=2r(x^*)
  \big(\alpha(x^*,x^*)\partial_{12}\alpha(x^*,x^*)
  -\partial_1\alpha(x^*,x^*)\partial_2\alpha(x^*,x^*)\big).
  $$
  Hence,
  $$
  \alpha(x^*,x^*)\partial_{12}\alpha(x^*,x^*)
  -\partial_1\alpha(x^*,x^*)\partial_2\alpha(x^*,x^*)\not=0\quad
  \Longleftrightarrow\quad a+c\not=0.
  $$
  In particular, this implies that the function $\hat{D}(x,y)$ is
  non-zero in a neighborhood of $x^*$.

  \noindent For Point~(ii), observe that
  $$
  \bar{n}_1(x,y)+\bar{n}_2(x,y)=\frac{r(x)\frac{\alpha(y,y)-\alpha(y,x)}{x-y}
    +r(y)\frac{\alpha(x,x)-\alpha(x,y)}{x-y}}{(x-y)\hat{D}(x,y)}.
  $$
  By the proof of Lemma~\ref{lem:regul}, the numerator can be extended
  as a ${\cal C}^3$ function $h(x,y)$ by setting
  $$
  h(x,y)=-r(x)\int_0^1\partial_2\alpha(y,y+(x-y)u)du
  +r(y)\int_0^1\partial_2\alpha(x,y+(x-y)u)du
  $$
  for all $x,y\in{\cal X}$. In particular, $h(x,x)=0$ for all
  $x\in{\cal X}$. Therefore, Lemma~\ref{lem:regul} can be applied once
  more to prove that $\bar{n}_1(x,y)+\bar{n}_2(x,y)$ can be extended
  as a ${\cal C}^2$ function in the neighborhood of $x^*$ and that
  $$
  \lim_{x,y\rightarrow
    x^*,\,x\not=y}\bar{n}_1(x,y)+\bar{n}_2(x,y)=\frac{\frac{\partial h}{\partial
      x}(x^*,x^*)}{\hat{D}(x^*,x^*)}=\frac{r(x^*)\partial_{12}\alpha(x^*,x^*)
    -r'(x^*)\partial_2\alpha(x^*,x^*)}
  {\alpha(x^*,x^*)\partial_{12}\alpha(x^*,x^*)
    -\partial_1\alpha(x^*,x^*)\partial_2\alpha(x^*,x^*)}.
  $$
  Hence,~(\ref{eq:lim-barn}) and then (\ref{eq:lim-f3}) follow
  from~(\ref{eq:pf-regul-1}).

\medskip
  \noindent Point~(iii) is obtained from the fact that $f(x; x)=0$, from
  Lemma~\ref{lem:regul} and from the second-order Taylor expansion of
  $f(y; x)$. In this computation, one must use the fact that $x^*$ is
  an ES and~(\ref{eq:pf-regul-0}).

  \medskip \noindent The fact that $f(z;x,y)$ is ${\cal C}^2$ in ${\cal U}\times{\cal
    U}\times{\cal X}$ can be proven exactly as the regularity of
  $\bar{n}_1(x,y)+\bar{n}_2(x,y)$ above, observing that
  $$
  f(z;x,y)=r(z)-\frac{r(x)\frac{\alpha(z,x)\alpha(y,y)-\alpha(z,y)\alpha(y,x)}{x-y}
    +r(y)\frac{\alpha(z,y)\alpha(x,x)-\alpha(z,x)\alpha(x,y)}{x-y}}
  {(x-y)\hat{D}(x,y)}.
  $$
  Therefore, using the fact that $f(x;x,y)=f(y;x,y)=0$,
  Lemma~\ref{lem:regul} can be applied twice to prove that
  $$
  f(z;x,y)=\frac{\gamma}{2}(z-x)(z-y)+o(|z-x|\,|z-y|)
  $$
  for some constant $\gamma\in\RR$. The second-order Taylor expansion
  of $f(z;x,y)$ shows that $\gamma=\partial_{11}
  f(x^*;x^*,x^*)$. Now, because of~(\ref{eq:lim-f3}),
  $
  \ \partial_{11}f(z;x^*,x^*)=\partial_{11}f(z;x^*)
  $
  for all $z\in{\cal X}$. Hence $\gamma=a$, which ends the proof of
  Point~(iv).
\end{proof}

\begin{rem}
  \label{rem:regul}
  Let us remark that, if $x^*$ is not an evolutionary singularity,
  Point~\textup{(ii)} of Proposition~\textup{\ref{prop:regul}} need
  not to be true anymore, which may be surprising for the intuition
  and which has been a source of errors in some biological works.

  Moreover, if $x^*$ is an ES but Assumption~\textup{(\ref{eq:cond2})}
  ($a+c\not=0$) is not true, Point~\textup{(ii)} of
  Proposition~\textup{\ref{prop:regul}} may also fail.
  Indeed, in the case where
  $\alpha(x,x)\partial_{12}\alpha(x,x)
  \not=\partial_1\alpha(x,x)\partial_2\alpha(x,x)$
  for $x\not=x^*$,
  \begin{multline*}
    \bar{n}_1(x,x)+\bar{n}_2(x,x)=\frac{r(x)\partial_{12}\alpha(x,x)
      -r'(x)\partial_2\alpha(x,x)}{\alpha(x,x)\partial_{12}\alpha(x,x)
      -\partial_1\alpha(x,x)\partial_2\alpha(x,x)} \\
    =\frac{r(x^*)\big(\partial_{112}\alpha(x^*,x^*)
      +\partial_{122}\alpha(x^*,x^*)\big)
      -r'(x^*)\partial_{22}\alpha(x^*,x^*)
      -r''(x^*)\partial_2\alpha(x^*,x^*)+o(1)}
    {\alpha(x^*,x^*)\big(\partial_{112}\alpha(x^*,x^*)
      +\partial_{122}\alpha(x^*,x^*)\big)
      -\partial_2\alpha(x^*,x^*)\partial_{11}\alpha(x^*,x^*)
      -\partial_1\alpha(x^*,x^*)\partial_{22}\alpha(x^*,x^*)+o(1)}
  \end{multline*}
  as $x\rightarrow x^*$. This expression involves $r''(x^*)$, whose
  value is not imposed by the assumptions. Therefore, changing the
  function $r$ in such a way that $r(x^*)$ and $r'(x^*)$ are fixed but
  $r''(x^*)$ changes also changes the value of $\lim_{x,y\rightarrow
    x^*}\bar{n}_1(x,y)+\bar{n}_2(x,y)$.
\end{rem}

\subsubsection{On triple coexistence in the neighborhood of $x^*$}
\label{sec:BEB} Points~(iii) and~(iv) of
Proposition~\ref{prop:regul} allow one to determine the signs of
the 2- and 3-dimensional fitnesses in a trimorphic population with
traits $x,y,z$ close to $x^*$.  Combining this with the
classification of Zeeman~\cite{zeeman-93} (see
Section~\ref{sec:triple} and Figure \ref{fig:10}) gives the
 following corollary.

\begin{cor}
  \label{cor:BEB}
  For all ES $x^*$ satisfying the assumptions of
  Theorems~\textup{\ref{thm:br}} and such that
  $\partial_{11}f(x^*;x^*)\not= 0$, there exists a neighborhood ${\cal
    U}$ of of $x^*$ such that, for all distinct $x,y,z\in{\cal U}$,
  $(x,y,z)\not\in C_{coex}$, where $C_{coex}$ is defined
  in~\textup{(\ref{eq:def-Cn})}.
\end{cor}

\begin{proof} Let us assume for simplicity that $x^*=0$. We shall
  distinguish between the cases $a>0$ and $a<0$, and prove in each
  case that the fitnesses cannot have any of the sign configuration
  corresponding to the classes 26, 29, 31 and 33 in the neighborhood
  of $x^*$. Since all these classes contain the pattern of
  Fig.~\ref{fig:pattern}, we can assume without loss of generality
  that $f(x;y)\geq 0$, $f(y;x)\geq 0$, $f(z;x,y)\geq 0$ and $x<y$.

  \noindent Consider first the case $a<0$. It follows from
  Proposition~\ref{prop:regul}~(iv) that the function
  $f(\cdot;\cdot,\cdot)$ has the shape of Fig.~\ref{fig:shape-f3}~(a)
  in the neighborhood of $x^*$. In particular, this implies that
  $x<z<y$, $f(z;x,y)>0$, $f(x;y,z)<0$ and $f(y;x,z)<0$ as soon as
  $x,y,z$ are sufficiently close to $x^*$. In view of
  Fig.~\ref{fig:10}, these conditions are incompatible with classes 31
  and 33.  Moreover, $\partial_{11}f(x;y)<0$ for all $x,y$
  sufficiently close to $x^*$. Therefore, by Lemma~\ref{lem:regul},
\begin{equation}
  \label{eq:pf-BEB}
  \frac{\partial}{\partial
    x}\Big(\frac{f(x;y)}{y-x}\Big)=-\int_0^1u\partial_{11}f(y+u(x-y);y)du
\end{equation}
is positive for all $x,y$ sufficiently close to $x^*$.  Hence,
since $x<z<y$, we have $f(z;y)/(y-z)>f(x;y)/(y-x)\geq 0$ and thus
$f(z;y)>0$. Similarly, $f(z;x)>0$. Together with $f(z;x,y)>0$,
these conditions are incompatible with classes 26 and 29. This
ends the proof in the case where $a<0$.

\begin{figure}[h]
  \centering
  \mbox{\subfigure[$a<0$.]{
      \psfrag{x}{\footnotesize{$x$}} \psfrag{y}{\footnotesize{$y$}}
      \psfrag{z}{\footnotesize{$z$}} \psfrag{f(zxy)}{$f(z;x,y)$}
      \epsfig{figure=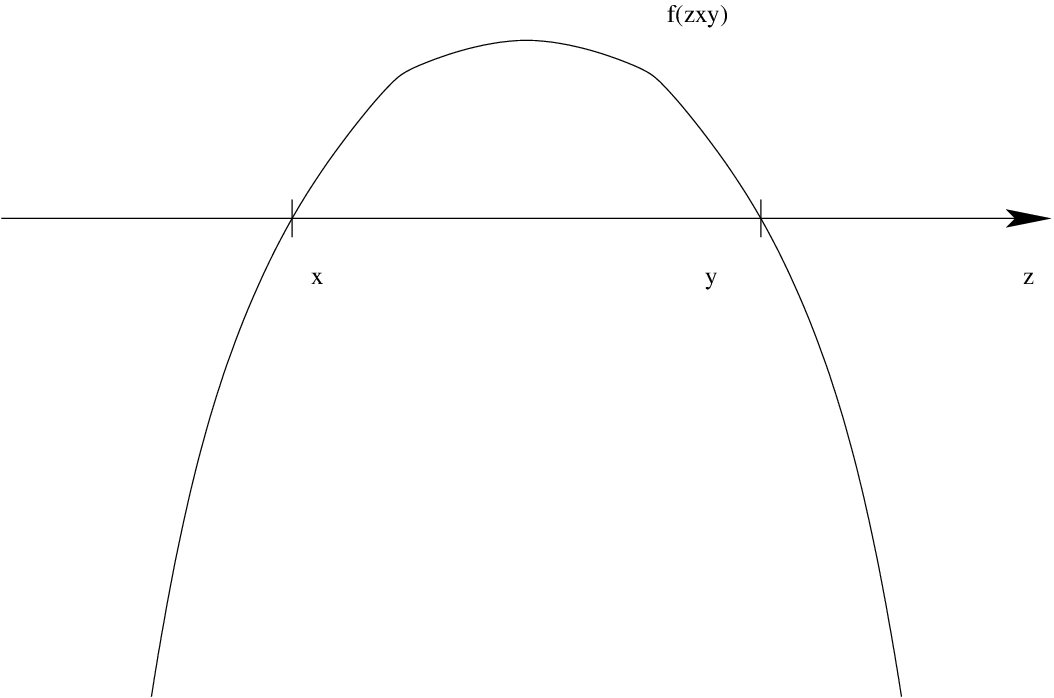, width=.30\textwidth}}\hspace{2cm}
    \subfigure[$a>0$.]{
      \psfrag{x}{\footnotesize{$x$}} \psfrag{y}{\footnotesize{$y$}}
      \psfrag{z}{\footnotesize{$z$}} \psfrag{f(zxy)}{$f(z;x,y)$}
      \epsfig{figure=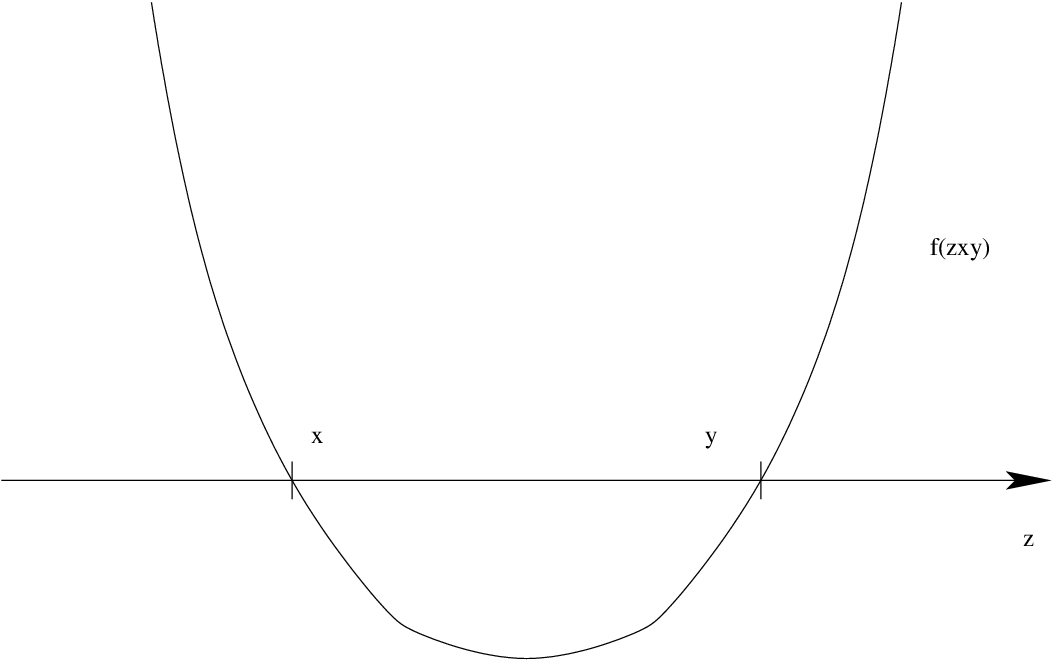, width=.30\textwidth}}}
  \caption{{\small The shape of the 3-dimensional fitness as a function of the
    sign of $a$.}}
  \label{fig:shape-f3}
\end{figure}

\noindent In the case where $a>0$, by
Proposition~\ref{prop:regul}~(iv), $f(\cdot;\cdot,\cdot)$
has the shape of Fig.~\ref{fig:shape-f3}~(b) in the neighborhood
of $x^*$. Therefore, $z\not\in[x,y]$. Assume for example that $\
z<x<y$. By Proposition~\ref{prop:regul}~(iv) again,
$f(x;y,z)<0$ and $f(y;x,z)>0$. These conditions are incompatible
with class 33. Moreover, using the fact that
$\partial_{11}f(x;y)>0$ for all $x,y$ sufficiently close to $x^*$,
it follows from the fact that~(\ref{eq:pf-BEB}) is negative that
$f(z;y)/(y-z)>f(x;y)/(y-x)\geq 0$ and thus that $f(z;y)>0$.
Similarly, because of Assumption~(\ref{eq:cond1}),
$\partial_{22}f(x;y)>0$ for all $x,y$ sufficiently close to $x^*$.
Therefore, by Lemma~\ref{lem:regul},
$$
\frac{\partial}{\partial
  x}\Big(\frac{f(y;x)}{y-x}\Big)=-\int_0^1u\partial_{22}f(y;y+u(x-y))du<0
$$
for all $x,y$ sufficiently close to $x^*$. Thus, $f(y;x)\geq 0$
implies that $f(y;z)>0$. Together with the fact that $f(x;y,z)<0$,
these conditions are incompatible with classes 26, 29 and 31.

\noindent In the case where $\ x<y<z$, the method above proves that
$f(x;z)>0$, $f(z;x)>0$ and $f(y;x,z)<0$, which is again incompatible
with classes 26, 29, 31 and 33. This ends the proof of
Corollary~\ref{cor:BEB}.\end{proof}

\subsubsection{Double coexistence region in the neighborhood of $x^*$}
\label{sec:coex}

We prove here Proposition \ref{prop:coex}, that gives a criterion for
the coexistence of two traits in the neighborhood of $x^*$, and we end
the proof of Theorem~\ref{thm:br}.  The proof of
Proposition~\ref{prop:coex} is based on the study of the region of
double coexistence, defined as $\{(x,y)\in{\cal X}:f(x;y)>0\mbox{\
  and\ }f(y;x)>0\}$ in the neighborhood of $x^*$. The proof of
Theorem~\ref{thm:br} is based on a case-by-case study that extends the
proof of Corollary~\ref{cor:BEB}.

\bigskip
\noindent {\bf Proof of Proposition~\ref{prop:coex}\phantom{9}} It
follows from Proposition~\ref{prop:regul}~(iii) that the set of
$(x,y)\in{\cal X}$ such that $f(y;x)=0$ is composed of the line
$\:\{y=x\}\:$ and of a set which is, because of the Implicit Function
Theorem, a curve in the neighborhood of $x^*$, containing $(x^*,x^*)$
and admitting as tangent at this point the line
$\{a(y-x^*)=c(x-x^*)\}$. Let us call $\gamma$ this curve. Since $a<c$,
the curves $\gamma$ and $\:\{y=x\}\:$ divide ${\cal X}^2$ in the
neighborhood of $(x^*,x^*)$ into 4 regions. Moreover, because of
\eqref{eq:dl-f2}, $f(y;x)$ changes sign when the point $(x,y)$ changes
region by crossing either the line $\:\{y=x\}\:$ or the curve
$\gamma$.

\noindent It is elementary from a case-by-case study to check that
coexistence can occur in the neighborhood of $x^*$ if $c>a>0$,
$a>c>0$, $-a<c<0<a$ and $a<0<-a<c$, and that coexistence cannot occur
in the neighborhood of $x^*$ if $c<-a<0<a$, $c<a<0$, $a<c<0$ and
$a<0<c<-a$. The cases where coexistence is possible are represented in
Fig.~\ref{fig:coex} in the case where $x^*=0$. In these figures, the
curve $\gamma$ is represented by its tangent line
$\{a(y-x^*)=c(x-x^*)\}$ and the sign of $f(y;x)$ is represented by $+$
and $-$ signs depending on the position of $(x,y)$ with respect to
$\gamma$ and $\:\{y=x\}$. The sign of $f(x;y)$ is obtained by an axial
symmetry of the figure with axis $\:\{y=x\}$. We denote by $\gamma^s$
the symmetric of the curve $\gamma$ with respect to this axis. The
region of coexistence is the one where $f(y;x)>0$ and $f(x;y)>0$.

\begin{figure}[h]
  \centering
  \mbox{\subfigure[$c>a>0$.]{
      \psfrag{x}{\footnotesize{$x$}} \psfrag{y}{\footnotesize{$y$}}
      \psfrag{g}{\footnotesize{$\gamma^s$}} \psfrag{gs}{\footnotesize{$\gamma$}}
      \psfrag{Coexistence}{\tiny{Coexistence}} \psfrag{x=y}{\footnotesize{$y=x$}}
      \psfrag{s1}{\footnotesize{$+$}} \psfrag{s2}{\footnotesize{$+$}}
      \psfrag{s3}{\footnotesize{$-$}} \psfrag{s4}{\footnotesize{$+$}}
      \epsfig{figure=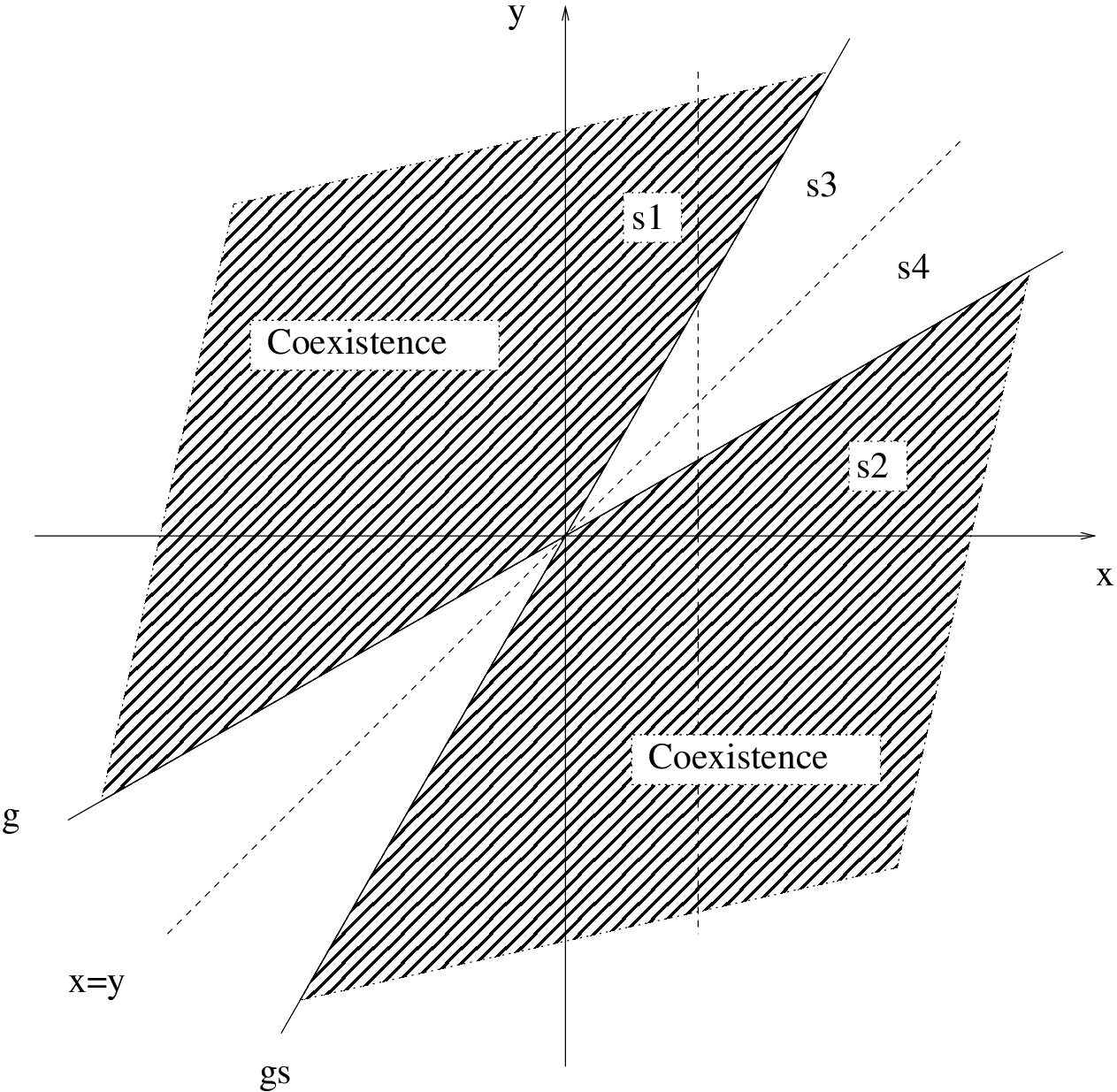, width=.47\textwidth}}\quad
    \subfigure[$-a<c<0<a$.]{
      \psfrag{x}{\footnotesize{$x$}} \psfrag{y}{\footnotesize{$y$}}
      \psfrag{g}{\footnotesize{$\gamma^s$}} \psfrag{gs}{\footnotesize{$\gamma$}}
      \psfrag{Coexistence}{\tiny{Coexistence}} \psfrag{x=y}{\footnotesize{$y=x$}}
      \psfrag{s1}{\footnotesize{$+$}} \psfrag{s2}{\footnotesize{$+$}}
      \psfrag{s3}{\footnotesize{$+$}} \psfrag{s4}{\footnotesize{$-$}}
      \psfrag{s5}{\footnotesize{$+$}} \psfrag{s6}{\footnotesize{$-$}}
      \epsfig{figure=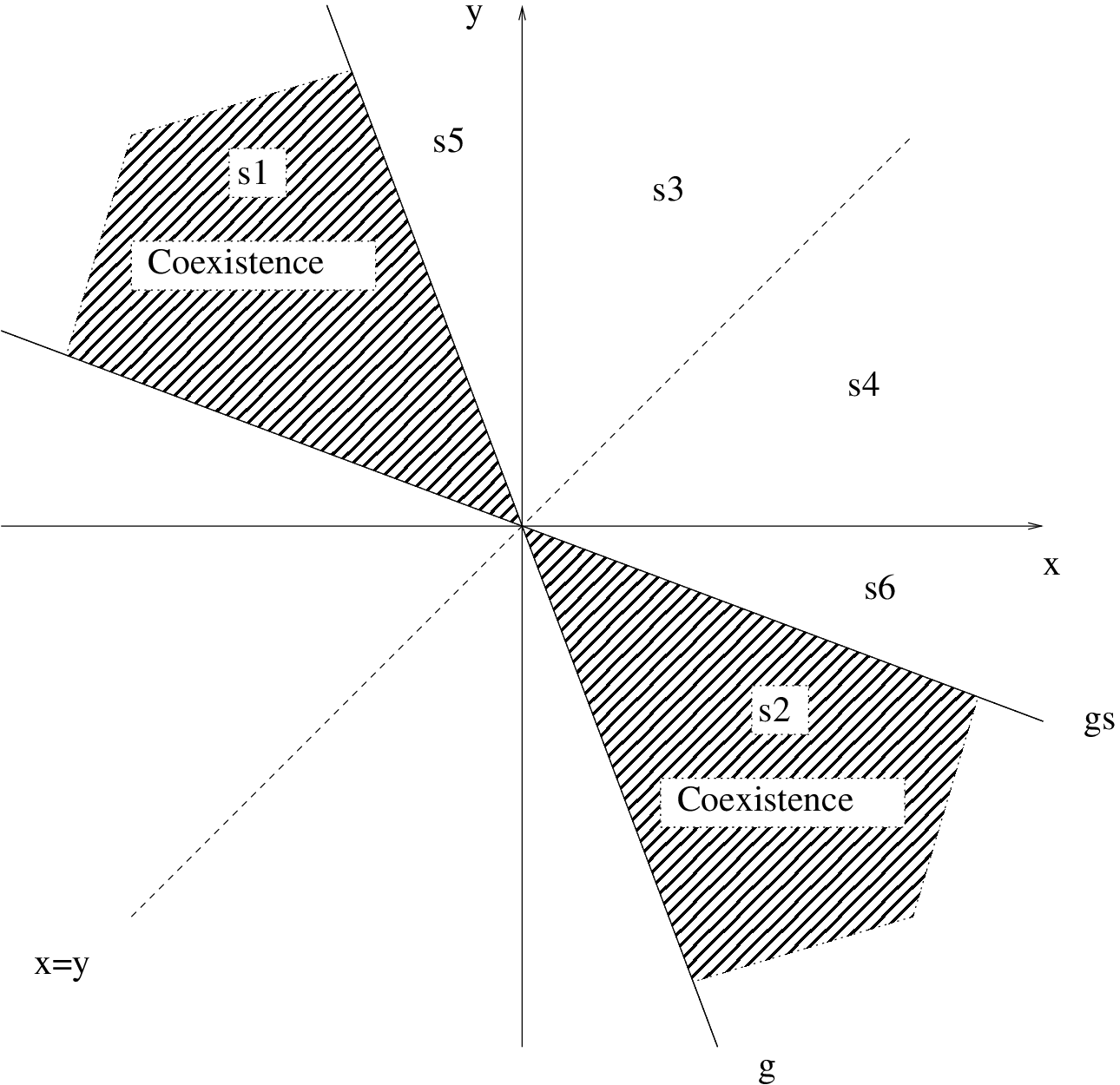, width=.47\textwidth}}} \\
  \mbox{\subfigure[$a>c>0$.]{
      \psfrag{x}{\footnotesize{$x$}} \psfrag{y}{\footnotesize{$y$}}
      \psfrag{g}{\footnotesize{$\gamma$}} \psfrag{gs}{\footnotesize{$\gamma^s$}}
      \psfrag{Coexistence}{\tiny{Coexistence}} \psfrag{x=y}{\footnotesize{$y=x$}}
      \psfrag{s1}{\footnotesize{$+$}} \psfrag{s2}{\footnotesize{$+$}}
      \psfrag{s3}{\footnotesize{$+$}} \psfrag{s4}{\footnotesize{$-$}}
      \epsfig{figure=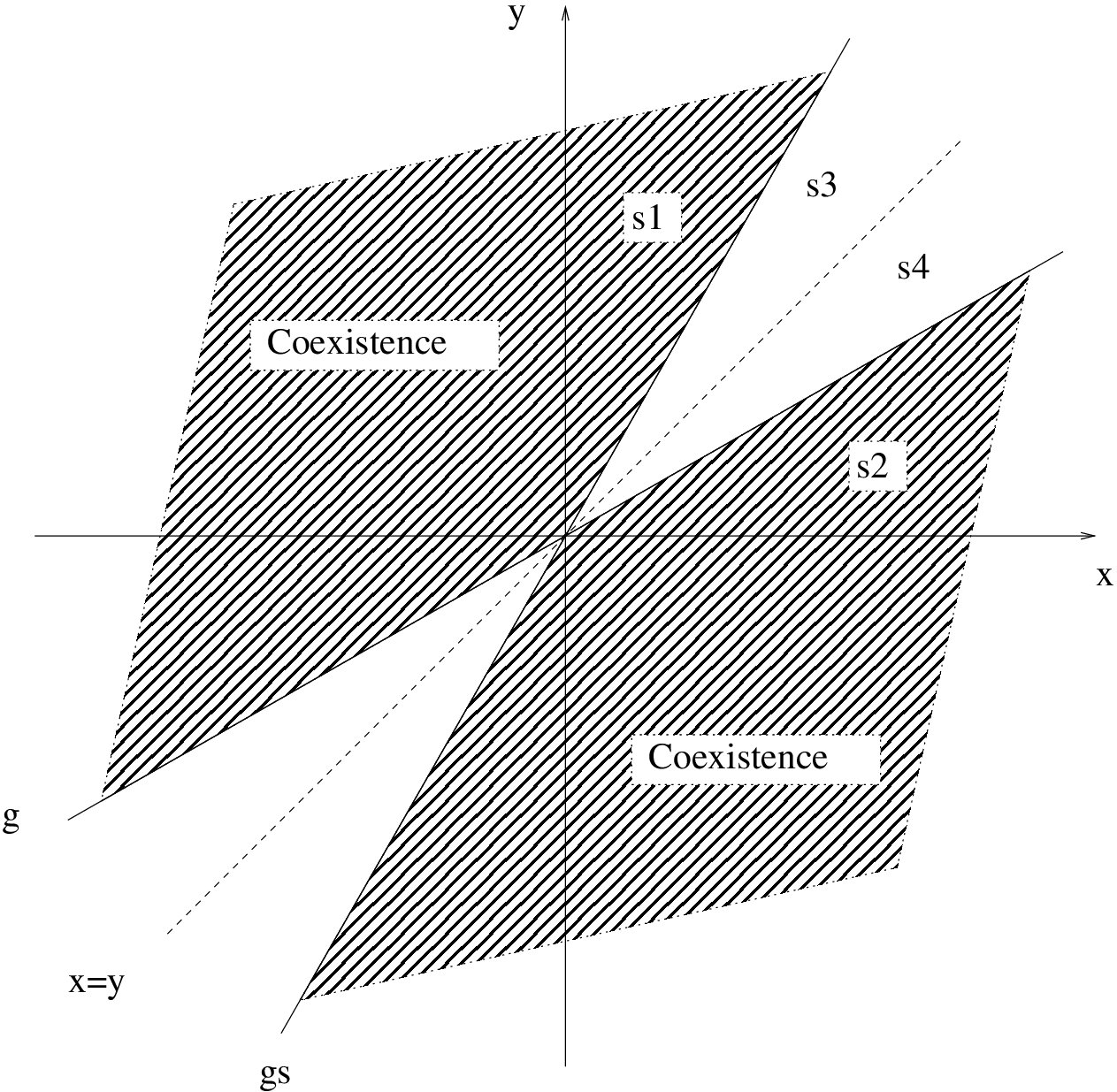, width=.47\textwidth}}\quad
    \subfigure[$a<0<-a<c$.]{
      \psfrag{x}{\footnotesize{$x$}} \psfrag{y}{\footnotesize{$y$}}
      \psfrag{gs}{\footnotesize{$\gamma^s$}} \psfrag{g}{\footnotesize{$\gamma$}}
      \psfrag{Coexistence}{\tiny{Coexistence}} \psfrag{x=y}{\footnotesize{$y=x$}}
      \psfrag{s1}{\footnotesize{$+$}} \psfrag{s2}{\footnotesize{$+$}}
      \psfrag{s3}{\footnotesize{$-$}} \psfrag{s4}{\footnotesize{$+$}}
      \psfrag{s5}{\footnotesize{$-$}} \psfrag{s6}{\footnotesize{$+$}}
      \epsfig{figure=fig_coex_2.eps, width=.47\textwidth}}}
  \caption{{\small In the four cases where coexistence is possible, these
    figures show the sign configuration of $f(y;x)$ depending on the
    position of $(x,y)$ with respect to the curve $\gamma$ and the line
    $\{y=x\}$ and the region of coexistence. For convenience, we
    assumed $x^*=0$.}}
  \label{fig:coex}
\end{figure}

\noindent Note that the expansion of $f(y; x)$ of
Proposition~\ref{prop:regul}~(iii) does not make use of any
assumption on $a$ and $c$. Therefore, a similar study can be done
to treat the degenerate cases. One easily obtains that coexistence
is possible in the neighborhood of $(x^*,x^*)$ if $c=a>0$, $c=0$
and $a>0$ or $a=0$ and $c>0$. Similarly, coexistence cannot occur
in the neighborhood of $(x^*,x^*)$ if $c=a<0$, $c=0$ and $a<0$ or
$a=0$ and $c<0$. The case $c=-a$ is undetermined and depends on
higher-order expansions of the fitness function.\hfill{$\Box$}


\bigskip \noindent
{\bf Proof of Theorem~\ref{thm:br}~(b): case $a<0$}

\noindent It follows from Theorems~\ref{thm:pes=tss} that for any
fixed $\eta>0$, for $\varepsilon$ small enough, the PES stays
monomorphic until it reaches $(x^*-\eta,x^*+\eta)$. Moreover, in view
of the proof of Proposition~\ref{prop:attract}, no mutant out of
$(x^*-\eta,x^*+\eta)$ can invade the population as long as it is
monomorphic with support inside this interval.

\noindent Now, by Proposition~\ref{prop:coex}, when $a<0$, coexistence
may happen in the rescaled PES if $c>-a$.
In this case, at the first coexistence time $\tau^\varepsilon$, the
two traits $x$ and $y$ that coexist belong to $(x^*-\eta,x^*+\eta)$
and are distant of less than $\varepsilon\mbox{Diam}({\cal X})$ since
$m(x,\cdot)$ has support in ${\cal X}-x$.

\noindent Let us examine what happens when a mutant trait $z$ invades
this population. Remind that we showed in the proof of
Corollary~\ref{cor:BEB} that, if $a<0$, $x<y$, $f(x;y)>0$, $f(y;x)>0$
and $f(z;x,y)>0$, then $f(x;y,z)<0$, $f(y;x,z)<0$, $f(z;y)>0$ and
$f(z;x)>0$. Examining Fig.~\ref{fig:10}, we see that these conditions
are incompatible with all classes except classes 7 and 9. Therefore,
once the mutant $z$ invades, the new state of the rescaled PES can be
either $\bar{n}(z)\delta_z$ in the case of class 7, or either
$\bar{n}_1(x,z)\delta_x+\bar{n}_2(x,z)\delta_z$ or
$\bar{n}_1(y,z)\delta_y+\bar{n}_2(y,z)\delta_z$ in the case of class
9. In particular, we see that either the population becomes
monomorphic again, or it stays dimorphic, but the distance between the
two traits of the support of the PES has decreased. In addition, in
both cases, the support of the new state of the PES is a subset of
$(x^*-\eta,x^*+\eta)$. Hence, $\eta$-branching, as defined in
Definition~\ref{def:br}, cannot occur as soon as
$\varepsilon<\eta/(2\mbox{Diam}({\cal X}))$. This ends the proof of
Theorem~\ref{thm:br}~(b).\hfill$\Box$

\bigskip \noindent
{\bf Proof of Theorem~\ref{thm:br}~(a): case $a>0$}

\noindent By Proposition~\ref{prop:coex}, when $a>0$, under the assumptions
of Theorem~\ref{thm:br}, we are in the situation of
Fig.~\ref{fig:coex}~(a), and hence coexistence is always possible
in the neighborhood of $x^*$. Fix $\eta>0$. We are going to prove
that, if $\eta$ is small enough, then for $\varepsilon$ small
enough,
\begin{description}
\item[\textmd{(i)}] the first time of coexistence $\tau^\varepsilon$ is a.s.\
  finite and $\mbox{Supp}(Z^\varepsilon_{\tau^\varepsilon-})\subset
  (x^*-\eta, x^*+\eta)$ a.s.;
\item[\textmd{(ii)}] after time $\tau^\varepsilon$, the distance
  between the two points of the support of the rescaled PES is
  non-descreasing and becomes a.s.\ bigger than $\eta/2$ in finite
  time, before exiting the interval $(x^*-\eta,x^*+\eta)$.
\end{description}
These two points will clearly imply Theorem~\ref{thm:br}~(a).

\noindent For Point~(i), observe first that, by
Proposition~\ref{prop:attract}, if $\tau^\varepsilon<+\infty$,
then $\mbox{Supp}(Z^\varepsilon_{\tau^\varepsilon-})\subset
(x^*-\eta, x^*+\eta)$. Thus we only have to prove that $\
\PP(\tau^\varepsilon<\infty)=1$.

 \noindent In view of Fig.~\ref{fig:coex}~(a), we observe that for
a given jump size, the closer the support is from $x^*$, the
easier co-existence is. The proof is based on this fact, taking
into account the additional difficulty that  the jump rate is
almost zero  in that case.

\noindent Fix $\kappa>0$.  Let us define
$$
\theta_\kappa=\inf\big\{t\geq 0:\
\mbox{Supp}(Z^\varepsilon_t)\subset(x^*-\kappa\varepsilon,x^*+\kappa\varepsilon)
\big\}.
$$

From Assumptions~(A'3) and (A'''), the functions
$$
x\mapsto\int_{0}^{+\infty}h\:m(x,h)dh\quad\mbox{and}\quad x\mapsto\int_{-\infty}^0
h\: m(x,h)dh
$$
are continuous and  there exists $\beta>0$ such that, for all
$x\in[x^*-\eta,x^*+\eta]$,
\begin{equation}
  \label{eq:on-saute}
  \int_{0}^{+\infty}h\:m(x,h)dh>\beta>0\quad\mbox{and}\quad\int_{-\infty}^0
  h\:m(x,h)dh<-\beta<0.
\end{equation}
It is thus elementary to check, using
 \eqref{eq:dl-f2}, that for any
 $x\in[x^*-\eta,x^*-\kappa\varepsilon]$, resp.
 $x\in[x^*+\kappa\varepsilon,x^*+\eta]$,
\begin{align*}
  & \int_{\beta/2}^{+\infty}[g(x+\varepsilon h,x)]_+m(x,h)dh\geq
  C\varepsilon^2\beta\kappa>0\ ; \\ \hbox{resp.} &
  \int^{-\beta/2}_{-\infty}[g(x+\varepsilon h,x)]_+m(x,h)dh\leq
  -C'\varepsilon^2\beta\kappa<0.
\end{align*}
Assume that $\ \PP(\tau^\varepsilon=\infty\ ;\
\theta_{\beta/2}=\infty)>0$. Then, on this event, in view
of~(\ref{eq:rescaled-generator}), the previous inequalities show that
there are infinitely many jumps in the TSS, with jump size bigger than
$\varepsilon \beta/2$. This yields a contradiction since the TSS is
monotonous before $\tau^\varepsilon$. Indeed, drawing a vertical line
at some level $x$ in Fig.~\ref{fig:coex}~(a) (for example the vertical
dotted line), one can see that all the mutants invading the
monomorphic population with trait $x$ either coexist with $x$ or are
closer to $x^*$ than $x$. On the other hand, it is clear from
Fig.~\ref{fig:coex} that the first jump after time $\theta_{\beta/2}$
in the TSS with jump size bigger than $\varepsilon\beta/2$ (which
almost surely happens) drives the TSS in the coexistence
region. Therefore, $\ \PP(\tau^\varepsilon=\infty\ ;\
\theta_{\beta/2}<\infty)=0$ and then $\
\PP(\tau^\varepsilon=\infty)=0$.

\medskip
\noindent For Point~(ii), assume that the rescaled PES is dimorphic at
some time $t$, with support $\{x,y\}$, $x<y$. Let us examine what
happens when a mutant trait $z$ invades this population. Remind that
we showed in the proof of Corollary~\ref{cor:BEB} that, if $a>0$ and
$x,y,z$ belong to $(x^*-\eta_0,x^*+\eta_0)$ for some $\eta_0>0$ and
satisfy $x<y$, $f(x;y)>0$, $f(y;x)>0$ and $f(z;x,y)>0$, then
\begin{itemize}
\item either $z<x<y$ and $f(x;y,z)<0$, $f(y;x,z)>0$, $f(z;y)>0$
and
  $f(y;z)>0$,
\item or $x<y<z$ and $f(x;y,z)>0$, $f(y;x,z)<0$, $f(z;x)>0$ and
  $f(x;z)>0$.
\end{itemize}
We can assume without loss of generality that $\eta<\eta_0$.
Examining Fig.~\ref{fig:10}, we see that both situations are only
compatible with classes 9, 10, 11 and 12. Therefore, once the mutant
$z$ invades, the new state of the rescaled PES is
$\bar{n}_1(x,z)\delta_x+\bar{n}_2(x,z)\delta_z$ if $x<y<z$ or
$\bar{n}_1(y,z)\delta_y+\bar{n}_2(y,z)\delta_z$ if $z<x<y$. In both
cases, we see that the distance between the two traits of the support
of the PES can only increase until the stopping time $\theta'$ where
one of the points of the support leaves $(x^*-\eta,x^*+\eta)$. In
order to end the proof, it suffices to prove that, if $\eta$ is
sufficiently small,
$$
\theta'<\infty\quad\mbox{\ a.s.}\quad\mbox{and}\quad
\mbox{Diam}(\mbox{Supp}(\tilde{Z}^\varepsilon_{\theta'}))>\eta/2.
$$
The fact that $\theta'<\infty$ a.s.\ can be proved
using~(\ref{eq:on-saute}) in a similar way as for Point~(i). The lower
bound of the diameter of the PES immediately follows from the fact
that
$$
\tau^\varepsilon>\theta_{\kappa_0}\quad\mbox{\ a.s.,}
\quad\mbox{where}\quad
\kappa_0=\frac{2c\mbox{Diam({\cal X})}}{c-a}.
$$
This inequality follows from the following argument: for any
$x,y\in\RR$ such that
\begin{equation}
  \label{eq:last}
  |x-x^*|\geq \frac{2c|x-y|}{c-a},
\end{equation}
it can be easily checked that
$$
|y-x^*|\geq \frac{1}{2}\Big(1+\frac{a}{c}\Big)|x-x^*|\quad
\hbox{and}\quad |y-x^*|\leq\Big(1+\frac{c-a}{2c}\Big)|x-x^*|.
$$
Since $0<a<c$, we have
$$
\frac{1}{2}\Big(1+\frac{a}{c}\Big)>\frac{a}{c}\quad\mbox{and}\quad
1+\frac{c-a}{2c}<1+\frac{c-a}{a}=\frac{c}{a}.
$$
Now, $\{(y-x^*)=(c/a)(x-x^*)\}$ is tangent to $\gamma$ at $(x^*,x^*)$
and $\{(y-x^*)=(a/c)(x-x^*)\}$ is tangent to $\gamma^s$ at
$(x^*,x^*)$. Therefore, in view of Fig.~\ref{fig:coex}~(a), any
$x,y\in\RR$ satisfying~(\ref{eq:last}) do not coexist together.

\noindent To conclude, it suffices to observe that, in the rescaled
PES $\tilde{Z}$, the distance between a mutant trait and the trait of
its progenitor in the PES is always smaller than
$\varepsilon\mbox{Diam}({\cal X})$. Therefore, for any
$x\in(x^*-\eta,x^*+\eta)$ such that $|x-x^*|\geq \varepsilon\kappa_0$,
any mutant trait $y$ born from $x$ do not coexist with
$x$.\hfill$\Box$



\appendix

\section{Proof of Theorem~\ref{thm:PES-fdd}}
\label{sec:pf-PES}

The proof of this result is very similar to the proof
of~\cite[Thm.1]{champagnat-06}. We will not repeat all the details and
we will restrict ourselves to the steps that must be modified. The
general idea of the proof follows closely the heuristic argument of
Section~\ref{sec:idea-pf}. Its skeleton is similar to the one
in~\cite{champagnat-06} for monomorphic populations.

\noindent For all $\varepsilon>0$, $t>0$, and $\Gamma\subset{\cal
X}$ measurable, let
\begin{multline*}
  A_{\varepsilon,d}(t,\Gamma):=\Big\{\mbox{Supp}(\nu_{t/Ku_K})\subset\Gamma \mbox{\
    has $d$ elements that coexist, say\ }x_1,\ldots,x_d, \\ \mbox{\ and\
  }\forall 1\leq i\leq d,\
  |\langle\nu_{t/Ku_k},\mathbf{1}_{\{x_i\}}\rangle
  -\bar{n}_i(\mathbf{x})|<\varepsilon\}.
\end{multline*}
To prove Theorem~\ref{thm:PES-fdd}, we establish that for all
$\varepsilon>0$, $t>0$ and $\Gamma\subset{\cal X}$ measurable,
\begin{equation}
  \label{eq:pf-C6-0}
  \lim_{K\rightarrow+\infty}\mathbb{P}(A_{\varepsilon,d}(t,\Gamma))
  =\mathbb{P}(\mbox{Supp}(Z_t)\subset\Gamma\mbox{\ and has $d$
    elements}).
\end{equation}
where $(Z_t,t\geq 0)$ is defined in Theorem~\ref{thm:PES-fdd}. The
first ingredient of the proof is the following proposition, which
generalizes Theorem~3~(a) and (b) of~\cite{champagnat-06}.
\begin{prop}
  \label{prop:EK}
  Assume that, for any $K\geq 1$,
  $\mbox{Supp}(\nu^K_0)=\{x_1,\ldots,x_d\}$ and $\langle
  \nu_0^{K},\mathbf{1}_{\{x_i\}}\rangle\in C$ a.s., where $C$ is a
  compact subset of $\RR_+$. Let $\phi(t,(n_1,\ldots,n_d))$ denote the
  value at time $t$ of the solution of $LV(d,\mathbf{x})$ with initial
  condition $(n_1,\ldots,n_d)$. Then, for all $T>0$,
  \begin{equation}
    \label{eq:EK}
    \lim_{K\rightarrow+\infty}\sup_{1\leq i\leq d,\ t\in[0,T]}
    \Big|\langle\nu^K_t,\mathbf{1}_{\{x_i\}}\rangle
    -\phi_i(t,(\langle\nu^K_0,\mathbf{1}_{\{x_1\}}\rangle,
    \ldots,\langle\nu^K_0,\mathbf{1}_{\{x_d\}}\rangle))\Big|=0\quad a.s.
  \end{equation}
\end{prop}
This result is a direct corollary of Theorem~11.2.1
of~\cite{ethier-kurtz-86}, except for two small difficulties. The
first one is that Theorem~11.2.1 of~\cite{ethier-kurtz-86} assumes
that the function $\mathbf{n}\mapsto F^{\mathbf{x}}(\mathbf{n})$
involved in the definition~(\ref{eq:LVcompl}) of the Lotka
Volterra system is uniformly Lipschitz on $\RR_+^d$, which is not
the case. However, observe first that, if $n_i\leq M$ for some
$M>0$ for all $i\in\{1,\ldots,d\}$, then
$\phi_i(t,(n_1,\ldots,n_d))\leq
M\vee(2\bar{\lambda}/\underline{\alpha})$ for all $t\geq 0$.
Indeed, if there is equality for some $t\geq 0$ and
$i\in\{1,\ldots,d\}$, then $\dot{\phi}_i(t,(n_1,\ldots,n_d))<0$.
Therefore, the coefficients of the system $LV(d,\mathbf{x})$ are
uniformly Lipschitz on the set of states that can be attained by
the solution of the system starting from any initial conditions in
a compact set. The second difficulty is that Theorem~11.2.1
of~\cite{ethier-kurtz-86} only implies that~(\ref{eq:EK}) holds on
the event where there is no mutation between $0$ and $T$. In
Lemma~2~(a) of~\cite{champagnat-06}, it is proved that for general
initial condition $\nu^K_0$, the probability of mutation on the
time interval $[0,T]$ converges to 0, thus the conclusion follows.

\medskip
\noindent The second ingredient is the following exponential deviation
estimate on the so-called ``problem of exit from an attracting
domain''~\cite{freidlin-wentzell-84}. It generalizes Theorem~3~(c)
of~\cite{champagnat-06}.
\begin{prop}
  \label{prop:FW}
  Let $x_1,\ldots,x_d\in{\cal X}$ coexist. Then there exist constants
  $c,V>0$ such that, for any sufficently small $\varepsilon>0$, if
  $(\langle \nu_0^{K},\mathbf{1}_{\{x_i\}}\rangle)_{1 \leq i\leq d}$
  belongs to the $(\varepsilon/2)$-neighborhood of
  $\bar{\mathbf{n}}(\mathbf{x})$, the time of exit of $(\langle
  \nu_t^{K},\mathbf{1}_{\{x_i\}}\rangle)_{1 \leq i\leq d}$ from the
  $\varepsilon$-neighborhhod of $\bar{\mathbf{n}}(\mathbf{x})$ is
  bigger than $e^{VK}\wedge\tau$ with probability converging to 1,
  where $\tau$ denotes the first mutation time. Moreover, the previous
  result also holds if, for all $i\in\{1,\ldots,d\}$, the death rate
  of an individual with trait $x_i$
  \begin{equation}
    \label{eq:FW}
    \mu(x_i)+\sum_{j=1}^d\alpha(x_i,x_j)\langle
    \nu_t^{K},\mathbf{1}_{\{x_j\}}\rangle
  \end{equation}
  is perturbed by an additional random process that is uniformly
  bounded by $c\varepsilon$.
\end{prop}
Such results are fairly standard and can be proved in a variety of
ways. We let the proof to the reader. The first part of this
proposition is used to prove that, when the first mutation occurs, the
population densities have never left the $\varepsilon$-neighborhood of
$\bar{\mathbf{n}}(\mathbf{x})$ and the second is used to prove that,
after the first mutation, as long as the mutant population is small,
the resident population densities do not leave the
$\varepsilon$-neighborhood of $\bar{\mathbf{n}}(\mathbf{x})$. In this
case, the additional term in~(\ref{eq:FW}) is $\alpha(x_i,y)\langle
\nu_t^{K},\mathbf{1}_{\{y\}}\rangle$, where $y$ is the mutant trait,
which is smaller that $\bar{\alpha}\varepsilon$ if $\langle
\nu_t^{K},\mathbf{1}_{\{y\}}\rangle\leq\varepsilon$.


\medskip
\noindent From these two results can be deduced the following
lemma, which is the extension of Lemma~2~(b) and~(c)
of~\cite{champagnat-06}. The proof is a simple copy of the
argument in~\cite{champagnat-06}.
\begin{lem}
  \label{lem:bef-mut}
  Let $\ \mbox{Supp}(\nu^K_0)=\{x_1,\ldots,x_d\}\ $ that coexist and let
  $\tau$ denote the first mutation time. There exists $\varepsilon_0$
  such that, if $(\langle \nu_0^{K},\mathbf{1}_{\{x_i\}}\rangle)_{1
    \leq i\leq d}$ belongs to the $\varepsilon_0$-neighborhood of
  $\bar{\mathbf{n}}(\mathbf{x})$, then, for any
  $\varepsilon<\varepsilon_0$,
  \begin{gather*}
    \lim_{K\rightarrow+\infty}\PP\Big(\tau>\log K,\ \sup_{1\leq i\leq d,\
      t\in[\log K,\tau]}|\langle
    \nu_t^{K},\mathbf{1}_{\{x_i\}}\rangle-\bar{n}_i(\mathbf{x})|<\varepsilon\Big)=1,
    \\ Ku_K \tau\ \overset{{\cal L}}{\underset{K\rightarrow
        \infty}{\Longrightarrow}}\
    \mbox{\textup{Exp}}\Big(\sum_{j=1}^dp(x_j)\lambda(x_j)\bar{n}_j(\mathbf{x})\Big)
    \\ \mbox{and}\quad
    \lim_{K\rightarrow+\infty}\PP(\mbox{at time $\tau$, the mutant is
      born from trait
      $x_i$})=\frac{p(x_i)\lambda(x_i)\bar{n}_i(\mathbf{x})}
    {\sum_{j=1}^dp(x_j)\lambda(x_j)\bar{n}_j(\mathbf{x})}
  \end{gather*}
  for all $i\in\{1,\ldots,d\}$, where $\:\overset{{\cal L}}{\Longrightarrow}\:$ denotes the
  convergence in law of real r.v.\ and $\mbox{\textup{Exp}}(u)$
  denotes the exponential law with parameter $u$.
\end{lem}

\noindent The fourth ingredient is the following lemma, which is
an extension of Lemma~3 of~\cite{champagnat-06}.
\begin{lem}
  \label{lem:aft-mut}
  Let $\mbox{Supp}(\nu^K_0)=\{x_1,\ldots,x_d,y\}$ where
  $x_1,\ldots,x_d$ coexist and $y$ is a mutant trait that satisfy
  Assumption~\textup{(B)}. Let $\tau$ denote the first next mutation time,
  and define
  \begin{align*}
    \tau_1&=\inf\{t\geq 0:\forall i\in I(\mathbf{n}^*),\ |\langle
    \nu_t^{K},\mathbf{1}_{\{x_i\}}\rangle-n^*_i|<\varepsilon \mbox{\
      and\ }\forall i\not\in I(\mathbf{n}^*),\ \langle
    \nu_t^{K},\mathbf{1}_{\{x_i\}}\rangle=0\}
    \\ \tau_2&=\inf\{t\geq 0:\langle
    \nu_t^{K},\mathbf{1}_{\{y\}}\rangle=0\mbox{\ and\ }\forall
    i\in\{1,\ldots,d\},\ |\langle
    \nu_t^{K},\mathbf{1}_{\{x_i\}}\rangle-\bar{n}_{i}(\mathbf{x})|<\varepsilon\}.
  \end{align*}
  Assume that $\langle \nu_0^{K},\mathbf{1}_{\{y\}}\rangle=1/K$ (a
  single initial mutant). Then, there exists $\varepsilon_0$ such that
  for all $\varepsilon<\varepsilon_0$, if $(\langle
  \nu_0^{K},\mathbf{1}_{\{x_i\}}\rangle)_{1 \leq i\leq d}$ belongs to
  the $\varepsilon$-neighborhood of $\bar{\mathbf{n}}(\mathbf{x})$,
  \begin{gather*}
    \lim_{K\rightarrow+\infty}\PP(\tau_1<\tau_2)
    =\frac{[f(y;\mathbf{x})]_+}{\lambda(y)},\quad
    \lim_{K\rightarrow+\infty}\PP(\tau_2<\tau_1)
    =1-\frac{[f(y;\mathbf{x})]_+}{\lambda(y)} \\
     \mbox{and}\quad\forall \eta>0,\quad
    \lim_{K\rightarrow+\infty}\PP\left(\tau_1\wedge\tau_2<\frac{\eta}{K
      u_K}\wedge\tau\right)=1.
  \end{gather*}
\end{lem}
The proof of this lemma is similar to the proof of Lemma~3
in~\cite{champagnat-06}. The main steps are the following. Assume
first that $\varepsilon<1/2$. We introduce the following stopping
times: \begin{align*}
  R_\varepsilon^K&=\inf\{t\geq 0:\exists i\in\{1,\ldots,d\},\ |\langle
    \nu_t^{K},\mathbf{1}_{\{x_i\}}\rangle-\bar{n}_{i}(\mathbf{x})|\geq\varepsilon\}
    \\ S_\varepsilon^K&=\inf\{t\geq 0:\langle
    \nu_t^{K},\mathbf{1}_{\{y\}}\rangle\geq\varepsilon\} \\
    S_0^K&=\inf\{t\geq 0:\langle
    \nu_t^{K},\mathbf{1}_{\{y\}}\rangle=0\}.
\end{align*}
$R^K_\varepsilon$ is the time of drift of the resident population away
from its equilibrium, $S_\varepsilon^K$ is the time of \emph{invasion}
of the mutant trait (time $t_1$ in Fig.~\ref{fig:inv-fix}) and $S_0^K$
is the time of extinction of the mutant trait. By the second part of
Proposition~\ref{prop:FW}, it can be proven exactly as
in~\cite{champagnat-06} that there exists $\rho,V>0$ and $c<1$ such
that, for $K$ large enough,
$$
\PP\Big(\frac{\rho}{K u_K}<\tau\Big)\geq 1-\varepsilon
\quad\mbox{and}\quad
\PP(S^K_{\varepsilon}\wedge\tau\wedge e^{KV}<R^K_{\varepsilon/c})\geq
1-\varepsilon.
$$
Then, on $[0,\tau\wedge S_\varepsilon^K\wedge R^K_{\varepsilon/c}]$,
by computing lower and upper bounds on the death rate of a mutant
individual, it can be easily checked that, for $K$ large enough,
almost surely,
$$
\frac{Z^{1,\varepsilon}_t}{K}\leq\langle
\nu_t^{K},\mathbf{1}_{\{y\}}\rangle\leq\frac{Z^{-1,\varepsilon}_t}{K}
$$
where, for $i=1$ or $-1$, $Z^{i,\varepsilon}$ is a continuous-time
branching process such that $Z^{i,\varepsilon}_0=1$ and with birth rate
$(1-i\varepsilon)\lambda(y)$ and death rate
$$
\mu(y)+\sum_{j=1}^d\alpha(y,x_i)\bar{n}_i(\mathbf{x})
+i(d+1)\bar{\alpha}\frac{\varepsilon}{c}.
$$
Next, we use the results of Theorem~4 of~\cite{champagnat-06} on
branching processes in order to control the probability that
$Z^{i,\varepsilon}/K$ exceeds $\varepsilon$ before it reaches 0,
and to upper bound the time at which one of these events happens.
As in~\cite{champagnat-06}, we obtain that there exists $C>0$ such
that, for all $\eta>0$, $\varepsilon>0$ sufficiently small and $K$
large enough, \begin{align}
  \PP\Big(\tau_2<\tau\wedge\frac{\eta}{Ku_K}\wedge
  S^K_\varepsilon\wedge R^K_{\varepsilon/c}\Big) & \geq
  1-\frac{[f(y;\mathbf{x})]_+}{\lambda(y)}-C\varepsilon \label{eq:pf-C6}
  \\
  \PP\Big(S_\varepsilon^K<\tau\wedge\frac{\eta}{Ku_K}\wedge
  S^K_0\wedge R^K_{\varepsilon/c}\Big) & \geq
  \frac{[f(y;\mathbf{x})]_+}{\lambda(y)}-C\varepsilon. \notag
\end{align} On the event $\{S_\varepsilon^K<\tau\wedge S^K_0\wedge
R^K_{\varepsilon/c}\}$, we introduce for $\varepsilon'>0$ the
 stopping times
\begin{align*}
  T^K_\varepsilon&=\inf\{t\geq S_\varepsilon^K:\forall
  i\in\{1,\ldots,d\},\ |\langle
  \nu_t^{K},\mathbf{1}_{\{x_i\}}\rangle-n^*_{i}|<\varepsilon^2\mbox{\
    and\ }|\langle
  \nu_t^{K},\mathbf{1}_{\{y\}}\rangle-n^*_{d+1}|<\varepsilon^2\}, \\
  U^K_{\varepsilon,\varepsilon'}&=\inf\{t\geq
  T^K_\varepsilon: \exists i\in I(\mathbf{n}^*),\ |\langle
  \nu_t^{K},\mathbf{1}_{\{x_i\}}\rangle-n^*_{i}|\geq\varepsilon'\} \\
   V^K_{\varepsilon}&=\inf\{t\geq
  T^K_\varepsilon: \exists i\not\in I(\mathbf{n}^*),\ \langle
  \nu_t^{K},\mathbf{1}_{\{x_i\}}\rangle\geq\varepsilon\}.
\end{align*}
 We next use the Markov property at time $S_\varepsilon^K$ and
apply Proposition~\ref{prop:EK} as in~\cite{champagnat-06} to
obtain that there exists $C'>C$ such that, for $K$ large enough,
$$
\PP\Big(S_\varepsilon^K<T^K_\varepsilon<\tau\wedge\frac{\eta}{K
  u_K}\Big)\geq \frac{[f(y;\mathbf{x})]_+}{\lambda(y)}-C'\varepsilon.
$$
Next, we can use again Proposition~\ref{prop:FW} to prove that
there exists $V'>0$, $C''>C'$ and $c'<1$ such that
$$
\PP\big(S_\varepsilon^K<T^K_\varepsilon<V_\varepsilon^K\wedge\tau\wedge
e^{KV'}< U^K_{\varepsilon,\varepsilon/c'}\big)\geq
\frac{[f(y;\mathbf{x})]_+}{\lambda(y)}-C''\varepsilon.
$$
In a last step, we can as before prove that, for all
$t\in[T_\varepsilon^K,U^K_{\varepsilon,\varepsilon/c'}\wedge
V^K_\varepsilon]$ and for all $i\not\in I(\mathbf{n}^*)$,
$$
\langle\nu^K_t,\mathbf{1}_{\{x_i\}}\rangle\leq\frac{\tilde{Z}^{i,\varepsilon}_t}{K},
$$
where $\tilde{Z}^{i,\varepsilon}$ is a continuous-time branching
process such that
$\tilde{Z}^{i,\varepsilon}_{T_\varepsilon^K}=\lceil\varepsilon^2
K\rceil$ and with birth rate $\lambda(x_i)$ and death rate
$$
\mu(x_i)+\sum_{j\in
  I(\mathbf{n}^*)}\alpha(x_i,x_j)n^*_j
-\mbox{Card}(I(\mathbf{n}^*))\bar{\alpha}\frac{\varepsilon}{c'}.
$$
Since, by Assumption~(B2), $f(x_i;\mathbf{x}^*)<0$, this branching
process is sub-critical if $\varepsilon$ is small enough. Hence,
with arguments similar to the ones in~\cite{champagnat-06}
(especially the results of Theorem~4), we can prove  that there
exist $C'''>0$ such that, for all $\eta>0$, $\varepsilon>0$
sufficiently small and $K$ large enough,
$$
\PP\Big(S^K_\varepsilon<\tau_1<\tau\wedge\frac{\eta}{K u_K}\wedge
U^{K}_{\varepsilon,\varepsilon/c'}\Big)\geq
\frac{[f(y;\mathbf{x})]_+}{\lambda(y)}-C'''\varepsilon.
$$
Combining this with~(\ref{eq:pf-C6}), we obtain
Lemma~\ref{lem:aft-mut}  by letting $\varepsilon$ go to 0.

\bigskip
\noindent Finally,~(\ref{eq:pf-C6-0}) is deduced from these lemmas
exactly as in~\cite{champagnat-06} and similarly, the proof of
Theorem~\ref{thm:PES-fdd} from~(\ref{eq:pf-C6-0}).\hfill$\Box$

\bigskip \noindent {\bf Acknowledgements\phantom{9}} We thank Michel
Bena\"im and Salome Martinez who brought to our attention the
reference \cite{zeeman-93} and Michel Bena\"im for helpful suggestions
concerning Lotka-Volterra systems.

\bibliographystyle{abbrv}
\bibliography{biblio-bio,biblio-math}

\end{document}